\documentclass[a4paper,10pt]{article}
\usepackage[centertags]{amsmath}
\usepackage{amsfonts}
\usepackage{amssymb}
\usepackage{amsthm}
\usepackage{epsfig}
\usepackage{setspace}
\usepackage{ae}
\usepackage{eucal}
\usepackage[usenames]{color}%
\usepackage{graphicx}

\theoremstyle{plain}
\newtheorem{thm}{Theorem}[section]
\newtheorem{lem}[thm]{Lemma}
\newtheorem{cor}[thm]{Corollary}

\theoremstyle{definition}

\theoremstyle{remark}

\newtheorem{rem}[thm]{Remark}

\title{On an auto-controlled global existence scheme of the incompressible  Navier Stokes equation  }
\author{J\"org Kampen }

\begin{document}

\maketitle

\begin{abstract} 
We propose a global scheme for the incompressible Navier Stokes equation, where at each time step a damping potential term is introduced via a time dilatation transformation of the equation itself. For a certain class of operators this technique leads to a global linear upper bound for a solution branch and a refinement of the argument leads even to a globally uniform upper bound of a solution branch. In the case of the incompressible Navier Stokes equation we get a global regular upper bound for the velocity value function and its spatial derivatives.  The regularity is limited only by the regularity of the viscosity coefficient function and by the regularity and a certain decay at infinity of the (derivatives of the) data. Especially, the regularity $H^m\cap C^m$ for $m\geq 2$ is sufficient in order to obtain a regularity in $C^{\lfloor\frac{m}{2}\rfloor, m} $ with respect to time and space. The proposed scheme is an alternative to schemes with external control functions on an analytical level. We mention here, that auto-controlled schemes can be applied to a large but specific class of equations where 'specific' means that the local contraction results in strong norms can be obtained via certain specific functional representations involving convolutions with first order derivatives of the heat kernel and are fundamentally different from schemes based on the energy identity and standard estimates (cf. (\ref{Navlerayusubschemepre}) in the preliminaries and the comments there). For the specific class of operators mentioned the local growth of a global solution branch is controlled, and for the incompressible Navier Stokes equation there is a decay to zero at infinite time. 
These observations are not in contradiction to the fact that there are equations which satisfy a local contraction result in strong norms but blow up. In this case a global existence result via auto-controlled schemes means that there is a global solution branch which exists next to a singular solution (cf. also similar techniques in order to derive singular solutions in \cite{K}). For the special case of the Navier Stokes equation there is a simple argument that there is no other solution branch (as it is well-known). In this extended version we add an appendix which contains a self-contained analytical proof of global regular upper bounds. We also add some remarks on recent research concerning supercritical barriers. We also sketch an argument why a global regular upper bound for equations with general time dependent $L^2$-source term (as proposed recently) seems not to exist.   Furthermore a proof of decay to zero of the velocity components at infinite time is included. 
\end{abstract}

%\begin{document}

%\maketitle
2010 Mathematics Subject Classification. 35Q30, 76D03.
\section{Statement of main theorem and the main idea}
We are concerned with the construction  of a global auto-controlled scheme for regular solutions of Cauchy problem
\begin{equation}\label{Navleray0}
\left\lbrace \begin{array}{ll}
\frac{\partial v_i}{\partial t}-\nu\sum_{j=1}^n \frac{\partial^2 v_i}{\partial x_j^2} 
+\sum_{j=1}^n v_j\frac{\partial v_i}{\partial x_j}=\\
\\ \hspace{1cm}\sum_{j,m=1}^n\int_{{\mathbb R}^n}\left( \frac{\partial}{\partial x_i}K_n(x-y)\right) \sum_{j,m=1}^n\left( \frac{\partial v_m}{\partial x_j}\frac{\partial v_j}{\partial x_m}\right) (t,y)dy,\\
\\
\mathbf{v}(0,.)=\mathbf{h}(.).
\end{array}\right.
\end{equation}
The equation (\ref{Navleray}) is the incompressible Navier Stokes equation in its Leray projection form (cf. \cite{L}),  to be solved for the velocity $\mathbf{v}=\left(v_1,\cdots ,v_n\right)^T$ on a domain $\left[0,\infty\right)\times {\mathbb R}^n$ of dimension $n\geq 3$, and where we assume that
the initial data $\mathbf{h}=(h_1,\cdots, h_n)^T$ satisfy $h_i\in C^m\cap H^m$ for $1\leq i\leq n$ and $m\geq 2$. Here, $C^m\equiv C^m\left({\mathbb R}^n\right) $ is the space of functions which have continuous multivariate derivatives up to order $m$, where next to standard notation multivariate derivatives with multiindex $\alpha=(\alpha_1,\cdots ,\alpha_n)$ are denoted by $D^{\alpha}_x$ or by subscript $_{,\alpha}$ in the following, and $H^s$ is the Sobolev space of order $s\in {\mathbb R}$ ( where in general $s$ will be a positive integer). $K_n$ denotes the fundamental solution of the 
Laplacian for $n\geq 3$. For $n=2$ the kernel $K_n$ has a different form and a little additional work is needed to get a similar result. Furthermore, it is sufficient to assume that the viscosity function satisfies $0<\lambda_0\leq \nu \in C^m_b$ for some constant $\lambda_0 >0$ and $m\geq 2$, where $C^m_b\equiv C^m_b\left({\mathbb R}^n\right)$ is the space of differentiable functions with multivariate bounded derivatives up to order $m$. However, we shall argue with constant $\nu$, because this simplifies some of our arguments- for the generalisation to variable viscosity the adjoint of the fundamental solution can be used. We prefer to consider the simple situation of constant viscosity in order to concentrate on the essential features of the argument, and make some additional remarks at relevant steps how the argument can be generalised occasionally. Furthermore, all these assumptions may all be fine-tuned a little, but we prefer to consider a simple set of assumptions which seems to be adapted to the situation quite naturally. If these assumptions are satisfied for a function space of regularity order $m\geq 2$ we speak of assumptions of regularity order $m\geq 2$. 
Before we state the main theorem let us point out the main idea of auto-controlled global schemes for the Navier Stokes equation. Consider another time scale for the velocity component functions, i.e., for some (in general small) constant $\rho >0$ consider the functions $v^{\rho}_i$ with
\begin{equation}
v^{\rho}_i(\tau,.)=v_i(t,.),~\mbox{where}~t=\rho \tau.
\end{equation}
At each time step of the scheme, i.e., on a time interval $[t_0,t_0+\rho)$ in original time coordinates $t\geq t_0\geq 0$ corresponding to a unit time interval $\tau\in [t_0,t_0+1)$ we consider the comparison function $u^{\rho,t_0}_i,~1\leq i\leq n$ with
\begin{equation}\label{globtimei}
\lambda(1+\tau)u^{\rho,t_0}_i(\sigma,.)=v^{\rho}_i(\tau,.)=v_i(t,.),~\rho\tau= t,~\sigma=\frac{\tau-t_0}{\sqrt{1-(\tau-t_0)^2}}.
\end{equation}
Here $\lambda >0$ is a constant.
If we want to construct upper bounds which are independent of the time horizon $T$, then a refined analysis is necessary, where we use localized schemes with comparison functions  of the form
\begin{equation}\label{locuuu}
\lambda(1+(\tau  -t_0))u^{\rho,t_0}_i(\sigma,.)=v^{\rho}_i(\tau,.)=v_i(t,.),~\rho\tau= t,~\sigma=\frac{\tau-t_0}{\sqrt{1-(\tau-t_0)^2}}.
\end{equation}
The comparison function $u^{\rho,t_0}_i,~1\leq i\leq n$ has local representations in terms of convolutions, where the nonlinear terms in the representation are convolutions with the first spatial derivative of the Gaussian type fundamental solution $G$ of a heat equation $G_{,\tau}-\rho \nu \mu^{\tau,1} \Delta G=0$ (cf. below). More precisely, for the multivariate spatial derivatives of order $|\alpha|\geq 1$ we have 
\begin{equation}\label{Navlerayusubschemeprei}
\begin{array}{ll}
D^{\alpha}_xu^{\rho,t_0}_i(\sigma,x)=\int_{{\mathbb R}^n}D^{\alpha}_xu^{\rho,t_0}_i(0,y)G(\sigma,x;0,y)dy\\
\\
-\int_{0}^{\sigma}\int_{{\mathbb R}^n}\mu(s)  u^{\rho,t_0}_{i,\alpha}(s,y)G(\sigma,x;s,y)dyds\\
\\
-\rho\int_{0}^{\sigma}\int_{{\mathbb R}^n}\mu^{\tau,2}(s)\sum_{j=1}^n \left( u^{\rho,t_0}_j\frac{\partial u^{\rho,t_0}_i}{\partial x_j}\right)_{,\beta} (s,y)G_{,j}(\sigma,x;s,y)dyds\\
\\
+ \rho\int_{0}^{\sigma}\int_{{\mathbb R}^n} \mu^{\tau,2}(s)\sum_{j,r=1}^n\int_{{\mathbb R}^n}\left( \frac{\partial}{\partial x_i}K_n(z-y)\right)\times\\
\\
\times \sum_{j,r=1}^n\left( \frac{\partial u^{\rho,t_0}_r}{\partial x_j}\frac{\partial u^{\rho,t_0}_j}{\partial x_r}\right)_{,\beta} (s,y)G_{,j}(\sigma,x;s,z)dydzds.
\end{array}
\end{equation}
For the value function itself we can still get for $u^{\rho,t_0}_i\in H^m\cap C^m,~m\geq 2$ for all $1\leq i\leq n$ for a ball $B^n_R(x)$ of arbitrary large radius $R>0$ in ${\mathbb R}^n$ around $x$ the representation
\begin{equation}\label{Navlerayusubschemepre2i}
\begin{array}{ll}
u^{\rho,t_0}_i(\sigma,x)=\int_{{\mathbb R}^n}u^{\rho,t_0}_i(0,y)G(\sigma,x;0,y)dy\\
\\
-\int_{0}^{\sigma}\int_{{\mathbb R}^n}\mu(s)  u^{\rho,t_0}_{i}(s,y)G(\sigma,x;s,y)dyds\\
\\
-\sum_j\rho\int_{0}^{\sigma}\int_{{\mathbb R}^n}\mu^{\tau,2}(s)\sum_{j=1}^n F_{ij}(\mathbf{u}^{\rho,t_0})(s,y)G_{,j}(\sigma,x;s,y)dyds\\
\\
+ \rho\int_{0}^{\sigma}\int_{B^n_R(x)} \mu^{\tau,2}(s)\sum_{j,r=1}^n\int_{{\mathbb R}^n}\left( K_n(z-y)\right)\times\\
\\
\times \sum_{j,r=1}^n\left( \frac{\partial u^{\rho,t_0}_r}{\partial x_j}\frac{\partial u^{\rho,t_0}_j}{\partial x_r}\right) (s,y)G_{,i}(\sigma,x;s,z)dydzds\\
\\
+ \rho\int_{0}^{\sigma}\int_{{\mathbb R}^n\setminus B^n_R(x)} \mu^{\tau,2}(s)\sum_{j,r=1}^n\int_{{\mathbb R}^n}\left( K_{n,i}(z-y)\right)\times\\
\\
\times \sum_{j,r=1}^n\left( \frac{\partial u^{\rho,t_0}_r}{\partial x_j}\frac{\partial u^{\rho,t_0}_j}{\partial x_r}\right) (s,y)G(\sigma,x;s,z)dydzds\\
\\
+\mbox{boundary terms},
\end{array}
\end{equation}
where $F_{ij}(\mathbf{u}^{\rho,t_0})$ are quadratic terms in the components of $\mathbf{u}^{\rho,t_0}=\left(u^{\rho,t_0}_1,\cdots ,u^{\rho,t_0}_n\right)^T$ such that $F_{ij,j}(\mathbf{u}^{\rho,t_0})=u^{\rho,t_0}_ju^{\rho,t_0}_{i,j}$. Note that  $F_{jj}(\mathbf{u}^{\rho,t_0})=\frac{1}{2}\left( u_j^{\rho,t_0}\right)^2$.
\begin{rem}
For small time step size $\rho >0$ the main contributions of the integrals are from the arguments  $y,z$ close to $x$. Especially the boundary terms are small and do not contribute essentially to the growth as we shall observe later.
\end{rem}
Note that at $t=t_0$ in (\ref{locuuu}) we have $\sigma=0$ and the initial values 
\begin{equation}\label{locuuu*}
u^{\rho,t_0}_i(0,.)=\frac{1}{\lambda}v^{\rho}_i(t_0,.)
\end{equation}
of the subproblem for the function $u_i,~1\leq i\leq n$ become large at time $\sigma=0$ for small $\lambda >0$ (corresponding to original time $t=t_0$). Now a spatial effect of the operator is used.  Consider the nonlinear increments in (\ref{Navlerayusubschemeprei}) and similarly in (\ref{Navlerayusubschemepre2i}). Note that with respect to the spatial variables we have convolutions with the first order spatial derivative of the Gaussian instead of the Gaussian itself. Let us consider the case of spatial derivatives of order $|\alpha|\neq 0$ - the case of $|\alpha|=0$ is analogous according to the representation of the value function in (\ref{Navlerayusubschemepre2i}) above. For $|\alpha|\neq 0$ in (\ref{Navlerayusubschemeprei}) we have  a nonlinear Burgers terms of the form 
\begin{equation}
-\rho\int_{0}^{\sigma}\int_{{\mathbb R}^n}\mu^{\tau,2}(s)\sum_{j=1}^n \left( u^{\rho,t_0}_j\frac{\partial u^{\rho,t_0}_i}{\partial x_j}\right)_{,\beta} (s,y)G_{,j}(\sigma,x;s,y)dyds
\end{equation}
and a nonlinear Leray projection term of the form
\begin{equation}
\begin{array}{ll}
\rho\int_{0}^{\sigma}\int_{{\mathbb R}^n} \mu^{\tau,2}(s)\sum_{j,r=1}^n\int_{{\mathbb R}^n}\left( \frac{\partial}{\partial x_i}K_n(z-y)\right)\times\\
\\
\times \sum_{j,r=1}^n\left( \frac{\partial u^{\rho,t_0}_r}{\partial x_j}\frac{\partial u^{\rho,t_0}_j}{\partial x_r}\right)_{,\beta} (s,y)G_{,j}(\sigma,x;s,z)dydzds.
\end{array}
\end{equation}
The equation for the Gaussian has purely time-dependent coefficients such that a first order spatial derivative $G_{,i}$ for some $1\leq i\leq n$ has a representation as in 
(\ref{gaussi}) with a factor $\frac{(x_i-y_i)}{\sigma-s}$, where for small $\epsilon_0 >0$ on a time interval $[\epsilon_0,\sigma)$ the spatial effect of the factor $(x_i-y_i)$, i.e., the change of the sign of the factor at $x$ becomes relevant (cf. below for more details). Furthermore, a small parameter $\rho >0$ makes these integrals small compared to the modulus of the potential damping term in (\ref{autodamping}) which appears also as a parameter, and which has no small parameter $\rho >0$. This holds even for small $\lambda$, where the damping becomes stronger. Hence, a spatial effect of the operator cooperates with a purely time delay effect of potential damping and with scaling. In this argument a viscosity effect is used implicitly as we use special representations with the Gaussian and spatial derivatives of the Gaussian. There is also a pure viscosity damping effect, which can be used in order to sharpen the estimates.
In the representations of $D^{\alpha}_xu^{\rho,t_0}_i$ for $0\leq |\alpha|\leq m$ for small time step size there is some local viscosity damping in the first term
\begin{equation}\label{initdataterm}
\int_{{\mathbb R}^n}D^{\alpha}_xu^{\rho,t_0}_i(0,y)G(\sigma,x;0,y)dy.
\end{equation}
More precisely, as coefficients of the equation for $G$ are only time-dependent, we have a convolution with respect to the spatial variables. Using Fourier transforms ${\cal F}$ with respect to the spatial variables convolutions transform to multiplications an the term in (\ref{initdataterm}) becomes
\begin{equation}\label{fourier}
{\cal F}\left( D^{\alpha}_xu^{\rho,t_0}_i\right)(\sigma,\xi)={\cal F}\left( D^{\alpha}_xu^{\rho,t_0}_i\right)(0,\xi)\exp\left(-c(\sigma,t_0)|\xi|^2(\sigma-t_0)\right),
\end{equation}
where $c(.,t_0)>0$ is a positive bounded time-dependent function (which can be computed explicitly, of course - cf. below). For some finite time $\sigma -t_0 =\Delta'>0$ we get a damping factor $\exp\left(-c(\Delta',t_0)|\xi|^2\Delta'\right)$ such that the initial data term can offset some growth. How much can be determined using Plancherel's theorem and trivial lower bounds of the damping term. We shall observe that it is sufficient to offset some growth term of the form $D^{\alpha}_xu^{\rho,t_0}_i\Delta^2$ with $\Delta $ related to $\Delta '$ by $\Delta'=\frac{\Delta}{\sqrt{1-\Delta^2}}$ which allows to get some sharper results on upper bounds.  Indeed, integrating  (fixed $0\leq |\alpha|\leq m$) in (\ref{fourier}), for $a\sim \Delta^s$ small, and (for example) with $s\in \left(1/3,1/2\right)$  we essentially get terms of the form  
\begin{equation}\label{fourier}
{\big |}{\cal F}\left( D^{\alpha}_xu^{\rho,t_0}_i\right)(\Delta',.){\big |}_{L^2}\left( \exp\left(-c(\Delta',t_0)a^2\Delta'\right)-\tilde{c}a^3\Delta'\right)
\end{equation}
as upper bounds for the $L^2$-norm of (\ref{fourier}) for some appropriate constant $\tilde{c}>0$. For small $\Delta' >0$ the damping factor $\exp\left(-c(\Delta',t_0)a\Delta'\right)$ becomes essential and can offset growth of order $(\Delta')^2$. Choosing a stepsize $\rho \sim\Delta'$ the nonlinear growth terms with coefficient $\rho$ and integrals of bounded functions over a time interval of length $\Delta'$ have such a growth.  We shall observe that for regular data from this viscosity damping together with the potential damping we can even can conclude that there is a global decay for large time of the velocity function components to zero (as time goes to infinity).  
Note that the representation for the spatial derivatives of the value function $u^{\rho,t_0}_i$ has spatial first order derivatives of the Gaussian $G_{,i}$, and this holds also for all $x\in {\mathbb R}^n$ in a ball $B^n_R(x)$ of arbitrarily large radius $R>0$ for the value function itself in (\ref{Navlerayusubschemepre2i}). 
As we have time dependence of the coefficients it may be impossible to have an explicit formula for $G_{,i}$, but in a ball of radius $R>0$ around $x$ and for $\mu\in \left(\frac{1}{2},1\right)$ we have a well-known locally integrable estimate
\begin{equation}\label{gcommai}
{\big |}G_{,i}{\big |}\leq \frac{C}{|s-\sigma|^{\mu}|x-y|^{n+1-2\mu}}
\end{equation}
for some constant $C>0$ which depends only on the dimension $n$. Outside a ball of radius $R$ around $x$ we have for some $c>0$
\begin{equation}
{\big |}G_{,i}{\big |}\leq \frac{c}{|s-\sigma|^{\frac{n+1}{2}}}\exp\left(-c\frac{R^2}{\rho(\sigma-s)}\right) \downarrow 0
\end{equation}
as $\rho \downarrow 0$. A further aspect of the representation with the Gaussian is that we have the representation
\begin{equation}\label{gaussi}
\begin{array}{ll}
G_{,i}(t,x;s,y)=\frac{(x_i-y_i)}{\nu\mu^{\tau,1}(s)(t-s)}G(t,x;s,y)\\
\\
+\mbox{lower order singularity terms with respect to time},
\end{array}
\end{equation}
where for $t-s\geq \epsilon_0>0$ the factor $(x_i-y_i)$ changes sign at $x$ such that for small $\epsilon_0$ the contributions of the spatial convolutions in (\ref{Navlerayusubschemeprei}) and in (\ref{Navlerayusubschemepre2i}) become small in a small radius around $x$ integrated up to small time $\sigma >\epsilon_0$, while for small time $s\leq \epsilon_0$  the integrable estimate in (\ref{gcommai}) ensures that we have small contributions of growth compared to the damping term of the equation (considered in (\ref{autodamping}) below) which does not depend on the time step size $\rho>0$.
These facts have the consequence that the integrals involved in the representations in (\ref{Navlerayusubschemepre2i}) and in (\ref{Navlerayusubschemeprei}) above of the form
\begin{equation}\label{rhof}
\rho\int\int f(y)G_{,i}(x-y,\tau-s)dyds
\end{equation}
are small for small time step size $\rho >0$ compared to the auto-control damping term 
\begin{equation}\label{autodamping}
-\int_{0}^{\sigma}\int_{{\mathbb R}^n}\mu(s)  u^{\rho,t_0,p}_{i}(s,y)G(\sigma,x;s,y)dyds.
\end{equation}
which have no such small time size factor $\rho>0$. Here we use the special structure of the operator which allows us to represent the nonlinear growth terms in each time step in the form (\ref{rhof}), where the Gaussian $G$ is a fundamental solution of a heat equation with purely time-dependent coefficients.
Let us emphasize this special spatial effects of the operator which lead to a relative small growth of the Burgers term increment and the Leray projection term increment at a small step size $\rho>0$ relative to the potential damping for larger values, where even for the global time factor transformation in (\ref{globtimei}) this time step size $\rho 0$ can be chosen of order $\frac{1}{T^s}$ for some $s\in (0,1)$ such that global existence results transfer from $\tau$-coordinates time scale to original time $t$-coordinates of original velocity function component. We write the first order spatial derivative of the Gaussian $G$ with respect to the variable $x_i$ in the form
\begin{equation}
G_{,i}(\sigma,x;s,y)=\frac{(x_j-y_j)}{\sigma -s}G^{*}(\sigma,x;s,y)
\end{equation}
Integrals of the form (\ref{rhof}) can be written as sums of integrals over a ball $B_{\rho^r}(x)$ of radius $\rho^r$ around $x$ and an intergal over the complement ${\mathbb R}^n\setminus B_{\rho^r}(x)$. In our schemes the Gaussian is always a fundamental solution of an heat equation with purely timedependent coefficients. This implies that we have a antisymmetry of the first order spatial derivatives of the Gaussian. We may write
\begin{equation}\label{ulmrep*333int}
 \begin{array}{ll}
{\Big |}\rho\int_{m-1}^{\sigma}\int_{B_{\rho^r}(x)}f_i(s,y)\frac{(x_j-y_j)}{\sigma -s}G^{*}(\sigma,x;s,y)dyds{\Big |}=\\
\\
{\Big |}\rho\int_{m-1}^{\sigma}\int_{B^{x_j\geq y_j}_{\rho^r}(x)}{\big (}f_i(s,y)-f_i(s,y^{-j}){\big )}\frac{(x_j-y_j)}{\sigma -s}G^{*}(\sigma,x;s,y)dyds{\Big |},
\end{array}
\end{equation} 
where in the last term we integrate over a half-sphere 
\begin{equation}
B^{x_j\geq y_j}_{\rho^r}(x)=\left\lbrace y\in B_{\rho^r}(x)|x_j\geq y_j\right\rbrace 
\end{equation}
and 
$y^{-j}=\left(y^{-j}_1,\cdots ,y^{-j}_n\right)$ is the vector with $y^{-j}=(-1)^{\delta_{ij}}$ if $y=(y_1,\cdots ,y_n)$ and $\delta_{ij}$ is the Kronecker $\delta$. For small $\rho \succsim \frac{1}{T^s}$ for some $s\in \left(r,1 \right)$ this term becomes small compared to the potential damping term which has no factor $\rho$. For the integral over ${\mathbb R}^n\setminus B_{\rho^r}(x)$ we have a strong damping via $\rho$ as the factor $\frac{1}{\rho^{1-2r}(\sigma -s)}$ appears in the exponent.

We prove
\begin{thm}\label{mainthm}
If the assumptions above of regularity order $m\geq 2$ hold, then the Cauchy problem in (\ref{Navleray}) has a classical solution in the function space $C^{\lfloor\frac{m}{2}\rfloor, m} $, where $\lfloor .\rfloor$ denotes the Gaussian floor.
\end{thm}
\begin{rem}\label{initialrem}
As we have 
\begin{equation}
h_i\in H^m\cap C^m
\end{equation}
for $1\leq i\leq n$ we know that $D^{\alpha}_xh_i\in C\cap L^2$ for $0\leq |\alpha|\leq m$. 
For each order of derivative $|\alpha|$ we have $\left( D^{\alpha}_xg_i\right) (.)=(D^{\alpha}h_i)(\tan(.))\in C_b\left(\left(-\frac{\pi}{2},\frac{\pi}{2} \right) \right) \cap L^2\left(\left(-\frac{\pi}{2},\frac{\pi}{2} \right)\right) $ 
where $\tan(y)=\left(\tan (y_1),\cdots ,\tan (y_n) \right)^T$ where $C_b$ denotes the 
function space of bounded continuous functions. As $D^{\alpha}h_i(.)=D^{\alpha}_ig_i(\arctan(.))$, where the latter function is a concatenation of bounded continuous functions, we have indeed $D^{\alpha}_xh\in C_b\cap L^2$, and, hence, do not only have $C^{\lfloor\frac{m}{2}\rfloor, m}=C^{\lfloor\frac{m}{2}\rfloor, m}\left(\left(0,\infty\right),{\mathbb R}^n\right) $, but also $C^{\lfloor\frac{m}{2}\rfloor, m}\left(\left[0,\infty\right),{\mathbb R}^n\right)$. In general we gain regularity in a scheme, such that after one time step we do not have to care about these niceties.
\end{rem}

\begin{rem}
For the incompressible Navier Stokes equation a global regular solution branch $v\in C^0\left([0,T],H^m\cap C^m\right)$ for arbitrary $T>0$ leads - via Cornwall's inequality-
to uniqueness, i.e., if $\tilde{v}_i,~1\leq i\leq n$ is another solution of the incompressible Navier Stokes equation, then we have for $n$
\begin{equation}
{\big |}\tilde{v}(t)-v(t){\big |}^2_{L^2}\leq {\big |}\tilde{v}(0)-v(0){\big |}^2_{L^2}
\exp\left(C\int_0^t\left( {\big |}{\big |}v(s){\big |}^p_{L^4}+{\big |}{\big |}v(s){\big |}^2_{L^4}\right)ds  \right) 
\end{equation}
where $C>0$ is a constant which depends on the dimension $n$ and the viscosity only and for some $p\geq 4$ which depends on the dimension $p=8$ in dimension $n=3$ is sufficient. This means that we find no other solution branch with the same $H^q\cap C^q$ data for $q\geq 2$.    
\end{rem}

We also prove
\begin{cor}\label{maincor}
In the situation of theorem \ref{mainthm} the solution $v_i,~1\leq i\leq n$ satisfies
\begin{equation}
\forall x\in {\mathbb R}^n~\lim_{t\uparrow \infty}{\big |}v_i(t,x){\big |}=0.
\end{equation}
\end{cor}

The huge amount of research concerning the regularity problem makes it difficult to give an overview and a fair discussion concerning all the contributions made. It seems to me that in the time after the paper of Hopf the contribution of the Cafferelli-Kohn-Nirenberg paper and the related papers by Lin are the most significant, because it seems that their arguments can be supplemented, such that their results are indeed very close to a full global regular existence result. In the next section we make some preliminary remarks, which may help to understand what spatial properties of the operators are needed in order to obtain global solution branches via auto-controlled scheme. In this context we discuss the Katz-Pavlovic model. In section 3 we make additional comments on very recent research concerning the construction of singular solutions in averaged Navier Stokes equations and a claimed global existence result for a more general system.  
\section{Some preliminary remarks}

The auto-controlled schemes considered here introduce a damping term for a comparative function where the damping term is caused by a transformation which depends on time exclusively. The following questions then immediately arise
\begin{itemize}
 \item[i)] why does the method not apply to a simple ODE such as $\stackrel{\cdot}{x}=x^2,~x(0)=x_0\neq 0$? Here, it is interesting to compare the spatial properties needed with the spatial properties of the Katz-Pavlovic model.
 \item[ii)] related to item i): what spatial properties of the operator are used and how? 
 \item[iii)] the method is fairly general, and seems to apply to many PDEs which are known to have singular solutions. How is it possible to claim that they have global solutions?
\end{itemize}

\begin{rem}
In any case it has to be shown that the step size $\rho >0$ does not depend on the time horizon essentially in the sense that global existence results for the function $v^{\rho}_i,~1\leq i\leq n$ with $v^{\rho}_i(\tau,.)=v_i(t,.)$ for $t=\rho \tau$ can be transferred to analogous statements for $v_i,~1\leq i\leq n$.
\end{rem}

Ad i) we remark that the method uses spatial effects of the operator. If the (modulus of) values of a solution function of a differential equation with quadratic terms of the value function itself can become large then the increment of this square function in the operator exceeds any linear control such that the growth cannot be offset (especially not by a linear autocontrol).

The solution of the ODE mentioned in i) has the increment
\begin{equation}\label{ODE}
x(t)-x(t_0)=\int_{t_0}^t x^2(s)ds
\end{equation}
for $t_0\geq 0$. This is a quadratic incremental growth and a damping term caused by time transformation cannot offset this growth. Consider an auto-controlled scheme for this equation. Defining for $t_0\geq 0$
\begin{equation}
(1+t)y(s)=x(t),~s=\frac{t-t_0}{\sqrt{1-(t-t_0)^2}},~t\in [t_0,t_0+1)
\end{equation}
we have
\begin{equation}
y(s)+(1+t)\stackrel{\cdot}{y}(s)\frac{ds}{dt}=(1+t)^2y(s)^2,
\end{equation}
which leads to
\begin{equation}\label{ydot}
\stackrel{\cdot}{y}(s)=\sqrt{1-(t-t_0)^2}^3(1+t)y(s)^2-\frac{\sqrt{1-(t-t_0)^2}^3}{(1+t)}y(s).
\end{equation}
In this case we know that
\begin{equation}
y(s)=\frac{1}{(1+t(s))(1-t(s))},
\end{equation}
where $t=t(s)$ denotes the inverse of $s= s(t)$ such that
\begin{equation}
\stackrel{\cdot}{y}(s)=\frac{\sqrt{1-(t-t_0)^2}^3}{1+t(s)}\left( \frac{1}{(1-t(s))^2}-\frac{1}{(1+t(s))(1-t(s))}\right),
\end{equation}
which is even for $t_0=0$ a growth of order $\sqrt{1-t^2}^{-1}$ as $t=t(s)\uparrow 1$ or $s\uparrow \infty$. The equation in (\ref{ydot}) shows that in auto-controlled schemes we need essentially a linear upper bound for the local growth of the original operator in order to obtain a global upper bound via an auto-controlled scheme.  
We can also observe the effect of scaling here. For the equation
\begin{equation}
\stackrel{\cdot}{x}_{\lambda}=\lambda x^2_{\lambda},~x_{\lambda}(0)=x_0,
\end{equation}
where (for $x_0=1$ for simplicity) we have the solution $x_{\lambda}(t)=\frac{x_0}{1-\lambda x_0t}$, which pushes the singularity to $t=\frac{1}{\lambda}$.
Defining for $t_0\geq 0$
\begin{equation}
\lambda (1+t)y_{\lambda}(s)=x_{\lambda}(t),~s=\frac{t-t_0}{\sqrt{1-(t-t_0)^2}},~t\in [t_0,t_0+1)
\end{equation}
we have
\begin{equation}
\lambda y_{\lambda}(s)+\lambda (1+t)\stackrel{\cdot}{y}_{\lambda}(s)\frac{ds}{dt}=\lambda^2(1+t)^2y(s)^2_{\lambda},
\end{equation}
which leads to
\begin{equation}\label{ydot}
\stackrel{\cdot}{y}_{\lambda}(s)=\lambda\sqrt{1-(t-t_0)^2}^3(1+t)y(s)^2_{\lambda}-\frac{\sqrt{1-(t-t_0)^2}^3}{(1+t)}y_{\lambda}(s).
\end{equation}
In this case we know that
\begin{equation}
y_{\lambda}(s)=\frac{1}{\lambda(1+t(s))(1-\lambda t(s))},
\end{equation}
such that
\begin{equation}
\stackrel{\cdot}{y}_{\lambda}(s)=\frac{\sqrt{1-(t-t_0)^2}^3}{1+t(s)}\left( \frac{1}{(1-\lambda t(s))^2}-\frac{1}{\lambda (1+t(s))(1-\lambda t(s))}\right),
\end{equation}
such that the growth of the transformed solution cannot be offset as $t_0$ becomes close to $\frac{1}{\lambda}$. Again this shows that we need some spatial effects of the operator in order to show that there are global solution branches. Note that auto-controlled schemes are designed in order to prove the existence of a global solution in regular function spaces. Uniqueness is a different matter. If an equation has multiple solutions, then an auto-controlled scheme may be also be used in order to obtain singular solutions. A spatial effect of operators which leads to global auto-controlled schemes is the existence of a first order spatial derivative in the nonlinear terms. For a class of such examples representations of local solutions involve the first order derivative of the Gaussian, and this leads to a comparatively small growth of the nonlinear terms. Even if the value function becomes large the growth of the nonlinear terms can be relatively small compared to the linear terms related to viscosity damping and potential damping. This is the main difference to operators which involve nonlinear terms with quadratic or hogher order powers of the value function itself. 

Such an example of an operator which involves first order spatial derivatives in the nonlinear terms is the inviscid n-dimensional Burgers problem $\frac{\partial u}{\partial t}-u\nabla u=0,~u(0)=u_0$ with $u=(u_1,\cdots ,u_n)^T$ on the torus. In terms of modes of the analytical basis $\left\lbrace \exp( k x)\right\rbrace_{k\in {\mathbb Z}^n}$, where ${\mathbb Z}^n$ is the set of $n$-tuples of integers and $n\geq 1$,
this corresponds to an infinite ODE for time dependent modes $u_{il}=u_{il}(t),~l\in {\mathbb Z}^n,~1\leq i\leq n$ of the form
\begin{equation}\label{burgmod}
\frac{d}{dt}u_{il}=\sum_{j=1}^n\sum_{k\in {\mathbb Z}^n}u_{j(l-k)}k_ju_{ik}=\sum_{j=1}^n\sum_{k\in {\mathbb Z}^n\setminus \{0\}}u_{j(l-k)}k_ju_{ik},
\end{equation}
where $k_j$ is the $j$th component of the $n$-tuple $k=(k_1,\cdots ,k_n)$.
If the data $u_0$ are regular, i.e., if the modes $u_{ik}(0),~k\in {\mathbb Z}^n,~1\leq i\leq n$ of $u_0$ have polynomial decay
\begin{equation}
{\big |}u_{ik}(0){\big |}\leq \frac{c}{1+|k|^m}~\mbox{for}~m\geq n+2,
\end{equation}
for some constant $c>0$, then the fact that
\begin{equation}\label{boundint}
\sum_{k\in {\mathbb Z}^n}{\big |}u_{j(l-k)}(0)k_ju_{ik}(0){\big |}\leq \frac{c_0}{1+|l|^{m+m-n-1}}\leq \frac{c_0}{1+|l|^{n+2}}
\end{equation}
for some $c_0>0$ leads to local contraction results and to the observation that the increment
\begin{equation}
u_{il}(t)-u_{il}(t_0)=\int_{t_0}^t\sum_{k\in {\mathbb Z}^n}u_{l-k}(s)ku_k(s)ds
\end{equation}
has a growth which becomes small on a small time scale $\rho$ in fundamental difference to (\ref{ODE}).  The spatial derivative in the nonlinear Burgers term
\begin{equation}
\sum_{j=1}^n\sum_{k\in {\mathbb Z}^n}k_ju_{j(l-k)}u_{ik}
\end{equation}
appears as the factor $k_j$ and ensures that in an iteration scheme contribution the zero order modes $u_{i0}$ are by modes of order different from zero only. Moreover, the relation in (\ref{boundint}) holds with a stronger upper bound of increasing order of polynomial decay. Let us have a closer look at this. 
First of all the effects described can be combined with a scaling effect in order to obtain global solution branches. Note that for small $\lambda >0$ the function $u^{\lambda}_i=\lambda u_i$ satisfies the equation
\begin{equation}\label{burginvisc}
\frac{d}{dt}u^{\lambda}_{il}=\lambda\sum_{j=1}^n\sum_{k\in {\mathbb Z}^n\setminus \{0\}}u^{\lambda}_{j(l-k)}k_ju^{\lambda}_{ik}.
\end{equation}
Given some natural number $M_0$ we may consider all modes $l$ with $|l|\leq M_0$. Then  for small $\lambda >0$ the sum
\begin{equation} 
\lambda\sum_{j=1}^n\sum_{|k|\leq M_0}u^{\lambda}_{j(l-k)}k_ju^{\lambda}_{ik}
\end{equation}
is small while an inheritance of regularity of the value function means that we have small higher order modes, i.e., the moduli
\begin{equation}
{\big |}u^{\lambda}_{ik}(_0){\big |}\lesssim\frac{1}{|k|^{2m}}
\end{equation}
are small for $|k|\geq M_0$. An auto-controlled scheme for the inviscid Burgers equation with comparison function
\begin{equation}
\lambda (1+t)w^{\lambda}_i(s,.)=u_i(t,.),~s=\frac{t-t_0}{\sqrt{1-(t-t_0)^2}}
\end{equation}
leads to 
\begin{equation}\label{burgmod}
\frac{d}{dt}w^{\lambda}_{il}=
\lambda \sqrt{1-(t-t_0)^2}^3(1+t)
\sum_{j=1}^n\sum_{k\in {\mathbb Z}^n\setminus \{0\}}w^{\lambda}_{j(l-k)}k_jw^{\lambda}_{ik}-\frac{\sqrt{1-(t-t_0)^2}^3}{1+t}w^{\lambda}_{il},
\end{equation}
Given a finite time horizon $T>0$ for original time $0\leq t\leq T$ a small $\lambda >0$, which depend only on the time horizon $T>0$ and the initial data $u_0i(0)\in h^s\left({\mathbb Z}^n\right), 1\leq i\leq n, s\geq n+2$ for a strong dual Sobolev space $h^s$ such that the damping term 
\begin{equation}
-\frac{\sqrt{1-(t-t_0)^2}^3}{1+t}w^{\lambda}_{il}
\end{equation}
can offset the growth of the nonlinear term, because we have a relation of the form \ref{boundint} and the decay is inherited by a time local scheme which is easily shown by contraction. 
For this reason the data have to be in regular spaces in order to have bounds as in (\ref{boundint}) and then it may be possible to get a global solution branch for the equation. This effect is even stronger if there is a viscous term. The equation corresponding to (\ref{burginvisc}) is
\begin{equation}\label{burgvisc}
\frac{d}{dt}v^{\lambda}_{il}=\nu \Delta v^{\lambda}_{il}+\lambda\sum_{j=1}^n\sum_{k\in {\mathbb Z}^n\setminus \{0\}}v^{\lambda}_{j(l-k)}k_jv^{\lambda}_{ik}.
\end{equation}
Here we have an additional damping of the higher order modes via the viscous damping as can be observed from Trotter product formulas (cf. \cite{KTr1, KTr2}).
Compare this with the Katz Pavlovic model investigated in \cite{KP} which is essentially described by the equation
\begin{equation}
 \frac{d}{dt}K_m-\mu^{2m\alpha}K_m+\mu^{m-1}K^2_{m-1}-\mu^mK_mK_{m-1},
\end{equation}
where $m\in {\mathbb Z}$ and $\mu$ is some positive constant. Here the existence of singular solutions has been shown for $\lambda >1$. Here, we observe that we do not have the property that we can have a small growth for the modes less than a certain $M_0$ (especially the growth at  $m=0$ the non-linear term is not zero in general). Here an auto-control scheme cannot offset the growth of the nonlinear term in general.

Ad iii) It is clear that singular solutions of the form $\frac{w(t,x)}{1-t}$ with $w(1,x)\neq 0$ can be constructed for the inviscid Burgers equation. A global solution construction via an auto-controlled scheme (if possible) then means that there is a global solution branch of the equation and that solutions are not unique in general. For some equations (such as for the incompressible Navier Stokes equation) there are uniqueness arguments in addition which lead to the conclusion that a global solution branch is the unique solution of the equation in an appropriate (strong) function space. For example uniqueness of solutions for a class of heatflow maps was proved by Lin in weak spaces while the existence of singular solutions is known for heat flow maps on manifolds of dimension four. In a strong space it may be possible to construct a global solution branch which coexists with the constructed singular solution.   

\section{Some more comments on recent research}

In order to apply the auto-controlled to extended models with additional source terms in form of external force functions $f_i,~1\leq i\leq n$ directly, the latter should be located in $C^1\cap H^1$ (in order to obtain a global scheme). A source (external force) term which is just in $L^2$ seems critical - at least in the inviscid limit we argue that there is a blow up. This argument also indicates that a weak external source term may be critical. In this context it is interesting  that there has been an announced regularity result for the Navier Stokes equation on the $n$-dimensional $3$-dimensional torus recently, which states that the
\begin{equation}\label{Navlerayolt} \begin{array}{ll}
\frac{\partial v_i}{\partial t}-\nu\sum_{j=1}^n \frac{\partial^2 v_i}{\partial x_j^2} 
+\sum_{j=1}^n v_j\frac{\partial v_i}{\partial x_j}=-\nabla p+f_i.
\end{array}
\end{equation}
with $f_i\in L^2$, constant $\nu>0$ and initial data $h_i\equiv 0$ has a global regular solution. The idea of the related work is to morph the velocity field by an additional vector field where the sum preserves some norm. It was criticised that the
homotop family of operators $A^{\theta}$ considered do not have a stability property in weak spaces as claimed. We mention here that singular solutions are quite generic in the inviscid limit. We shall consider in more detail elsewhere whether there are also singular solutions for some positive viscosity $\nu >0$. In this case the proposed global existence theorem could not be valid in a function space of the velocity functions in which solutions are known to be unique. We sketch an argument for the existence of singular solutions in the inviscid situation which we may consider in more detail elsewhere.  The following argument is in the framework of calculus with explicit infinitesimals as described in \cite{KTr1}. There we mentioned that
for that the modes $v_{i\alpha},~1\leq i\leq n, \alpha\in {\mathbb Z}^n$ of the velocity component functions for an incompressible Navier Stokes equation (without external force terms) satisfy the infinitesimal Euler scheme 
\begin{equation}\label{navode200second}
\begin{array}{ll}
v_{i\alpha}((m+1)\delta t)=v_{i\alpha}(m\delta t)+\sum_{j=1}^n\nu \left( -\frac{4\pi^2 \alpha_j^2}{l^2}\right)v_{i\alpha}(m\delta t)\delta t\\
\\
+\sum_{j=1}^n\sum_{\gamma\in {\mathbb Z}^n}e_{ij\alpha\gamma}(m\delta t)v_{j\gamma}(m\delta t)\delta t.
\end{array} 
\end{equation}
Abbreviating 
\begin{equation}
\begin{array}{ll}
e_{ij\alpha\gamma}(m dt)=-\frac{2\pi i (\alpha_j-\gamma_j)}{l}v_{i(\alpha-\gamma)}(m\delta t)\\
\\
+ 2\pi i\alpha_i1_{\left\lbrace \alpha\neq 0\right\rbrace}
4\pi^2 \frac{\sum_{k=1}^n\gamma_j(\alpha_k-\gamma_k)v_{k(\alpha-\gamma)}(m\delta t)}{\sum_{i=1}^n4\pi^2\alpha_i^2}.
\end{array}
\end{equation}
this leads to the Trotter product formula (for all $t_e=N_0\delta t$ for some finite or infinite number $N_0$)
\begin{equation}\label{aa}
\mathbf{v}^{F}(t_e)\doteq \Pi_{m=0}^{N_0-1}\left( \delta_{ij\alpha\beta}\exp\left(-\nu 4\pi^2 \sum_{i=1}^n\alpha_i^2 \delta t \right)\right)  \left( \exp\left( \left( \left( e_{ij\alpha\beta}\right)_{ij\alpha\beta}(m\delta t)\right)\delta t \right) \right) \mathbf{h}^F, 
\end{equation}
and where $\doteq$ means that the identity holds up to an infinitesimal error. Here, recall that the entries in $(\delta_{ij\alpha\beta})$ are Kronecker-$\delta$s which describe the unit $n{\mathbb Z}^n\times n{\mathbb Z}^n$-matrix. At each time step $m$ we have 
\begin{equation}
\left( \delta_{ij\alpha\beta}\exp\left(-\nu 4\pi^2 \sum_{i=1}^n\alpha_i^2 \delta t \right)\right)  \left( \exp\left( \left( \left( e_{ij\alpha\beta}(m\delta t)\right)_{ij\alpha\beta}\right)\delta t\right)  \right)\mathbf{v}^F(m\delta t).
\end{equation}
which is a correct local formula up to an error of order $\delta t^2$.
We argued in \cite{KTr1} that upper bounds can be obtained if the initial data are located in rather strong Sobolev spaces. For weaker norms the deacy of the modes cannot be controlled in general. First let is recall the idea of a construction of global regular upper bounds for data in strong norms based on the Trotter product formula.
We may consider the scheme on the time scale $\delta t=\rho\delta s$ for finite but small $\rho>0$, for $v^{\rho}_i(s,.)=v_i(t,.)$ with $t=\rho s$))
\begin{equation}\label{navode200second*}
\begin{array}{ll}
v^{\rho}_{i\alpha}((m+1)\delta s)=v^{\rho}_{i\alpha}(m\delta s)+\rho\sum_{j=1}^n\nu \left( -\frac{4\pi^2 \alpha_j^2}{l^2}\right)v^{\rho}_{i\alpha}(m\delta s)\delta s\\
\\
+\rho\sum_{j=1}^n\sum_{\gamma\in {\mathbb Z}^n}e^{\rho}_{ij\alpha\gamma}(m\delta t)v^{\rho}_{j\gamma}(m\delta s)\delta s.
\end{array} 
\end{equation}
Abbreviating 
\begin{equation}
\begin{array}{ll}
e^{\rho}_{ij\alpha\gamma}(m \delta s)=-\rho\frac{2\pi i (\alpha_j-\gamma_j)}{l}v_{i(\alpha-\gamma)}(m\delta s)\\
\\
+ \rho 2\pi i\alpha_i1_{\left\lbrace \alpha\neq 0\right\rbrace}
4\pi^2 \frac{\sum_{k=1}^n\gamma_j(\alpha_k-\gamma_k)v_{k(\alpha-\gamma)}(m\delta s)}{\sum_{i=1}^n4\pi^2\alpha_i^2}.
\end{array}
\end{equation}
this leads to the the Trotter product formula 
\begin{equation}\label{aa}
\mathbf{v}^{\rho,F}(t_e)\doteq \Pi_{m=0}^{N_0-1}\left( \delta_{ij\alpha\beta}\exp\left(-\rho\nu 4\pi^2 \sum_{i=1}^n\alpha_i^2 \delta s \right)\right)  \left( \exp\left( \left( \left( e^{\rho}_{ij\alpha\beta}\right)_{ij\alpha\beta}(m\delta t)\right)\delta s \right) \right) \mathbf{h}^F, 
\end{equation}
where at each time step $m$ we have 
\begin{equation}\label{eulerstep}
\left( \delta_{ij\alpha\beta}\exp\left(-\rho\nu 4\pi^2 \sum_{i=1}^n\alpha_i^2 \delta s \right)\right)  \left( \exp\left( \left( \left( e^{\rho}_{ij\alpha\beta}(m\delta s)\right)_{ij\alpha\beta}\right)\delta t\right)  \right)\mathbf{v}^{\rho,F}(m\delta s).
\end{equation}
Now the idea of a regular upper bound argument is simply this. Assume inductively that the data $\mathbf{v}^{\rho,F}(m\delta s)$ at step $m$ satisfy
\begin{equation}\label{uppernote}
{\big |}v^{\rho}_{i\alpha}(m\delta s){\big |}\leq \frac{C}{1+|\alpha|^{n+2}}
\end{equation}
This data go into the Euler term of the euler step in (\ref{eulerstep}), and, according to the rule
\begin{equation}\label{rule}
\sum_{\gamma\in {\mathbb Z}^n}\frac{C|\alpha-\gamma|}{1+|\alpha-\gamma|^{n+2}} \frac{C}{1+|\gamma|^{m+2}}\leq  \frac{cC}{1+|\alpha|^{n+2+n+2-(n+1)}}=\frac{cC}{1+|\alpha|^{n+3}}
\end{equation}
for the Euler terms, i.e. the burgers term and the Leray projection term, we get decreasing order of modes in (\ref{rule}) such that for small $\rho$ the growth behavior after finitely many step $m+p$ can be estimated up to first order in $\delta s$ by 
\begin{equation}\label{estrho}
{\big |}v^{\rho}_{i\alpha}((m+p)\delta s){\big |}\leq {\big |}v^{\rho}_{i\alpha}(m\delta s){\big |}+\frac{\rho \tilde{c}C}{1+|\alpha|^{n+2}}\delta s-p\rho \nu 4\pi^2 \sum_{i=1}^n\alpha_i^2 {\big |}v^{\rho}_{i\alpha}(m\delta s){\big |}\delta s,
\end{equation}
where we use that for small $\rho$ the first term on the right side of (\ref{estrho}) can be estimated by a (derivative of) a geometric series (small $\rho$ ensures that we have a converging series for all modes $|\alpha|\neq 0$). Note the last term in (\ref{estrho}): we have replaced $\rho \nu 4\pi^2 \sum_{i=1}^n\alpha_i^2\sum_{l=1}^p{\big |}v^{\rho}_{i\alpha}((m+l)\delta s){\big |}\delta s$ by the first-order-equal term $p\rho \nu 4\pi^2 \sum_{i=1}^n\alpha_i^2{\big |}v^{\rho}_{i\alpha}(m\delta s){\big |}$ as $p$ is a finite number. This shows that the viscosity damping is microscopically (i.e. for example in a monad) similar to a damping in an Ornstein-Uhlenbeck process: if the value of a mode becomes too large then it is forced back by the viscosity damping. Note furtermore that the formula in (\ref{estrho}) holds for all finite numbers $p$ while $p$ does not appear in the first term on the right side of (\ref{estrho}) which is related to the nonlinear terms. The reason is that we could sum them up for small $\rho$ in a (first order derivative) of a geometric series. The reason is that elliptic intergals related to the rule (\ref{rule}) lead to strong polynomial decay behavior for data in strong morms. Now if the data become large enough, let's say
\begin{equation}
{\big |}v^{\rho}_{i\alpha}(m\delta s){\big |}\in \left[ \frac{0.5 C}{1+|\alpha|^{n+2}},\frac{C}{1+|\alpha|^{n+2}}\right] 
\end{equation}
Then the viscosity damping will offset the growth of the nonlinear terms for this modes after finitely may time steps (for a small $\rho>0$ related to $C>0$). On the other hand, if
\begin{equation}
{\big |}v^{\rho}_{i\alpha}(m\delta s){\big |}\in \left[ 0,\frac{0.5 C}{1+|\alpha|^{n+2}}\right],
\end{equation}
then an upper bound of the form (\ref{uppernote}) will be preserved for some finite  time, i.e. in a time interval $[t_0,t_0+\Delta]$ with finite positive real $\Delta >0$ if we started at time $t_0$. In the former case such a positive real $\Delta >0$ also exists by transfinite induction or by an overflow principle. This argument cannot be applied if we have weaker data or a time-dependent force term which is just in $L^2$. Then solution may blow up. Let us consider this in more detail.

Note that $g\in H^s\left({\mathbb Z}^n\right)$ for $s\geq 0$ can be characterised by dual Sobolev space $h^s\left({\mathbb Z}^n\right)$ which measure the growth properties of the modes, i.e.,  
\begin{equation}
(g_{\alpha})_{\alpha\in {\mathbb Z}^n}\in h^s\left({\mathbb Z}^n\right):\Longleftrightarrow \sum_{\alpha\in {\mathbb Z}^n}~\left(1+|\alpha|^{2s} \right) {\big |}g_{\alpha}^2{\big |}\leq c.
\end{equation}
For $s=0$ we have $h^0\left({\mathbb Z}^n\right)=l^2\left({\mathbb Z}^n\right)$,the dual of $L^2$. Next we show that $f_i\in L^2$ is an assumption which is to weak in order to obtain a global regular existence theorem.  
For some time dependent bounded continuous functions $c_i\neq 0,~1\leq i\leq n$ consider the source data $h_{i},~1\leq i\leq n$ with modes 
\begin{equation}\label{fialpha1}
f_{i\alpha}(t)=\frac{c_i(t)}{1+|\alpha|^{\frac{3+\epsilon}{2}}}~\mbox{ if }~ \alpha_i\geq 0 \mbox{ for all}~1\leq i\leq n,
\end{equation}
and for small $\epsilon >0$. 
For $f_i(t,.)\in L^2$ for all $t\in [0,T]$ for some time horizon such that the main theorem cannot be true in this general form. As this work is on the Navier Stokes equation on the whole domain  we shall brief here. As additional material cf. our considerations about Trotter product formulas related to Navier Stokes equations in \cite{KTr1,KTr2}. For a Navier Stokes equation with zero initial data and a source term $f_i,¸1\leq i\leq n$ we write  the velocity component $v^f_i=v^f_i(t,x)$ for fixed $t\geq 0$ in the analytic basis $\left\lbrace \exp\left( \frac{2\pi i\alpha x}{l}\right),~\alpha \in {\mathbb Z}^n\right\rbrace $
\begin{equation}
v^f_i(t,x):=\sum_{\alpha\in {\mathbb Z}^n}v^f_{i\alpha}(t)\exp{\left( \frac{2\pi i\alpha x}{l}\right) },
\end{equation}
the problem in (\ref{Navleray}) can be rephrased in terms of velocity modes $v^f_{i\alpha},~\alpha\in {\mathbb Z}^n,~1\leq i\leq n$, where
\begin{equation}\label{navode200first}
\begin{array}{ll}
\frac{d v^f_{i\alpha}}{dt}=\sum_{j=1}^n\nu \left( -\frac{4\pi^2 \alpha_j^2}{l^2}\right)v^f_{i\alpha}
-\sum_{j=1}^n\sum_{\gamma \in {\mathbb Z}^n}\frac{2\pi i \gamma_j}{l}v^f_{j(\alpha-\gamma)}v_{i\gamma}\\
\\
+2\pi i\alpha_i1_{\left\lbrace \alpha\neq 0\right\rbrace}\frac{\sum_{j,k=1}^n\sum_{\gamma\in {\mathbb Z}^n}4\pi^2 \gamma_j(\alpha_k-\gamma_k)v^f_{j\gamma}v^f_{k(\alpha-\gamma)}}{\sum_{i=1}^n4\pi^2\alpha_i^2}+f_{i\alpha},
\end{array} 
\end{equation}
for all $1\leq i\leq n$ and where for all $\alpha\in {\mathbb Z}^n$ we have $v^f_{i\alpha}(0)=0$. We consider an infinite scheme with explicit infinitesimals (a hyperfinite scheme or a scheme in the calculus of Connes for example) of the form
\begin{equation}\label{navode200first}
\begin{array}{ll}
v^f_{i\alpha}((m+1)\delta t)=v^f_{i\alpha}(m\delta t)+\sum_{j=1}^n\nu \left( -\frac{4\pi^2 \alpha_j^2}{l^2}\right)v^f_{i\alpha}(m\delta t)\delta t\\
\\
-\sum_{j=1}^n\sum_{\gamma \in {\mathbb Z}^n}\frac{2\pi i \gamma_j}{l}v^f_{j(\alpha-\gamma)}(m\delta t)v^f_{i\gamma}(m\delta t)\delta t\\
\\
+2\pi i\alpha_i1_{\left\lbrace \alpha\neq 0\right\rbrace}\frac{\sum_{j,k=1}^n\sum_{\gamma\in {\mathbb Z}^n}4\pi^2 \gamma_j(\alpha_k-\gamma_k)v^f_{j\gamma}(m\delta t)v^f_{k(\alpha-\gamma)}(m\delta t)}{\sum_{i=1}^n4\pi^2\alpha_i^2}\delta t+f_{i\alpha}(m\delta t)\delta t.
\end{array} 
\end{equation}
We may abbreviate the nonlinear Euler term by
\begin{equation}
\begin{array}{ll}
e_{ij\alpha\gamma}(m dt)=-\frac{2\pi i (\alpha_j-\gamma_j)}{l}v^f_{i(\alpha-\gamma)}(m\delta t)\\
\\
+2\pi i\alpha_i1_{\left\lbrace \alpha\neq 0\right\rbrace}
4\pi^2 \frac{\sum_{k=1}^n\gamma_j(\alpha_k-\gamma_k)v^f_{k(\alpha-\gamma)}(m\delta t)}{\sum_{i=1}^n4\pi^2\alpha_i^2}.
\end{array}
\end{equation}
Note that with this abbreviation (\ref{navode200first}) becomes
\begin{equation}\label{navode200second}
\begin{array}{ll}
v^f_{i\alpha}((m+1)\delta t)=v^f_{i\alpha}(m\delta t)+\sum_{j=1}^n\nu \left( -\frac{4\pi^2 \alpha_j^2}{l^2}\right)v_{i\alpha}(m\delta t)\delta t\\
\\
+\sum_{j=1}^n\sum_{\gamma\in {\mathbb Z}^n}e_{ij\alpha\gamma}(m\delta t)v^f_{j\gamma}(m\delta t)\delta t+f_{i\alpha}(m\delta t)\delta t.
\end{array} 
\end{equation}
We have not defined the external force functions modes $f_{i\alpha}$ in the case where for some $1\leq j\leq n$ we have $\alpha_j<0$ yet. In order to obtain lower bounds for our scheme, we define these modes dynamically, i.e., we define
\begin{equation}\label{fialpha2}
\begin{array}{ll}
f_{i\alpha}(m\delta t)=-v^f_{i\alpha}(m\delta t)-\sum_{j=1}^n\nu \left( -\frac{4\pi^2 \alpha_j^2}{l^2}\right)v_{i\alpha}(m\delta t)\delta t\\
\\
-\sum_{j=1}^n\sum_{\gamma\in {\mathbb Z}^n}e_{ij\alpha\gamma}(m\delta t)v^f_{j\gamma}(m\delta t)\delta t,
\end{array}
\end{equation}
if $\alpha_i< 0 \mbox{ for }~1\leq i\leq n$. This means that
\begin{equation}
v^f_{i\alpha}(m\delta t)=0\Longrightarrow v^f_{i\alpha}((m+1)\delta t)=0 ~\mbox{if}~\alpha_i< 0 \mbox{ for }~1\leq i\leq n.
\end{equation}
This way it is easier to get dynamic lower bounds for the velocity component functions starting with zero initial data.
Note that $f_i\in L^ 2$ with this definition as long as the velocity function is smooth at least.

It is easy to check that the scheme above leads to real solutions, i.e.
$\forall 1\leq i\leq n~\forall m\geq 1~\forall x:~v^f_{i}(m\delta t,x)$ is real (where in a hyperfinite scheme these the velocity component functions values are taken to be standard (shadows). As for $m=0$ we have $v^f_{i}(m\delta t,x)=v^f_{i}(0,x)=0$ in the situation considered such that for the Euler scheme above we have
\begin{equation}
v^f_{i\alpha}(\delta t)=f_{i\alpha}(0)\delta t
\end{equation}
for all $\alpha \in {\mathbb Z}^n$.
In this form the damping effect of the unbounded Laplacian is not obvious. 
Therefore, Trotter product formulas (similar as in \cite{KTr1}) can be used, but with a source term.
We denote $\mathbf{v}^F=(v^F_1,\cdots v^F_n)^T$ with $n$ infinite vectors $v^F_i=(v_{i\alpha})_{\alpha \in {\mathbb Z}^n}$. Moreover, in the following the symbol $\doteq$ means that the identity holds up to an infinitesimal error. Furthermore, the entries in $(\delta_{ij\alpha\beta})$ are Kronecker-$\delta$s which describe the unit $n{\mathbb Z}^n\times n{\mathbb Z}^n$-matrix and we denote
\begin{equation}
\mathbf{f}^F=(f_1,\cdots ,f_n)^T,~\mbox{where}~f_i=(f_{i\alpha})_{\alpha\in {\mathbb Z}^n},~1\leq i\leq n.
\end{equation}
At each time step $m$ we have
\begin{equation}
\begin{array}{ll}
\mathbf{v}^F((m+1)\delta t)\doteq
\left( \delta_{ij\alpha\beta}\exp\left(-\nu 4\pi^2 \sum_{i=1}^n\alpha_i^2 \delta t \right)\right)\times\\
\\
\times \left( \exp\left( \left( \left( e_{ij\alpha\beta}(m\delta t)\right)_{ij\alpha\beta}\right)\delta t\right)  \right)\mathbf{v}^F(m\delta t)+\mathbf{f}^F(m\delta t)
\end{array}
\end{equation}
(as a representation for $\mathbf{v}^F(m\delta t)$) solves the equation (\ref{navode200first}) with an error of order $\delta t^2$.
The equality up to infinitesimals $\doteq$ still holds of we evaluate exponentials up to first order, i.e., we have
\begin{equation}\label{aabb}
\begin{array}{ll}
\mathbf{v}^{F}((m+1)\delta t)\doteq \left( \delta_{ij\alpha\beta}\left(1-\nu 4\pi^2\sum_{i=1}^n\alpha_i^2 \delta t \right)\right)\times\\
\\
\times \left( 1+ \left( \left( e_{ij\alpha\beta}\right)_{ij\alpha\beta}(m\delta t)\right)\delta t \right)  \mathbf{v}^F(m\delta t)+\mathbf{f}^F(m\delta t).
\end{array}
\end{equation}
A precise formal interpretation of the scheme may be given in a nonstandard framework of enlarged universes or in other analytical frameworks such as the functional analytic frame work of Connes. 
We provide here a more informal argument for a bow up at the inviscid limit ($\nu=0$) and will give more details elsewhere. 
Now, for an appropriate choice of $f_i\in L^{2},~1 \leq i\leq n$ we have  $f_{i\alpha}\sim\frac{1}{|\alpha|^{1.5+\epsilon}}$ for small $\epsilon >0$ and the analysis of a hyperfinite scheme with appropriate time step size $\delta t$ leads first to
product term
\begin{equation}
{\Big |}\sum_{j=1}^n\sum_{\gamma\in {\mathbb Z}^n}e_{ij\alpha\gamma}(\delta t)v^f_{j\gamma}(\delta t){\Big |}_{\nu=0}\gtrsim\frac{1}{|\alpha|^{5+2\epsilon-1-3}},
\end{equation}
such in the next step the velocity component $v_{i\alpha}(2\delta t)$ with $\alpha_i\geq 0$ for all $1\leq i\leq n$ is of order
\begin{equation}
\frac{1}{|\alpha|^{2+2\epsilon }}.
\end{equation}
This implies that the next Euler  product term gets an order
\begin{equation}
\gtrsim\frac{1}{|\alpha|^{4+4\epsilon-1-3}}=\frac{1}{|\alpha|^{4\epsilon}}
\end{equation}
such that in the next step the velocity component $v_{i\alpha}(2\delta t)$ with $\alpha_i\geq 0$ for all $1\leq i\leq n$ is of order
\begin{equation}
\frac{1}{|\alpha|^{1+4\epsilon }},
\end{equation}
and so on, indicating a blow-up.
More recently still singular solutions of an averaged model have been announced in \cite{T}. Independently of the pending verification and its result this model is a good example in order to illustrate why an auto-controlled global regular scheme cannot be applied. Working in analogy with the Katz -Pavlovic model T. Tao proposes a construction of a blow up solution of an averaged equation solution $u^a(.)$ of
\begin{equation}
u^a(t)=e^{\Delta t}u^a_0+\int_0^te^{(t-s)\Delta }\tilde{B}\left(u^a(s),u^a(s) \right)ds 
\end{equation}
with the  averaged nonlinear Euler term
\begin{equation}
\begin{array}{ll}
\tilde{B}(u,v):=\\
\\
\int_{\Omega} \mbox{Rot}_{R^{-1}_{3,\omega}}\overline{m_{3,\omega}(D)}B\left(m_{1,\omega}(D)\mbox{Rot}_{R_{1,\omega}}u, m_{2,\omega}(D)\mbox{Rot}_{R_{2,\omega}}v\right)  d\mu ({\omega}), 
\end{array} 
\end{equation}
where $m_{i}(D):\Omega\rightarrow {\cal M}_0,~\hat{m_i(D)(\xi)u}(\xi)=m_i(\xi)\hat{u}(\xi)$ is a random Fourier multiplier and ${\cal M}_0$ is a $\sigma$-algebra in a measure space $\left(\Omega,\mu,{\cal M}_0\right)$, and $\mbox{Rot}_{R_i,\omega}$ are random rotations with values in $SO_3$ (cf. \cite{T} for details). Similar as in true Navier Stokes model the equation for the related function
 \begin{equation}
u^{\lambda,a}(t)=\lambda u^a(s) 
\end{equation}
\begin{equation}
u^{\lambda,a}(t)=e^{\Delta t}u^{\lambda,a}_0+\int_0^te^{(t-s)\Delta }\lambda\tilde{B}\left(u^{\lambda,a}(s),u^{\lambda,a}(s) \right)ds, 
\end{equation}
which gives an additional small parameter $\lambda$ to the nonlinear term. The effect of this additional $\lambda>0$ may be similar as in the case of the ODE $\stackrel{\cdot}{x}^{\lambda}=\lambda x^2,~x^{\lambda}(0)=\lambda x_0\neq 0$ considered above, where a singularity at $t=1$ of the solution $\frac{1}{1-t}$ of the $\stackrel{\cdot}{x}= x^2,~x^(0)=x_0\neq 0$  to a 'delayed' singularity at $t=\frac{1}{\lambda}$ of the solution $\frac{1}{1-\lambda t}$ of the former equation with parameter $\lambda$.  Furthermore, similarly as in the case of the original Navier Stokes equation models an auto-control function via time dilatation can be set up. So there is no difference from the perspective of these two features of the auto-controlled scheme between an averaged model and the original model. However both features in combination with a spatial effect make the difference. Let us consider this in more detail. For $t_0\geq 0$ and a time interval $\left[t_0,t_0+T\right]$ for some $T\in (0,1)$ we consider a subscheme for a comparison function $u^{t_0}_i,~\leq i\leq n$ for the Navier Stokes velocity component function $v_i,~1\leq i\leq n$ based on localized transformations of the form
\begin{equation}\label{uvloc}
\lambda(1+(t-t_0))u^{t_0}_i(\tau,x)=v_i(t ,x),~\lambda>0,
\end{equation}
where 
\begin{equation}\label{timedil}
(\tau (t) ,x)=\left(\frac{t-t_0}{\sqrt{1-(t-t_0)^2}},x\right).
\end{equation}
On the considered time interval this leads to
\begin{equation}
\frac{\partial}{\partial t}v_i(t,x)=\rho u_i(\tau,x)+\lambda(1+(t-t_0))\frac{\partial}{\partial \tau}u_i(\tau,x)\frac{t-t_0}{\sqrt{1-(t-t_0)^2}^3}.
\end{equation}
% Alternatively we could use the transformations of the form
% \begin{equation}\label{uvmu}
% (1+t)^{\mu}u_i(\tau,x)=v_i(t ,x),~\mu>0
% \end{equation}
% (or similar transformation with factor $\rho$ as above), and have
% \begin{equation}
% \frac{\partial}{\partial t}v_i(t,x)=\mu (1+t)^{\mu-1}u_i(\tau,x)+(1+t)^{\mu}\frac{\partial}{\partial \tau}u_i(\tau,x)\frac{d \tau}{d t}.
% \end{equation}
% Prima facie, it seems that the parameter $\mu$ measures the order of local growth of local upper bounds (consider  the scheme below): if $u_i$ has preservation of a local upper bound, then $v_i$ has a local growth of order $\mu>0$ with respect to time, especially linear growth in the case $\mu=1$. However, we shall see below that the transformation with $\mu=1$ can be used in order to obtain global uniform upper bounds which do not depend on time. Hence we forget the parameter $\mu >0$ in the following and may use this symbol otherwise.  
%\end{rem}
Denoting the inverse of $\tau(t)$ by $t(\tau)$, for the modes $u_{i\alpha},~\alpha\in {\mathbb Z}^n$ of $u_i,~1\leq i\leq n$ we get the equation   
\begin{equation}\label{navode200firsttimedil}
\begin{array}{ll}
\frac{d u_{i\alpha}}{d\tau}=\sqrt{1-t(\tau)^2}^3
\sum_{j=1}^n\nu \left( -\frac{4\pi \alpha_j^2}{l^2}\right)u_{i\alpha}-\\
\\
\lambda(1+t(\tau))\sqrt{1-t(\tau)^2}^3\sum_{j=1}^n\sum_{\gamma \in {\mathbb Z}^n}\frac{2\pi i \gamma_j}{l}u_{j(\alpha-\gamma)}u_{i\gamma}+\\
\\
\lambda(1+t(\tau))\sqrt{1-t(\tau)^2}^3\frac{2\pi i\alpha_i1_{\left\lbrace \alpha\neq 0\right\rbrace}\sum_{j,k=1}^n\sum_{\gamma\in {\mathbb Z}^n}4\pi^2 \gamma_j(\alpha_k-\gamma_k)u_{j\gamma}u_{k(\alpha-\gamma)}}{\sum_{i=1}^n4\pi^2\alpha_i^2}\\
\\
-\sqrt{1-t^2(\tau)}^3(1+t(\tau))^{-1}u_{i\alpha}.
\end{array} 
\end{equation}
Now the crucial observation that spatial regularity (polynomial decay of some order $\geq n+2$ is inherited by the indicated scheme. Assume that at time $t_0\geq 0$ we have
\begin{equation}
\forall 1\leq i\leq n~\forall \alpha\in {\mathbb Z}^n:~{\big |}u^{t_0}_{i\alpha}(t_0){\big |}\leq \frac{C}{1+|\alpha|^{n+2}},
\end{equation}
where in the following we may assume that the torus diameter is $l=1$. 
Concerning the nonlinear Burgers term we get for some constants $C>0$ and $c>0$
\begin{equation}\label{cC}
\lambda\sum_{j=1}^n\sum_{\gamma \in {\mathbb Z}^n}{\big |}2\pi i \gamma_j{\big |}{\big |}u_{j(\alpha-\gamma)}{\big |}{\big |}u_{i\gamma}{\big |}\leq \sum_{\gamma\in {\mathbb Z}^n}\lambda \frac{C}{1+|\alpha-\gamma|^{n+2}}\frac{C}{1+|\gamma|^{n+1}}\leq \frac{\lambda cC^2}{1+|\alpha|^{n+2}},
\end{equation}
a growth which becomes small for small $\lambda >0$, especially  compared to the damping term of the auto-controlled scheme, i.e., the term
\begin{equation}
-\sqrt{1-t^2(\tau)}^3(1+t(\tau))^{-1}u_{i\alpha}.
\end{equation}
Similarly, concerning the Leray projection term we have for some $c,C>0$ (depending only on the dimension)
\begin{equation}\label{cCC}
\begin{array}{ll}
\lambda\frac{|2\pi i\alpha_i|1_{\left\lbrace \alpha\neq 0\right\rbrace}\sum_{j,k=1}^n\sum_{\gamma\in {\mathbb Z}^n}4\pi^2 |\gamma_j||(\alpha_k-\gamma_k)|{\big |}u_{j\gamma}{\big |}{\big |}u_{k(\alpha-\gamma)}{\big |}}{\sum_{i=1}^n4\pi^2\alpha_i^2}\\
\\
\leq \lambda\frac{|\alpha_i|}{\sum_{i=1}^n\alpha_i^2}\sum_{\gamma\in {\mathbb Z}^n}\frac{C}{1+|\alpha-\gamma|^{n+1}}\frac{C}{1+|\gamma|^{n+1}}\leq \frac{\lambda cC^2}{1+|\alpha|^{n+2}},
\end{array}
\end{equation}
where for $|\alpha|=0$ the Leray projection term cancels. Then a local contraction result shows that this regularity is preserved in the time interval $\left[t_0,t_0+T \right]$, i.e. we get 
\begin{equation}
\forall t\in [t_0,T]~\forall 1\leq i\leq n~\forall \alpha\in {\mathbb Z}^n:~{\big |}u^{t_0}_{i\alpha}(t){\big |}\leq \frac{C}{1+|\alpha|^{n+2}},
\end{equation}
where we may use the damping term of the auto-control scheme in order to ensure that we can use the same constants $c,C>0$ in the latter statement as before (note that we have a freedom of choice for $T\in (0,1)$, and a moderate choice such as $T=\frac{1}{2}$ ensures that the the damping term is dominant). All these observation hold for the incompressible Euler equation as well. However, this does not mean that there are no singular solutions for the Euler equation. Indeed, as we show below singular solutions for the incompressible Euler equation can be constructed via auto-controlled schemes as well. The upshot of this is that we loose uniqueness in the case of the incompressible Euler equation. For the incompressible Navier Stokes equation with the additional viscosity term we have uniqueness by the Cronwall lemma. It is not obvious to me whether the averaged Navier Stokes equation has a unique solution, but we note that uniqueness is needed in order to conclude that there is no global regular solution. The additional effect of the viscosity term (the Laplacian) is the damping of all non-zero modes and the strong damping of the higher order modes. This effect can be expressed and investigated by Trotter product formulas as considered in \cite{KTr1,KTr2}. However, it is obvious that the increment of the averaged Navier Stokes equation cannot be estimated by an auto-controlled scheme as in the case of the original Navier Stokes equation as described in the introduction and argued in detail below.
\section{Proof of theorem \ref{mainthm} by a global self-controlled scheme}
In a first step we define a natural iteration scheme for the velocity function.
In a second step we set up a global scheme with an auto-control caused by a damping term introduced via time coordinate transformation at each time step. In a third step  we prove local existence and  regularity of the Navier Stokes equation via a local contraction result for the original scheme using representations which involve {\it first order spatial derivatives} of the Gaussian. These representations are fundamentally different from weak schemes based on energy inequalities as the first order spatial derivative of the Gaussian allows for an upper bound such that the growth of the nonlinear terms can be offset by the damping term of the auto-controlled scheme and by viscosity damping. At  each time step of the original  iteration scheme for the velocity function of a time-local Navier Stokes equation we get - via time-dilatation- a global subscheme on a global domain (which is time-local in original coordinates), where the functional series members of the subscheme inherit local regularity of the local velocity value function components. The local regularity can also be derived for the functions with the damping term a fortiori, of course. In any case, the subscheme is used to prove certain growth properties of the velocity functions which are related to the preservation of upper bounds by the global subscheme of each time step of the original scheme for the velocity, where the damping terms of the subscheme are used at each substep. Variations of time dilatation transformations lead to a global linear upper bound of the velocity solution functions and their derivatives up to order $m\geq 2$ if the regularity assumptions of order $m\geq 2$ are satisfied. Furthermore, certain variations of the scheme lead to constructive forms of the argument. This refinement of the upper bound argument is also interesting as it facilitates the application of the local contraction result. Note that we used a similar idea in \cite{K} in order to prove the existence of singular vorticity solutions of the incompressible Euler equation, although the transformation is naturally different as we do not want to prove the existence of singularities but global regular existence. If we rely on local existence via contraction results of the original equation, then the proof becomes non-constructive. Constructive results can be obtained if we consider local contraction properties of the subscheme functions where we may take viscosity damping into account. Such local contraction properties hold a fortiori for the subschemes as they have a damping term. For the subscheme functions time steps can be chosen uniformly. However, this transfers to numerical (finite time step schemes) for the original scheme only if we can transfer the uniform upper bounds from the subschemes to the original velocity function.  This article is mainly on global linear upper bounds and analytical proofs of a global regular solution of the incompressible Navier Stokes equation, but we obtain a constructive version as well. Moreover, taking viscosity damping into account we get a global regular upper bound which is independent of time. Here, we observe that for a given subscheme applied at time  $t_0\geq 0$ the viscosity damping becomes can offset the growth of the nonlinear terms after finite time (not after infinitesimal time) if the data at are sufficiently large. If the data are smaller, then growth of the nonlinear terms can be allowed for some time. Both phenomena are in accordance with the observation that the Navier stokes equation operator does not define a contractive semigroup.

Next we define the global scheme. In order to have time step size $1$ along with small coefficients of the spatial derivative terms at each time step we sometimes introduce the time transformation $\tau =\rho t$ for small time step size $\rho >0$. For the subscheme we shall observe that the time step size $\rho>0$ can be chosen independently of the time step number of the subscheme. A global scheme with constant time step size can be obtained for the original time scheme for uniform global upper bounds. If we have time-linear upper bounds, then we have to rely on properties of the subscheme in order to set up a numerical scheme for the original velocity function. Note that the time step size with respect to original time may depend on the time step number $l\geq 1$ in the form $\rho_l\gtrsim \frac{1}{l}$, which still leads to a global scheme. However, we shall obtain constructive versions which allow for a constant time step size $\rho >0$ which is independent of the time step number especially. In any case, in the following we suppress the possible dependence of the time step size $\rho$ of the time step number $l\geq 1$ and will come back to this issue only in special situations when it is essential in order to simplify arguments. This policy of notation seems to be justified, as the growth estimates are via the subschemes of  transformed functions of comparison, where these subschemes have uniform time step size. Next we consider the local scheme for the velocity function of the original equation (\ref{Navleray0}). For each time step $l\geq 1$ we assume that we have determined data $v^{\rho,l}_i(l-1,.),~1\leq i\leq n$ at the previous time step number $l-1\geq 0$, where at the first time step $l=1$ we start the recursion with $v^{\rho,l}_i(0,.)=h_i$. At each time step the local functions $v^{\rho,l}_i:[l-1,l]\times {\mathbb R}^n\rightarrow {\mathbb R}$ are determined as solutions of the local equations in $(\tau,x)$-coordinates of the form   

\begin{equation}\label{Navleray2a}
\left\lbrace \begin{array}{ll}
\frac{\partial v^{\rho,l}_i}{\partial \tau}-\rho\nu\sum_{j=1}^n \frac{\partial^2 v^{\rho,l}_i}{\partial x_j^2} 
+\rho\sum_{j=1}^n v^{\rho,l}_j\frac{\partial v^{\rho,l}_i}{\partial x_j}=\\
\\ \hspace{1cm}\rho\sum_{j,m=1}^n\int_{{\mathbb R}^n}\left( \frac{\partial}{\partial x_i}K_n(x-y)\right) \sum_{j,m=1}^n\left( \frac{\partial v^{\rho,l}_m}{\partial x_j}\frac{\partial v^{\rho,l}_j}{\partial x_m}\right) (\tau,y)dy,\\
\\
\mathbf{v}^{\rho,l}(l-1,.)=\mathbf{v}^{\rho,l-1}(l-1,.),
\end{array}\right.
\end{equation}
on the domain $[l-1,l]\times {\mathbb R}^n$, where it is essential that the scheme inherits regularity of the final data of the previous time step $\mathbf{v}^{\rho,l-1}(l-1,.)=\left(v^{\rho,l}_1(l-1,.),\cdots ,v^{\rho,l}_n(l-1,.) \right)^T$ recursively (shown below).  
Next for each time step $l\geq 1$ we define a subscheme. First, at each time step $l$ we transform the local equation via a time dilatation coordinate transformation with $\sigma:[l-1,l)\rightarrow [0,\infty)$ (supressing a superscript $l$ of $\sigma$ since it will be implicitly clear where we are) and of the form
\begin{equation}\label{timedil}
(\sigma (\tau_l) ,x)=\left(\frac{\tau_l}{\sqrt{1-\tau_l^2}},x\right), 
\end{equation}
(where $\tau_l=\tau-(l-1)$), to the related local equation for the function
\begin{equation}\label{uvtrans}
(1+\tau )u^{\rho,l}_i(\sigma,x)=v^{\rho,l}_i(\tau ,x),
\end{equation}
and then to the function $u^{\rho,l}_i,~1\leq i\leq n$ itself. Note that in (\ref{uvtrans}) the $\tau$ of the factor in front of $u^{\rho,l}_i$ is global and the $\tau_l$ is local. Later we shall see that a localized and global versions of this transformation, where at time step $l$ we choose
\begin{equation}\label{uvl0}
\lambda(1+(\tau-(l-1)) )u^{\rho,l}_i(\sigma,x)=v^{\rho,l}_i(\tau ,x),
\end{equation}
or
\begin{equation}\label{uvlglob}
\lambda(1+\tau )u^{\rho,l}_i(\sigma,x)=v^{\rho,l}_i(\tau ,x),
\end{equation}
along with $\tau\in [l-1,l)$ and some appropriate (in general small) constant $\lambda>0$ lead to constructive versions of the scheme. This $\lambda$ is not essential but convenient as it simplifies some estimates of damping. Sometime we shall set this parameter equal to $1$, but for local forms of the time delay transformation we shall use it as it simplifies the prove of preservation of regular upper bounds for the localized scheme. The effect on the equation is that the nonlinear terms get an additional coefficient factor $\lambda$ (while initial data become larger).   
Note the difference to the transformation in \cite{K} with respect to time: here, we have a factor $(1+\tau)$ which will lead to global upper bounds of the velocity component functions $v_i$ which are linear in time. In a localised interpretation of the argument the factor  $(1+(\tau-(l-1)) )$ of $u^{\rho,l}$ is in the interval $[1,2]$ and the damping (in form of viscosity damping and potential damping) is designed in order to compensate for the growth of the velocity functions and their multivariate derivatives up to some order $p$ which is not only sufficient in order to prove global regular existence but also to prove global regular bounded solutions.
\begin{rem}
We shall use a localized transformation similar as in (\ref{uvl0}) in order to show the statement of decay to zero at infinite time as stated in Corollary \ref{maincor}.
\end{rem}

\begin{rem}
Alternatively and with a constant parameter $\mu_0>0$ we could use the transformations of the form
\begin{equation}\label{alt}
(1+\tau)^{\mu_0}u^{\rho,l}_i(\sigma,x)=v^{\rho,l}_i(\tau ,x),~\mu>0
\end{equation}
(or similar transformation with factor $\rho$ as above), and have
\begin{equation}
\frac{\partial}{\partial \tau}v^{\rho,l}_i(\tau,x)=\mu_0 (1+t)^{\mu_0-1}u^{\rho,l}_i(\sigma,x)+(1+\tau)^{\mu_0}\frac{\partial}{\partial \sigma}u^{\rho,l}_i(\sigma,x)\frac{d \sigma}{d \tau},
\end{equation}
where for some $\mu_0>0$ a variation of the argument leads to global linear upper bounds. \end{rem}
   
We denote the inverse of the time dilatation in (\ref{timedil}) by $\tau\equiv \tau(\sigma)$.
Suppressing subscripts $l$ of $\tau_l$ for convenience of notation note that for the transformation in (\ref{uvtrans}) we have
\begin{equation}
\frac{\partial}{\partial \tau}v^{\rho,l}_i(\tau ,x)=u^{\rho,l}_i(\sigma,x)+(1+\tau)\frac{\partial}{\partial \sigma}u^{\rho,l}_i(\sigma,x)\frac{d \sigma}{d \tau},
\end{equation}
where 
\begin{equation}
\begin{array}{ll}
\frac{d\sigma}{d \tau}=\frac{d}{d\tau}\left(\frac{\tau}{\sqrt{1-\tau^2}}\right)
=\frac{1}{\sqrt{1-\tau^2}}+\frac{-\frac{1}{2}\tau (-2\tau)}{\sqrt{1-\tau^2}^3}=\frac{1}{\sqrt{1-\tau^2}^3}.
\end{array}
\end{equation}
 Again $\tau$ and $\sigma$ may have a superscript ornament $l$ but we know implicitly where we are and may keep notation simple.
For each time step $l\geq 1$ the equation for $u^{\rho,l}_i$ is an equation on the whole domain $[0,\infty)\times {\mathbb R}^n$. Abbreviating for $k\in \left\lbrace 0,1,2 \right\rbrace $
\begin{equation}\label{mu}
\mu =\mu(\sigma)=\frac{\sqrt{1-\tau^2(\sigma)}^3}{1+\tau(\sigma)},~\mu^{\tau, k}:=(1+\tau(\sigma))^k\mu,
\end{equation}
where for all $\sigma\in [0,\infty)$ $\tau(\sigma)\in [l-1,l)$ and $\tau_l(\sigma)\in [0,1)$ at time step $l$, the equation for $u^{\rho,l}_i,~1\leq i\leq n$ is of the form (let $\lambda =1$ for simplicity)
\begin{equation}\label{Navleray2}
\left\lbrace \begin{array}{ll}
\frac{\partial u^{\rho,l}_i}{\partial \sigma}-\rho\mu^{\tau,1}\nu\sum_{j=1}^n \frac{\partial^2 u^{\rho,l}_i}{\partial x_j^2} 
+\rho\mu^{\tau,2}\sum_{j=1}^n u^{\rho,l}_j\frac{\partial u^{\rho,l}_i}{\partial x_j}+\mu u^{\rho,l}_i=\\
\\ \rho \mu^{\tau,2}\sum_{j,m=1}^n\int_{{\mathbb R}^n}\left( \frac{\partial}{\partial x_i}K_n(x-y)\right) \sum_{j,m=1}^n\left( \frac{\partial u^{\rho,l}_m}{\partial x_j}\frac{\partial u^{\rho,l}_j}{\partial x_m}\right) (\sigma,y)dy,\\
\\
\mathbf{u}^{\rho,l}(0,.)=\frac{1}{1+(l-1)}\mathbf{v}^{\rho,l-1}(l-1,.),
\end{array}\right.
\end{equation}
and may to be solved on the whole domain $[0,\infty)\times {\mathbb R}^n$. For more constructive arguments we may consider such transformations on the time interval $\left[l-1,l-\frac{1}{2}\right]$ which lead to transformed subproblems on finite time horizons. In any case we need localized transformations in order to make the argument constructive for large times as in (\ref{uvl0}). If we choose a transformation as in (\ref{uvl0}), then we have a dominant damping term independently of the size of time $\tau$, but we have to show then that the upper bounds are preserved for the original velocity component functions in a strict sense. This is only possible if we take viscosity damping into account. In contrast for a coefficient as in (\ref{uvtrans}) we have linear growth with respect to time of the terms with spatial derivatives while the damping term decreases linearly with respect to time. This is no obstacle to prove global existence of a solution analytically, but it is clear that we need some localized variation in order to get a numerical and efficient scheme. The analysis of the localized scheme shows that the semigroup induced by Navier Stokes equation operator has no strong contractive property. The transformations in (\ref{uvtrans}) or in (\ref{uvlglob}) are designed for an analytical global regular existence proof for the original Navier Stokes equation in (\ref{Navleray2}) via global upper bounds which are linear with respect to time. Anyway, the difference of the original velocity function to the transformed comparison function $u^{\rho,l}_i$ is that the latter has the potential term
\begin{equation}
\mu u^{\rho,l}_i,
\end{equation}
where this term has the 'right' sign as it has a damping effect, which gives us hopes that we can prove the existence of global upper bounds. Note that this potential term is the only term along with the partial time derivative term  of the value function $u^{\rho,l}_i,~1\leq i\leq n$, which does not have the small time step size $\rho$ as a coefficient, and, hence, may dominate all the spatial terms after some local time of the time step, whenever the value functions become large. This observation increases that hopes. 
Note also that we suppress an upper script $l$ for $\mu$ as we did for $\tau$, since we always know where we are. Next for each time step $l$ having determined $\frac{1}{1+(l-1)}\mathbf{v}^{\rho,l-1}(l-1,.)$ we solve (\ref{Navleray2}) by a subscheme of functions. Although for given $l$ and $k\in \left\lbrace 0,1,2\right\rbrace$ we have
\begin{equation}\label{mucoeffrem}
\mu^{\tau,k}(\sigma)=\frac{\sqrt{1-\tau^2(\sigma)}^3}{1+\tau(\sigma)}(1+\tau(\sigma))^k\leq \sqrt{1-\tau^2(\sigma)}^3(1+\tau(\sigma))\leq c
\end{equation}
for some constant $c>0$ dependent on the time horizon (as $\tau\equiv \tau_l(\sigma)\in [0,1]$ at each $l$), for any finite time horizon $T>0$ we have small coefficients with time step size factor $\rho$ everywhere except for the partial time derivative and the damping potential term. Note that any upper bound $c>0$ in (\ref{mucoeffrem}) depends on the time step number $l$ or on a finite time horizon $T$, but once an arbitrary finite time horizon $T>0$ is chosen the coefficients become small for small $\rho$. There is no such restriction for local transformations of course. At this point we cannot be sure that the potential term really dominates the Leray projection term (in case of a multivariate Burgers equation this is different of course). For this reason we have to set up a doubled subscheme. First we define a family of recursively defined local nonlinear incompressible Navier Stokes equations with damping terms, where for each $l\geq 1$ and each $m\geq 1$ the value function  $u^{\rho,l,m}_i,~1\leq i\leq n$ is a solution of
\begin{equation}\label{Navleray3}
\left\lbrace \begin{array}{ll}
\frac{\partial u^{\rho,l,m}_i}{\partial \sigma}-\rho\mu^{\tau,1}\nu\sum_{j=1}^n \frac{\partial^2 u^{\rho,l,m}_i}{\partial x_j^2} 
+\rho\mu^{\tau,2}\sum_{j=1}^n u^{\rho,l,m}_j\frac{\partial u^{\rho,l,m}_i}{\partial x_j}+\mu u^{\rho,l,m}_i=\\
\\ \rho \mu^{\tau,2}\sum_{j,p=1}^n\int_{{\mathbb R}^n}\left( \frac{\partial}{\partial x_i}K_n(x-y)\right) \sum_{j,p=1}^n\left( \frac{\partial u^{\rho,l,m}_p}{\partial x_j}\frac{\partial u^{\rho,l,m}_j}{\partial x_p}\right) (\sigma,y)dy,\\
\\
\mathbf{u}^{\rho,l,m}(m-1,.)=\mathbf{u}^{\rho,l,m-1}(m-1,.),
\end{array}\right.
\end{equation} 
and on the domain $[m-1,m]\times {\mathbb R}^n$. In addition for $m=1$ we assume
\begin{equation}
\mathbf{u}^{\rho,l,1}(0,.)=\frac{1}{1+(l-1)}\mathbf{v}^{\rho,l-1}(l-1,.).
\end{equation}
Note that $\sigma$ is now restricted to $[l-1,l]$ as is $\mu^{\tau,k}=\mu^{\tau,k}(\sigma)=\frac{\sqrt{1-\tau^2(\sigma)}^3}{1+\tau(\sigma)}(1+\tau(\sigma)^k),~\sigma\in [m-1,m],~k\in \left\lbrace 0,1,2\right\rbrace$. Note that $\mu=\mu^{\tau,0}$ in our notation. We could make this more explicit adding to $\mu$ and $\mu^{\tau,k}$ a double superscript $l,m$, but, as the value function $u^{\rho,l,m}_i$ bears this double superscript we know where we are and suppress such kind of notation for the sake of simplicity. 
Note that we do not need to solve the time-local equation in (\ref{Navleray3}) by a local iteration scheme if we have solved the local equation for $v^{\rho,l}_i,~1\leq i\leq n$, since $u^{\rho,l,m}_i,~1\leq i\leq n$ inherits the relevant properties from local solutions of $v^{\rho,l}_i,~1\leq i\leq n$. We may use the local scheme for the velocity function $v^{\rho,l}_i,~1\leq i\leq n,$  in order to prove local regular existence at time step $l\geq 1$, and then we use the scheme of the functions $u^{\rho,l,m}_i,~1\leq i\leq n$ (which is local with respect to original time and global with respect to transformed time) in order to prove that this scheme inherits global upper bounds for all substeps $m\geq 1$. This is one option. Well it leads to versions of global existence proofs which have some non-constructive features. One non-constructive feature is that we have infinitely many Cauchy problems for $u^{\rho,l,m}_i$ on infinite domains. Another non-constructive feature is that we may need transfinite induction if we derive the step size from the local equations for the functions $v^{\rho,l}_i,~1\leq i\leq n$. Well, constructive variations can be obtained by a) using step size $\left[l-1,l-\frac{1}{2} \right]$ or similar step sizes in order to have subschemes for $u^{\rho,l,m}_i$ of finite time horizon, and b) we may derive contraction results for the subscheme equation $u^{\rho,l,m}_i$ itself.    The interplay of both schemes leads then to a global  upper bound of the velocity function via a recursion argument. As we indicated above the type of upper bounds which we obtain depend on the variation of argument which we consider. For the scheme considered above the upper bound is linear in time while for certain localized schemes we get a stronger result of a uniform global upper bound independent of time. We emphasize again that some variations of arguments are constructive while others are not, but even the simpler variations of arguments lead to a global regular existence proof from the analytical point of view. Next we prove local existence of regular solutions $v^{\rho,l}_i,~1\leq i\leq n,~l\geq 1$ for the equation system in (\ref{Navleray2a}). Local existence means that there is some step size $\rho >0$ for which we have $v^{\rho,l}_i\in C^1\left(\left(l-1,l\right],H^q\cap C^q  \right)\cap C^0\left(\left[l-1,l\right],H^q\cap C^q  \right)$ for $l\geq 2$ and $q\geq 4$  if the data $v^{\rho,l-1}_i(l-1,.)\in H^q\cap C^q$ for $m\geq 2$. At the first time step we have $v^{\rho,l}_i\in C^1\left(\left(l-1,l\right],H^q\cap C^q  \right)\cap C^0\left(\left[l-1,l\right],H^q\cap C^q  \right)$ (cf. remark \ref{initialrem}).

In order to obtain such local existence results, we consider the functional series $v^{\rho,l,m}_i,~1\leq i\leq n,~m\geq 0$, where each member $v^{\rho,l,m}_i,~1\leq i\leq n$ satisfies an equation of the form
\begin{equation}\label{Navleray4}
\left\lbrace \begin{array}{ll}
\frac{\partial v^{\rho,l,m}_i}{\partial \tau}-\rho\nu\sum_{j=1}^n \frac{\partial^2 v^{\rho,l,m}_i}{\partial x_j^2}
+\rho\sum_{j=1}^n v^{\rho,l,m-1}_j\frac{\partial v^{\rho,l,m-1}_i}{\partial x_j}=\\
\\ \rho\sum_{j,p=1}^n\int_{{\mathbb R}^n}\left( \frac{\partial}{\partial x_i}K_n(x-y)\right) \sum_{j,p=1}^n\left( \frac{\partial v^{\rho,l,m-1}_p}{\partial x_j}\frac{\partial v^{\rho,l,m-1}_j}{\partial x_p}\right) (\tau,y)dy,\\
\\
\mathbf{v}^{\rho,l,m}(l-1,.)=\mathbf{v}^{\rho,l-1}(l-1,.).
\end{array}\right.
\end{equation}
For each iteration step $m$ we are on the domain $[l-1,l]\times {\mathbb R}^n$ of the original scheme (in transformed time coordinates). The local solution function is constructed via the representation
\begin{equation}
v^{\rho,l}_i=v^{\rho,l,0}_i+\sum_{m= 1}^{\infty}\delta v^{\rho,l,m}_i,~1\leq i\leq n,
\end{equation}
along with the increments
\begin{equation}
\delta v^{\rho,l,m}:=v^{\rho,l,m}-v^{\rho,l,m-1},~\delta v^{\rho,l,0}:=v^{\rho,l,0}-v^{\rho,l-1}
\end{equation}
for $m\geq 1$ and $m=0$ respectively.
Here, for $m=1$ we define $v^{\rho,l,0}_i$ to be the solution of the linearized equation
\begin{equation}\label{Navleray5}
\left\lbrace \begin{array}{ll}
\frac{\partial v^{\rho,l,0}_i}{\partial \tau}-\rho\nu\sum_{j=1}^n \frac{\partial^2 v^{\rho,l,0}_i}{\partial x_j^2}
+\rho\sum_{j=1}^n v^{\rho,l-1}_j(l-1,.)\frac{\partial v^{\rho,l-1}_i(l-1,.)}{\partial x_j}=\\
\\ \rho \sum_{j,p=1}^n\int_{{\mathbb R}^n}\left( \frac{\partial}{\partial x_i}K_n(x-y)\right) \times \\
\\
\times \sum_{j,p=1}^n\left( \frac{\partial v^{\rho,l,l-1}_p(l-1,.)}{\partial x_j}\frac{\partial v^{\rho,l-1}_j(l-1,.)}{\partial x_p}\right) (\tau,y)dy,\\
\\
\mathbf{v}^{\rho,l}(l-1,.)=\mathbf{v}^{\rho,l-1}(l-1,.),
\end{array}\right.
\end{equation}
and the increment $\delta v^{\rho,l,m}_i$ satisfies 
\begin{equation}\label{Navleray6}
\left\lbrace \begin{array}{ll}
\frac{\partial \delta v^{\rho,l,m}_i}{\partial \tau}-\rho\nu\sum_{j=1}^n \frac{\partial^2 \delta v^{\rho,l,m}_i}{\partial x_j^2}+\rho\sum_{j=1}^n \delta v^{\rho,l,m-1}_j\frac{\partial v^{\rho,l,m-1}_i}{\partial x_j}\\
\\
+\rho\sum_{j=1}^n v^{\rho,l,m-2}_j\frac{\partial \delta v^{\rho,l,m-1}_i}{\partial x_j}\\
\\
=\rho \sum_{j,p=1}^n\int_{{\mathbb R}^n}\left( \frac{\partial}{\partial x_i}K_n(x-y)\right) \sum_{j,p=1}^n\left( \frac{\partial \delta v^{\rho,l,m-1}_p}{\partial x_j}\frac{\partial v^{\rho,l,m-1}_j}{\partial x_p}\right) (\tau,y)dy\\
\\
+ \rho\sum_{j,p=1}^n\int_{{\mathbb R}^n}\left( \frac{\partial}{\partial x_i}K_n(x-y)\right) \sum_{j,p=1}^n\left( \frac{\partial v^{\rho,l,m-2}_p}{\partial x_j}\frac{\partial \delta v^{\rho,l,m-1}_j}{\partial x_p}\right) (\tau,y)dy,\\
\\
\delta \mathbf{v}^{\rho,l,m}(l-1,.)=0.
\end{array}\right.
\end{equation}
Note that for $m=1$ we have $v^{\rho,l,m-2}_j=v^{\rho,l,-1}_j:=v^{\rho,l-1}_j(l-1,.)$ by definition.
The next step is to prove a local contraction and existence result via the local scheme described. Functions in $H^m\cap C^m$ have the convenience that they are bounded, and that their spatial derivatives up to order $m$ are bounded. As the auto-controlled scheme can offset the growth for some semi-groups such upper bounds are preserved by the scheme.
It may be found convenient then to estimate some of the nonlinear terms with their products of functions and spatial derivatives of functions using pointwise upper bounds, i.e., instead of using standard estimates of the form
\begin{equation}
{\big |}fg{\big |}_{H^m}\leq C_m {\big |}f{\big |}_{H^m}{\big |}g{\big |}_{H^m}
\end{equation}
for $m\geq 2$, we may use the fact that one factor and its spatial derivatives up to order $m\geq 2$ have a pointwise upper bound. As such upper bounds are preserved we shall assume in the following that $C$ is generically used such that it stands also for pointwise upper bounds of the value functions and its spatial derivatives up to order $m\geq 2$ if this is convenient. For certain schemes we shall determine $C>0$ explicitly, and in general it will be larger than the ${\big |}.{\big |}$-norm.

We underline the latter statements by adding a preliminary remark about the use of the spatial function space $H^p\cap C^p$ is in order here. The function space $C^p$ of continuous functions with continuous multivariate derivatives up to order $p$ is certainly not a closed space, but certain restrictions to compact domains are. For a functional series $(f_m)$ with $f_m\in C^p\cap H^p$  we now that for all $0\leq |\alpha|\leq p$ the multivariate derivative functions $D^{\alpha}_xf_n$ are in $C_0\equiv C_0\left( {\mathbb R}^n\right) $, where the latter function space is the close function space of continuous functions with decay to zero at infinity - note that $C_0(X)$ is a closed function space for any locally compact Hausdorff space $X$. According to the energy form of the Sobolev lemma we then have for $p>\frac{n}{2}$ (especially for $p\geq 2$ for $n=3$) that the limit function $f=\lim_{n\uparrow {\mathbb N}} f_{n}$ is H\"{o}lder continuous of some positive exponent and bounded as there is decay to zero at infinity.
However in our construction the approximating functions are located in the function space $C^p\cap H^p$, and this leads to contraction with spatial dependence in $H^{p,\infty}\cap C^2$
at the same time. Or, more precisely, for $0\leq |\alpha|\leq p$ the approximating functions $D^{\alpha}_xv^{\rho,l,k}_i,~1\leq i\leq n,~k\geq 0,~l\geq 1$ are in the function space $C_b$ of bounded continuous functions. This is not really needed as it is always possible to use standard estimates for products of functions in Sobolev norms, but it is convenient. Note that standard estimates for products of Sobolev norms can always be used as we have representations in the form (\ref{deltavrholrep}) where for the estimates of second order derivatives the Burgers term and the Leray projection term have only derivatives of order leads or equal to $m$ as the fundamental solution can take one spatial derivative and still has a weak singularity. 

The limit function of local scheme solutions $v^{\rho,l,m}_i,~1\leq i\leq n,$ of linearized equations approximating the local function $v^{\rho,l}_i,~1\leq i\leq n$ can   be used to find a classical representation of $v^{\rho,l}_i,~1\leq i\leq n$ in terms a a limit function $\lim_{m\rightarrow \infty}v^{\rho,l,m}_i,~1\leq i\leq n$, which serves as a bounded H\"{o}lder-continuous coefficient function of a fundamental solution equation in terms of which $v^{\rho,l}_i,~1\leq i\leq n$ is hypothetically represented. It is easy to verify that this is indeed the correct representation (with a coefficient which has indeed the regularity of a classical solution in $C^{1,2}$ with decay to zero at infinity). For higher regularity data in $C^m\cap H^m$ with $m>2$, and especially with $m\gg 2$ the argument can be simplified of course. This is one of the alternative ways to construct local classical solutions of the incompressible Navier Stokes equation. Concerning the verification that the limit function $\lim_{m\uparrow \infty}v^{\rho,l,m}_i,~1\leq i\leq n,$ is indeed the local classical solution $v^{\rho,l}_i,~1\leq i\leq n,$ of the local Navier Stokes equation Cauchy problem with data in $H^m\cap C^m,~m\geq 2$ we shall include a remark below. 
If we want to prove analytical results, then the result above together with the relations (\ref{uvrel2}) below for all time step numbers $l\geq 1$ dispenses us from proving local contraction and local existence results of the subscheme 
$u^{\rho,l,m}_i,~1\leq i\leq n,~m\geq 1$. 
However, the analysis of the equation for the scheme $u^{\rho,l,m}_i,~1\leq i\leq n$ is useful if we want to get a constructive existence proof which is also useful in order to 
construct numerical schemes. Well, let us first look at the velocity function components of a local scheme $v^{\rho,l}_i,~1\leq i\leq n,~l\geq 1$. 
 We have
\begin{thm}\label{loccontr}
Let $q\geq 2$. Given $v^{\rho,l-1}_i(l-1,.)\in C^q\cap H^q$ for $1\leq i\leq n$ along with
\begin{equation}
{\big |}v^{\rho,l-1}_i(l-1,.){\big |}_{C^q\cap H^q}\leq C
\end{equation}
for small $\rho >0 $ the 
local scheme functional sequence $v^{\rho,l,p}_{i}(\tau,.)\in C^q\cap H^q,~1\leq i\leq n$ and $\tau\geq 0$ has a limit $v^{\rho,l}_i,~1\leq i\leq n,$ along with
\begin{equation}
v^{\rho,l}_i=\lim_{p\uparrow \infty}v^{\rho,l,p}_i\in C^{1,q},~1\leq i \leq n,
\end{equation}
which solves the time-local Cauchy problem and such that $v^{\rho,l}_i(\tau,.)\in C^q\cap H^q$ for all $\tau\in [l-1,l]$. Here for time $l-1$ recall the specifications made in remark \ref{initialrem}. Furthermore, for small $\rho >0$ the first approximation increment at step $l$ satisfies 
\begin{equation}
\sup_{\tau\in [l-1,l]}{\big |}\delta v^{\rho,l,0}_i(\tau,.){\big |}_{C^q\cap H^q}\leq \frac{1}{4},
\end{equation}
and for $p+1\geq 1$ we have the contraction 
\begin{equation}
\sup_{\tau\in [l-1,l]}{\big |} \delta v^{\rho,l,p+1}_i(\tau,.){\big |}_{C^q\cap H^q}
\leq \frac{1}{2}\sup_{\tau\in [l-1,l]}{\big |}\delta v^{\rho,l,p}_i(\tau,.){\big |}_{C^q\cap H^q}
\end{equation}
for all $1\leq i\leq n$, $p\geq 1$.
\end{thm}
\begin{rem}
We have not determined the quantity of the time step size $\rho>0$ in the statement above because we use it only for analytical non-constructive versions of the argument. Note that for linear $C_l$ an estimate of the uncontrolled Navier Stokes equation leads contraction constant $\rho_l\sim \frac{1}{C_l^2}$. This does not lead to a finite scheme. However, we shall prove that there is no finite maximal $T>0$ such that for the velocity components of the Navier Stokes equation we have
\begin{equation}
\max_{1\leq i\leq n}{\big |}v_i(T_{max},.){\big |}_{H^q\cap C^q}\leq C+{T_{\max}}C
\end{equation}
for some $C>0$, or in the scheme we use that the equivalent claim that there is no maximal finite $\tau_{max}$ (and no finite maximal time step number $l$ corresponding to $\tau_{max}$) such that
\begin{equation}
\max_{1\leq i\leq n}{\big |}v^{\rho,l}_i(\tau_{max},.){\big |}_{H^q\cap C^q}\leq C+\tau_{max}C
\end{equation}
holds. For this we can argue by a contradiction argument. For this reason we considered the abstract existence of a $\rho$ in the statement of theorem (\ref{loccontr}) above.
Numerical schemes are much easier obtained for schemes with external control functions. The proof below, however, also leads to explicit upper bounds for the time step size if analyzed further. This is done in item (iii) below. Here we note that a refined analysis for a scheme with a global linear upper bound may lead to a contraction step size of order $\rho_l\sim \frac{1}{l}$ in general. Therefore an upper bound which is independent of the time horizon is a significant step forward from a numerical perspective. We shall use a localized form of the transformation in order to obtain such uniform global upper bounds below. Especially they lead to converging algorithmic scheme, which do not depend on the time step size. However, as a preservation of each type of upper bound is obtained by the subscheme with a time dilatation argument we can use even a simplified (weak) form of contraction result for analytical purposes of global existence. 
\end{rem}

\begin{proof}
In the following we provide a simplified proof of the simplified result stated where we have obtained stronger results elsewhere. The simplified argument is sufficient in order to get global upper bounds and a global regular existence result. We shall show below that the subscheme with the damping term will lead to the preservation of linear upper bounds.    
We denote multivariate derivatives with multiindices $\beta=(\beta_1,\cdots,\beta_n)$ of a function $f$ with respect to the spatial variables $x$ evaluated at $y$ by $D^{\beta}_xf(y)$. We assume recursively that at the beginning of time step $l\geq 1$ we have
\begin{equation}
{\big |}v^{\rho,l}_i(l-1,.){\big |}_{H^q\cap C^q}\leq C\equiv C_{l-1}.
\end{equation}
If we have a linear upper bound at the previous time step, then for some finite constant $c$ we have $C_{l-1}\leq c+(l-1)c$, and if we have a global uniform upper bound, then the constant $C>0$ is independent of the time step number $l\geq 1$. As we have a fixed upper bound constant $C$ at the beginning of time step $l$ we suppress the subscript $l$ anyway. The preservation of the linear upper bound or of the uniform global upper bound will be shown by a preservation of upper bounds by the subscheme and is not of further interest for the local result of this proof. 
% In the following we use classical representation of the solutions in term of the fundamental solution $p^l$ of the equation
% \begin{equation}
% \frac{\partial  p^l}{\partial \sigma}-\rho\nu\sum_{j=1}^n \frac{\partial^2 p^l}{\partial x_j^2}=0
% \end{equation}
% on the domain $[l-1,l]\times {\mathbb R}^n$.
For all $1\leq i\leq n$ and $0\leq |\alpha|\leq q$, and for $\alpha=\beta+1_j=(\beta_1,\cdots, \beta_j+1,\cdots ,\beta_n)$ for the functions $\delta D^{\alpha}_xv^{\rho,l,0}_i= D^{\alpha}_x\delta v^{\rho,l,0}_i=D^{\alpha}_xv^{\rho,l,0}_i(\tau,x)-D^{\alpha}_xv^{\rho,l,0}_i(l-1,x)$ have the representation
\begin{equation}\label{deltavrholrep}
\begin{array}{ll}
\delta D^{\alpha}_xv^{\rho,l,0}_i(\tau,x)=D^{\alpha}_xv^{\rho,l,0}_i(\tau,x)-D^{\alpha}_xv^{\rho,l-1}_i(l-1,x)\\
\\
=\int_{{\mathbb R}^n}D^{\alpha}_xv^{\rho,l-1}_i(l-1,y)p^l(\tau ,x-y)dy-D^{\alpha}_xv^{\rho,l-1}_i(l-1,x)\\
\\
-\rho\int_{l-1}^{\tau}\int_{\mathbb R}^n\sum_{j=1}^n D^{\beta}_x\left( v^{\rho,l-1}_j(l-1,y)\frac{\partial v^{\rho,l-1}_i}{\partial x_j}(l-1,y)\right)\times\\
\\
\times p^l_{,j}(\tau-s,x-y)dy ds\\
\\
+\rho \int_{l-1}^{\tau}\int_{\mathbb R}^n\sum_{j,p=1}^n\int_{{\mathbb R}^n}\left( \frac{\partial}{\partial x_i}K_n(z-y)\right) \times\\
\\
\sum_{j,p=1}^nD^{\beta}_x \left( \frac{\partial v^{\rho,l-1}_p(l-1,.)}{\partial x_j}\frac{\partial v^{\rho,l-1}_j(l-1,.)}{\partial x_p}\right) (s,y)p^l_{,j}(\tau-s,x-z)dydz ds,
\end{array}
\end{equation}
where we use the convolution rule, and where $p^l$ is the fundamental solution of
\begin{equation}
\frac{\partial v^{\rho,l,0}_i}{\partial \tau}-\rho\nu\sum_{j=1}^n \frac{\partial^2 v^{\rho,l,0}_i}{\partial x_j^2}=0
\end{equation}
on $[l-1,l]\times {\mathbb R}^n$ for all given $y\in {\mathbb R}^n$ (a family of Cauchy problems with $\delta_y(x)=\delta(x-y)$ Dirac delta distribution data). The natural choice of estimation is the standard Young inequality along with the splitting of integrals into integrals on a unit ball and its complement. First we observe that for the characteristic function $1_{B_1(0)}$ of the ball of radius $1$ around zero in ${\mathbb R}^n$ we have for $\tau>0$
\begin{equation}
{\big |}1_{B_1(0)}p_{,i}(\tau,.){\big |}_{L^1}\leq C_{pB_1}<\infty,
\end{equation}
for some constant $C_{pB_1}$ and
\begin{equation}
{\big |}1_{{\mathbb R}^n\setminus B_1(0)}p_{,i}(\tau,.){\big |}_{ L^2\cap L^1}\leq C_{pB^c_1} <\infty
\end{equation}
for some constant $C_{pB^c_1}$. Furthermore, let $C_p=C_{pB_1}+C_{pB^c_1}$.
Furthermore, concerning the Laplacian kernel in the Leray projection term we have for $n\geq 3$
\begin{equation}
1_{B_1(0)}K_{n,i}(.)\in L^1~\mbox{ and }~1_{{\mathbb R}^n\setminus B_1(0) }K_{n,i}(.)\in L^2,
\end{equation}
with the corresponding upper bounds $C_{KB_1}$ and $C_{KB^c_1}$. Recall that the subscript $,i$ denotes the spatial derivative of first order with respect to the $i$th variable. Furthermore, the characteristic function $1_{ {\mathbb R}^n\setminus B_1(0) }$ equals $1$ on the complement of the $n$-dimensional unit ball $B_1(0)$ around zero in ${\mathbb R}^n$, and equals zero elsewhere.  We use the natural abbreviation $C_K=C_{KB_1}+C_{KB^c_1}$.
Next we estimate the two nonlinear terms. Using Young's inequality for the convection term we have for fixed $\tau>0$ 
\begin{equation}\label{oneest}
\begin{array}{ll}
{\big |}\rho\int_{l-1}^{\tau}\int_{\mathbb R}^n\sum_{j=1}^n D^{\beta}_x\left( v^{\rho,l-1}_j(l-1,.)\frac{\partial v^{\rho,l-1}_i(l-1,y)}{\partial x_j}(l-1,y)\right)\times\\
\\
\times \left(1_{B_1(0)} p^l_{,j}(\tau-s,.-y)+ 1_{{\mathbb R}^n\setminus B_1(0)} p^l_{,j}(\tau-s,.
.-y)\right) dyds{\Big |}_{L^2}\\
\\
\leq \rho C_{pB_1}\int_{l-1}^{\tau}{\big |}\sum_{j=1}^n D^{\beta}_x\left( v^{\rho,l-1}_j(l-1,.)\frac{\partial v^{\rho,l-1}_i(l-1,.)}{\partial x_j}(l-1,.)\right){\big |}_{L^2}ds\\
\\
+\rho C_{pB^c_1}\int_{l-1}^{\tau}{\big |}\sum_{j=1}^n D^{\beta}_x\left( v^{\rho,l-1}_j(l-1,.)\frac{\partial v^{\rho,l-1}_i(l-1,.)}{\partial x_j}(l-1,.)\right){\big |}_{L^1}ds.
\end{array}
\end{equation}
Expanding the derivatives of order $|\beta|\leq q-1$ of the product we get at most $2^{q }$ terms. Using $ab\leq \frac{1}{2}a^2+\frac{1}{2}b^2$ we observe that the last summand on the right side of (\ref{oneest}) has the upper bound
\begin{equation}
\begin{array}{ll}
 \rho C_{pB^c_1}\int_{l-1}^{\tau}{\big |}n 2^{q}\max_{0\leq |\alpha|\leq q}\left( D^{\alpha}_xv^{\rho,l-1}_j(l-1,.)(l-1,.)\right)^2{\big |}_{L^1}ds\\
 \\\leq \rho C_{pB^c_1} 2^{q}nC^2 \mbox{ for }\tau\in [l-1,l].
 \end{array}
\end{equation}
For the first summand on the right side of (\ref{oneest}) we could use Fourier transforms in order to transform a convolution into a product and then use weighted $L^2$-estimates for products. We did that elsewhere. Here we remark that there is a simple alternative: since $D^{\alpha}_xv^{\rho,l-1}_j(l-1,.)\in C^q\cap H^q$ we have
$\sup_{x\in {\mathbb R}^n}
{\big |} D^{\alpha}_xv^{\rho,l-1}_j(l-1,x){\big |}\leq C'~\mbox{ for }~0\leq |\alpha|\leq q$ where we may call this $C'$ again $C$ for generic $C\geq 1$ by inductive assumption, and hence we have the upper bound
 \begin{equation}
\begin{array}{ll}
 \rho C_{pB_1}C\int_{l-1}^{\tau}{\big |}n 2^{q}\max_{0\leq |\alpha|\leq q}D^{\alpha} v^{\rho,l-1}_j(l-1,.)(l-1,.){\big |}_{L^2}ds\\
 \\\leq \rho C_{pB_1} 2^{q}nC^2 \mbox{ for }\tau\in [l-1,l].
 \end{array}
\end{equation}
Summing up, we have
\begin{equation}
\begin{array}{ll}
\sup_{\tau\in [l-1,l]}{\big |}\rho\int_{l-1}^{\tau}\int_{\mathbb R}^n\sum_{j=1}^n D^{\beta}_x\left( v^{\rho,l-1}_j(l-1,.)\frac{\partial v^{\rho,l-1}_i(l-1,y)}{\partial x_j}(l-1,y)\right)\times\\
\\
\times p^l_{,j}(\tau-s,.-y) dyds{\big |}_{L^2}\leq \rho C_{p} 2^{q}nC^2.
\end{array}
\end{equation}
A similar estimate holds for the Leray projection term with the two differences that we get an additional constant $C_K$ from the additional convolution involving a partial first order derivative of the Laplacian kernel $K_{n,i}$ (partial derivative with respect to $x_i$), and that we have $n^2$ terms.  
\begin{equation}
\begin{array}{ll}
 \rho \int_{l-1}^{\tau}{\big |}\int_{{\mathbb R}^n}\sum_{j,p=1}^n\int_{{\mathbb R}^n}\left( \frac{\partial}{\partial x_i}K_n(z-y)\right) \times \\
\\
\times D^{\beta}_x\left( \sum_{j,p=1}^n\left( \frac{\partial v^{\rho,l,l-1}_p(l-1,.)}{\partial x_j}\frac{\partial v^{\rho,l-1}_j(l-1,.)}{\partial x_p}\right)\right)  (s,y)p^l_{,j}(\tau-s,x-z)dydz ds{\big |}_{L^2}ds\\
 \\\leq \rho C_KC_{p} 2^{q}n^2C^2 \mbox{ for }\tau\in [l-1,l].
 \end{array}
\end{equation}
In order to have
\begin{equation}
{\big |}\delta v^{\rho,l,0}_i(\tau,.){\big |}_{C^q\cap H^q}\leq \frac{1}{4}
\end{equation}
from the estimate the last two summands via the representation of $\delta D^{\alpha}_xv^{\rho,l,0}_i(\tau,x)$ in (\ref{deltavrholrep}) we observe that the choice
\begin{equation}
\rho \leq \frac{1}{8C_KC_{p} 2^{q+1}n^{q+1}n^2C^2 }
\end{equation}
is sufficient in order to obtain a rough upper bound $\frac{1}{8}$, where we assume w.l.o.g. that $C_K\geq 1$, and where we are rough saying (among other things) that there are less than $n^{q+1}$ terms in the norm of order $q$ -note that we want a short notation of a constant which still bears some information.  
Next concerning the two linear terms on the right side of (\ref{deltavrholrep})  for some small $0<\rho \leq \frac{1}{8C_KC_{p} 2^{q+1}n^{q+1}n^2C^2}$ we know that
\begin{equation}
{\big |}\int_{{\mathbb R}^n}v^{\rho,l-1}_i(l-1,y)p^l(\tau ,x-y)dy-D^{\alpha}_xv^{\rho,l-1}_i(l-1,.){\big |}_{C^q\cap H^q}\leq \frac{1}{8}
\end{equation}
for all $\tau\in [l-1,l]$ such that the first approximation increment indeed satisfies ${\big |}\delta v^{\rho,l,0}_i(\tau,.){\big |}_{C^q\cap H^q}\leq \frac{1}{4}$ as stated in the theorem.
Next we prove the contraction result. First  note that for $m\geq 2$ the increment $\delta v^{\rho,l,m}_i$ satisfies 
\begin{equation}\label{Navleray6increment}
\left\lbrace \begin{array}{ll}
\frac{\partial \delta v^{\rho,l,m}_i}{\partial \sigma}-\rho\nu\sum_{j=1}^n \frac{\partial^2 \delta v^{\rho,l,m}_i}{\partial x_j^2}\\
\\
+\rho\sum_{j=1}^n \delta v^{\rho,l,m-1}_j\frac{\partial \delta v^{\rho,l,m-1}_i}{\partial x_j}+\rho\sum_{j=1}^n v^{\rho,l,m-2}_j\frac{\partial \delta v^{\rho,l,m-1}_i}{\partial x_j}\\
\\
=\rho \sum_{j,p=1}^n\int_{{\mathbb R}^n}\left( \frac{\partial}{\partial x_i}K_n(x-y)\right) \sum_{j,p=1}^n\left( \frac{\partial \delta v^{\rho,l,m-1}_p}{\partial x_j}\frac{\partial v^{\rho,l,m-1}_j}{\partial x_p}\right) (\tau,y)dy\\
\\
+ \rho\sum_{j,p=1}^n\int_{{\mathbb R}^n}\left( \frac{\partial}{\partial x_i}K_n(x-y)\right) \sum_{j,p=1}^n\left( \frac{\partial v^{\rho,l,m-2}_p}{\partial x_j}\frac{\partial \delta v^{\rho,l,m-1}_j}{\partial x_p}\right) (\tau,y)dy,\\
\\
\delta \mathbf{v}^{\rho,l,m}(l-1,.)=0.
\end{array}\right.
\end{equation}
We have to estimate convolutions of the fundamental solutions with the convection term increments and the Leray projection term increments in (\ref{Navleray6increment}). As an essential term we estimate a Leray projection term of the form
\begin{equation}\label{lerayterm}
\begin{array}{ll}
\rho \int_{l-1}^{\tau}\int_{{\mathbb R}^n}\sum_{j,p=1}^n\int_{{\mathbb R}^n}\left( \frac{\partial}{\partial x_i}K_n(z-y)\right) \sum_{j,p=1}^n\left( \frac{\partial \delta v^{\rho,l,m-1}_p}{\partial x_j}\frac{\partial v^{\rho,l,m-1}_j}{\partial x_p}\right) (s,z)\times \\
\\
\times p^l(\tau-s,x-z)dydzds
\end{array}
\end{equation}
We may use $p^l(\tau,.)\in L^1\cap L^2$ and $1_{B_1(0)}K_{,i}\in L^1$ and $1_{{\mathbb R}^n\setminus B_1(0)}K_{,i}\in L^2$ along with the characteristic function $1_{B_1(0) }$ of the unit ball around zero. Furthermore $p^l_{,j}\in L^1$ is useful if we consider the highest order of derivatives. Note also that for the first order spatial derivatives if the fundamental solution we have  $1_{B_1(0)}p^l_{,i}(\tau,.)\in L^1$ with corresponding norm ${\big |}1_{B_1(0)}p^l(\tau,.){\big |}_L^1=C_{pB_1}< \infty$. For the convolution (we use $\star$ as a symbol for the corresponding operation) with the first derivative of the Laplacian kernel we have for for multiindices $0\leq |\beta|\leq q-1$
\begin{equation}
\begin{array}{ll}
{\Big |}1_{B_1(0)}K_{n,i}\star\sum_{j,p=1}^nD^{\beta}_x\left( \frac{\partial \delta v^{\rho,l,m-1}_p}{\partial x_j}\frac{\partial v^{\rho,l,m-1}_j}{\partial x_p}\right)\leq {\Big | }_{L^2}\\
\\
\leq C_{KB_1}{\Big |}\sum_{j,p=1}^n D^{\beta}_x\left( \frac{\partial \delta v^{\rho,l,m-1}_p}{\partial x_j}\frac{\partial v^{\rho,l,m-1}_j}{\partial x_p}\right){\Big |}_{L^2}\leq C< \infty
\end{array}
\end{equation}
for some constant $C>0$ such that the cut-off integral summand with $1_{B_1(0)}K_{n,i}$ can be estimated with $p^l(\tau,.)\in L^1$ (up to order of regularity $q-1$) and $p^l_{,i}(\tau,.)\in L^1$ (for order of regularity $q$) and another application of the Young inequality. 
We may use standard product estimates in Sobolev norms in order to extract increments (we did that elswhere), but we have bounded suprema and may use a simpler estimate here).
Using $ab\leq \frac{1}{2}a^2+\frac{1}{2}b^2$, the suprema upper bound for $\sup_{x\in {\mathbb R}^n}{\big |}D^{\alpha} v^{\rho,l,m}_i(\tau,x){\big |}\leq C$ for $0\leq |\alpha|\leq q$  (which exist since $v^{\rho,l,m}_i(\tau,.)\in C^q\cap H^q$ implies $D^{\alpha} v^{\rho,l,m}_i(\tau,.)\in C_0$ for $0\leq |\alpha|\leq q$)
\begin{equation}
 {\Big |}D^{\beta}_x\left( \frac{\partial \delta v^{\rho,l,m-1}_p}{\partial x_j}\frac{\partial v^{\rho,l,m-1}_j}{\partial x_p}\right)(\tau,y){\Big |}\leq n^22^{q}C^2\max_{0\leq |\beta|\leq q}{\Big|}D^{\alpha}_x\delta v^{\rho,l,m-1}(\tau,y){\Big |}
\end{equation}
Similar estimates can be done outside $B_1(0)$ such that, as there are less than $\sum_{p=0}^{q} n^p\leq n^{q+1}$ terms in the definition of a $|.|_{C^q\cap H^q}$  for fixed $\tau \in [l-1,l]$ the $|.|_{C^q\cap H^q}$-norm of (\ref{lerayterm}) has the upper bound
\begin{equation}
\rho C_{p}C_K C^22^{q}n^{q+3}{\Big|}\delta v^{\rho,l,m-1}(\tau,.){\big |}_{C^q\cap H^{q}}
\end{equation}
Similar estimates for the other Leray projection term (estimated with increment index $q-1$) and the other two convection terms
lead to the conclusion that there is a $\rho$ with
\begin{equation}
\rho \frac{1}{8C_{p}C_K C^22^{q}n^{q+3}}
\end{equation}
we get a contraction constant $\frac{1}{2}$ inductively.

Finally we conclude that the limit $v^{\rho,l}_i=\lim_{m\uparrow \infty}v^{\rho,l,m}_i,~1\leq i\leq n$ is a local regular solution of the Navier Stokes equation in (\ref{Navleray2}). First note that for all $m\geq 1$ the function $v^{\rho,l,m}_i,~1\leq i\leq n$ satisfy the equation in (\ref{Navleray4}), where the initial data $v^{\rho,l,m}_i(l-1,.)=v^{\rho,l-1}_i(l-1,.)\in C^q\cap H^q$ for $q\geq 2$, and the coefficients $v^{\rho,l,m-1}_i(\tau,.),~1\leq i\leq n$ are in $C^q\cap H^q$ by inductive assumption uniformly for all time $\tau$. Furthemore, $v^{\rho,l,m-1}_i,~1\leq i\leq n$ has a continuous time derivative inductively. Moreover the source term of Leray projection form is a first order spatial derivative of the solution of a Poisson equation with source in $C^{q-1}\cap H^{q-1}$ and similar estimates as above involving splitting of the Laplacian kernel show that it is itself for fixed $\tau$ in $C^{q-1}\cap H^{q-1}$ uniformly with respect to $\tau\in [l-1,l]$. It follows by classical representation of the function $v^{\rho,l,m}_i,~1\leq i\leq n$ that  $v^{\rho,l,m}_i\in C^{1,2}$ for all $1\leq i\leq n$ for all $m\geq 1$, and, moreover, that $\tau \rightarrow \left(x\rightarrow v^{\rho,l,m}_i(\tau,x) \right)\in C^0\left([l-1,l],C^q\cap H^q \right)$ for all $1\leq i\leq n$ and $m\geq 1$ inductively. Plugging in these approximative solutions of linear equations into the original local equation (\ref{Navleray2}) and using (\ref{Navleray2}) we get
\begin{equation}\label{Navleraylocex2}
\begin{array}{ll}
\frac{\partial v^{\rho,l,m}_i}{\partial \tau}-\rho\nu\sum_{j=1}^n \frac{\partial^2 v^{\rho,l,m}_i}{\partial x_j^2}
+\rho\sum_{j=1}^n v^{\rho,l,m}_j\frac{\partial v^{\rho,l,m}_i}{\partial x_j}-\\
\\ \rho\sum_{j,p=1}^n\int_{{\mathbb R}^n}\left( \frac{\partial}{\partial x_i}K_n(x-y)\right) \sum_{j,p=1}^n\left( \frac{\partial v^{\rho,l,m}_p}{\partial x_j}\frac{\partial v^{\rho,l,m}_j}{\partial x_p}\right) (\tau,y)dy\\
\\
+\rho\sum_{j=1}^n v^{\rho,l,m}_j\frac{\partial \delta v^{\rho,l,m-1}_i}{\partial x_j}+\rho\sum_{j=1}^n \delta v^{\rho,l,m}_j\frac{\partial v^{\rho,l,m-1}_i}{\partial x_j}\\
\\+ \rho\sum_{j,p=1}^n\int_{{\mathbb R}^n}\left( \frac{\partial}{\partial x_i}K_n(x-y)\right) \sum_{j,p=1}^n\left( \frac{\partial v^{\rho,l,m}_p}{\partial x_j}\frac{\partial \delta v^{\rho,l,m}_j}{\partial x_p}\right) (\tau,y)dy\\
\\
+\rho\sum_{j,p=1}^n\int_{{\mathbb R}^n}\left( \frac{\partial}{\partial x_i}K_n(x-y)\right) \sum_{j,p=1}^n\left( \frac{\partial \delta v^{\rho,l,m}_p}{\partial x_j}\frac{\partial v^{\rho,l,m-1}_j}{\partial x_p}\right) (\tau,y)dy.
\end{array}
\end{equation}
From contraction we know that the increment ${\big |}\delta v^{\rho,l,m}_i(\tau,.) {\big |}_{C^q\cap H^q}\downarrow 0$ 
for $q\geq 2$ and all $\tau\in [l-1,l]$. We can use 
$\sup_{x\in {\mathbb R}^n}{\big |}D^{\alpha}_xv^{\rho,l,m}_i(\tau,x){\big |}\leq C$ for some generic constant $C>0$ independent of the time sub-step number $m$ (this independence follows from the contraction) to conclude
\begin{equation}
{\Big |} \rho\sum_{j,p=1}^n\int_{{\mathbb R}^n}\left( \frac{\partial}{\partial x_i}K_n(x-y)\right) \sum_{j,p=1}^n\left( \frac{\partial v^{\rho,l,m}_p}{\partial x_j}\frac{\partial \delta v^{\rho,l,m}_j}{\partial x_p}\right) (\tau,y)dy{\Big |}\downarrow 0
\end{equation}
as $m\uparrow \infty$ (which is simpler then the alternative standard estimates for Sobolev norms of products. For the Leray projection term we first estimate the convolution where we split again the first derivatives of the Laplacian kernel $K_{n,i}=1_{B_1(0)}K_{n,i}+1_{{\mathbb R}^n\setminus B_1(0)}K_{n,i}$ 'eliminating' the latter by two applications of Young's inequality and estimate the appearing upper bound similar as the convection term product. Hence we get
\begin{equation}
\begin{array}{ll}
\lim_{m\uparrow \infty}\rho{\Big |}\sum_{j,p=1}^n\int_{{\mathbb R}^n}\left( \frac{\partial}{\partial x_i}K_n(x-y)\right) \sum_{j,p=1}^n\left( \frac{\partial v^{\rho,l,m}_p}{\partial x_j}\frac{\partial \delta v^{\rho,l,m}_j}{\partial x_p}\right) (\tau,y)dy{\Big |}+\\
\\
\lim_{m\uparrow \infty}{\Big |}\rho\sum_{j,p=1}^n\int_{{\mathbb R}^n}\left( \frac{\partial}{\partial x_i}K_n(x-y)\right) \sum_{j,p=1}^n\left( \frac{\partial \delta v^{\rho,l,m}_p}{\partial x_j}\frac{\partial v^{\rho,l,m-1}_j}{\partial x_p}\right) (\tau,y)dy{\Big |}\\
\\
\downarrow 0 \mbox{ as }~m\uparrow \infty,
\end{array}
\end{equation}
and $v^{\rho,l}_i,~1\leq i\leq n$ is indeed a local regular solution as stated.
\end{proof}

The local information obtained for $v^{\rho,l}_i,~1\leq i\leq n,$ is inherited by all functions $u^{\rho,l,m}_i,~1\leq i\leq n$ for all $m\geq 1$ simultaneously. For this reason we do not have to analyze the regularity of the latter functions. On the other hand, it is clear that local contraction results hold a fortiori for the functions $u^{\rho,l}_i,~1\leq i\leq n$, and for the functions $u^{\rho,l,m}_i,~1\leq i\leq n$, as we have an additional damping term. The only purpose of the latter functions is to estimate the growth of the component functions $v^{\rho,l}_i$ for all $1\leq i\leq n$ at one time step $l$. The growth control is inherited for all time steps $l\geq 1$ and this leads to global upper bounds. If we determine the time step size via the local contraction result theorem \ref{loccontr} of the original equation (proving only the existence of a linear regular upper bound), then it is clear that the scheme cannot be finite. Nevertheless, even this leads to a global regular existence proof on an analytical level - although not to a numerical scheme. For a numerical scheme we may work with the comparison function $u^{\rho,l}_i,~1\leq i\leq n$ for a localized transformation as in (\ref{uvl0}) with appropriate parameter $\lambda$. We come back to this issue.  From (\ref{uvtrans}) we have
\begin{equation}\label{uv00}
u^{\rho,l,m}_i(\sigma,x)=\frac{v^{\rho,l}_i(\tau(\sigma) ,x)}{1+\tau(\sigma)}~\mbox{ for }\sigma\in [m-1,m].
\end{equation}
Now assume that at time step $l\geq 1$ we have a certain time step size $\rho>0$ such that the contraction result above holds for $v^{\rho,l}_i,~1\leq i\leq n$ on the next unit interval in transformed time coordinates $\tau$. If we construct a global linear upper bound, then we expect an appropriate time step size to be of order $\rho_l\sim \frac{1}{l}$ for the scheme $u^{\rho,l,m}_i,~1\leq i\leq n,~l,m\geq 1$ on the global time scale (original time $t$), and for the alternative transformation in (\ref{alt}) we expect that this time step size can be chosen uniformly, i.e., independently of the time step number $l$ (for the appropriate choices of $\lambda$ and $\rho$; these two choices are not independent, cf. below). However it will be shown below that even the weak (rough) local contraction result above can be applied anyway. 
Next we observe that if for a time step number $l\geq 1$ for a given order of regularity $q\geq 2$ and time step size $\rho>0$ the contraction result in theorem \ref{loccontr} holds (and note again that this means that we work on an analytical level, not on a constructive level then), and we have for all $1\leq i\leq n$, and all substeps $m\geq 1$ the preservation of an upper bound in the sense that for all $0\leq |\alpha|\leq q$ the implication
\begin{equation}
{\big |}u^{\rho,l,m-1}_i(m-1,.){\big |}_{C^q\cap H^q}\leq C\Rightarrow {\big |}u^{\rho,l,m}_i(m,.){\big |}_{C^q\cap H^q}\leq C
\end{equation}
holds, then this implies that
\begin{equation}\label{uv00}
{\big |}v^{\rho,l}_i(l ,.){\big |}_{C^q\cap H^q}\leq  \left(1+l \right) \sup_{m\in {\mathbb N}}{\big |}u^{\rho,l,m}_i(m,.){\big |}_{C^q\cap H^q}\leq (1+l)C.
\end{equation}
This conclusion is immediate from the relation of the schemes. At this point it may be helpful to remark that we do not have any minimal stepsize $\rho>0$ which would turn the description into a global scheme. It is essentially an argument by contradiction: if we have a maximal time $t_{max}$ with
\begin{equation}\label{uv01}
{\big |}v_i(t_{max} ,.){\big |}_{C^q\cap H^q}\leq  \left(1+t_{max} \right) C,
\end{equation}
for some $q\geq 2$ in original coordinates then we find via the subscheme described a $t'>t_{\max}$ such that
\begin{equation}\label{uv01}
{\big |}v_i(t' ,.){\big |}_{C^q\cap H^q}\leq  \left(1+t_{max} \right) C,
\end{equation}
still holds, and this is a contradiction to the maximality of $t_{max}$. Therefore, we always refer to the scheme of the comparison function $u^{\rho,l}_i$ if we talk about step sizes for numerical schemes.
For the localized scheme with the prescription in (\ref{uvl0}) along with appropriate choices of the parameter $\lambda>0$ and the time step size $\rho$ we can even show that an upper bound $C>0$ is preserved which is independent of the time step size, i.e., 
\begin{equation}\label{uv001}
{\big |}v^{\rho,l}_i(l ,.){\big |}_{C^q\cap H^q}\leq   \sup_{m\in {\mathbb N}}{\big |}u^{\rho,l,m}_i(m,.){\big |}_{C^q\cap H^q}\leq C
\end{equation}
holds.
As (\ref{uv00}) holds for all time step numbers $l\geq 1$ we can conclude by contradiction of the maximality of time where a linear upper bound may hold (using the scheme in (\ref{alt})) that we have a global linear upper bound for the value function (where we discuss the matter of time step size in correlation to the time step number $l\geq 1$ below). 
The knowledge of the local solution $v^{\rho,l}_i,~1\leq i\leq n$ implies that we know the solution functions $u^{\rho,l,m}_i,~1\leq i\leq n$ at time step $l\geq 1$ for all substeps $m\geq 1$ simultaneously, i.e., we do not have to set up a subsequence of iteration schemes to construct them. The following considerations concerning the fundamental solution hold for the scheme based on the transformation (\ref{uvtrans}) and for the scheme with a global uniform upper bound based on the transformation (\ref{uvl0}) as well. They also hold for (\ref{alt}). In order to cover all cases we introduce the notation
\begin{equation}
\mu'=\mu^{\tau,1},
\end{equation}
keeping in mind that with a similar $\mu'$ related to transformations (\ref{uvl0}) and (\ref{alt}) analogous considerations concerning the fundamental solutions hold. 

We can represent these solutions in terms of fundamental solutions $p^{l,m}$ of the equation
\begin{equation}\label{plm}
\frac{\partial p^{l,m}}{\partial \sigma}-\rho\nu \mu' \Delta p^{l,m}=0,
\end{equation}
where the index $m$ reminds us that $p^{l,m}$ is defined on the domain $[m-1,m]\times {\mathbb R}^n$ (for each parameter $y\in {\mathbb R}^n$ and $s\in (m-1,m]$ with $s<\sigma$). As $\rho$ is a constant time step size and $\mu'$ depends only on time $\sigma$, these fundamental solutions $p^{l,m}$ have a special Levy expansion. We may fix the parameter $s>0$ first  
and denote the lowest order Gaussian approximation of the Levy expansion by $G_{\mu'(s)}$, where
\begin{equation}
G_{\mu'(s)}(\sigma,x;s,y)=\frac{1}{\sqrt{4\pi \rho\nu\mu'(s)(\sigma-s)}^n}\exp\left(-\frac{|x-y|^2}{4\rho\nu\mu'(s)(\sigma-s)} \right) 
\end{equation}
for all $x,y\in {\mathbb R}^n$ and $s<\sigma$ of the corresponding domain. Again we know where we are from the indices of the local value function $u^{\rho,l,m}_i,~1\leq i\leq n$, and do not equip this Gaussian and the higher order terms with superscript indices. The classical Levy expansion of $p^{l,m}$ is then given by
\begin{equation}\label{fund}
p^{l,m}(\sigma,x;s,y)=G_{\mu'(s)}(\sigma,x;s,y)+\int_{s}^{\sigma}G_{\mu'(s)}(\sigma,x;\zeta,z)\phi(\zeta ,z;s,y)dzd\zeta,
\end{equation}
where
\begin{equation}\label{phi}
\phi(\sigma ,x;s,y)=\sum_{p\geq 1}\rho^p(L_{\mu'}G_{\mu'(s)})_p(\sigma ,x;s,y),
\end{equation}
along with the recursion
\begin{equation}
L_{\mu'}G_{\mu'(s)}(\sigma ,x;s,y)=(L_{\mu'}G_{\mu'(s)})_1(\sigma ,x;s,y)=\left(\mu'(\sigma)-\mu'(s) \right)\Delta G_{\mu'(s)}(\sigma ,x;s,y),
\end{equation}
and
\begin{equation}
\begin{array}{ll}
(L_{\mu'}G_{\mu'(s)})_{p+1}(\sigma ,x;s,y)=\\
\\
 \int_{s}^{\sigma}\int_{{\mathbb R}^n}\left(\mu'(\sigma)-\mu'(s) \right)\Delta G_{\mu'(s)}(\sigma ,x;\zeta,z)(LG_{\mu'(s)})_{p+1}(\zeta ,z;s,y).
\end{array}
\end{equation}
Note that we have extracted the powers of $\rho$ from the linear operator and added a subscript $\mu'$ at $L$ to indicate this. Since $0<\rho< 1$ (we may even have $0<\rho \ll 1$) the representation of the fundamental solution $p^{l,m}$ in (\ref{fund}) and (\ref{phi}) shows that the Levy expansion converges if $(L_{\mu'}G_{\mu'(s)})_p$ have a uniform upper bound in form of a Gaussian times a constant. This is more then expected since the Laplacian of the Gaussian is involved, but $\mu'$ is only time dependent and Lipschitz and this facilitates the estimate. Consider $L_{\mu'}$. 
We compute
\begin{equation}\label{lmu}
\begin{array}{ll}
L_{\mu'(s)}(\sigma ,x;s,y)=\left(\mu'(\sigma)-\mu'(s) \right)\Delta G_{\mu'(s)}(\sigma ,x;s,y)\\
\\
=\left(\mu'(\sigma)-\mu'(s) \right)\Delta \frac{1}{\sqrt{4\pi \rho\mu'(s)(\sigma-s)}^n}\exp\left(-\frac{|x-y|^2}{4\rho\mu'(s)(\sigma-s)} \right) \\
\\
=\left(\mu'(\sigma)-\mu'(s) \right)\sum_j\left[ \left(\frac{2(x_j-y_j)}{4\mu'(s)(\sigma-s)} \right)^2-  \frac{2}{4\mu'(s)(\sigma-s)} \right] \frac{1}{\sqrt{4\pi \rho\mu'(s)(\sigma-s)}^n}\times\\
\\
\times \exp\left(-\frac{|x-y|^2}{4\rho\mu'(s)(\sigma-s)} \right).
\end{array}
\end{equation}
Since we have $|\mu'(\sigma)-\mu'(s)|\leq c|\sigma-s|$ from (\ref{lmu}) we get
\begin{equation}\label{lmu2}
\begin{array}{ll}
{\big |}L_{\mu'(s)}G_{\mu'(s)}(\sigma ,x;s,y){\big |}\leq \left( {\big |} \sum_j\frac{4c(x_j-y_j)^2}{16\mu'(s)(\sigma-s)}{\big |} + \frac{c}{2\mu'(s)} \right)\times\\
\\
 \times \frac{1}{\sqrt{4\pi\rho \mu'(s)(\sigma-s)}^n}\exp\left(-\frac{|x-y|^2}{4\rho \mu'(s)(\sigma-s)} \right).
\end{array}
\end{equation}
Writing 
\begin{equation}
\begin{array}{ll}
\frac{1}{\sqrt{4\pi \rho\mu'(s)(\sigma-s)}^n}\exp\left(-\frac{|x-y|^2}{4\rho\mu'(s)(\sigma-s)} \right)\\
\\
 =\exp\left(-\frac{|x-y|^2}{8\mu'(s)(\sigma-s)} \right) \frac{\sqrt{2}}{\sqrt{8\pi \rho\mu'(s)(\sigma-s)}^n}\exp\left(-\frac{|x-y|^2}{8\rho\mu'(s)(\sigma-s)} \right),
\end{array}
\end{equation}
and with
\begin{equation}
C^0:=\sup_{z\in {\mathbb R}^n}2|z|^2\exp\left(-z^2 \right)
\end{equation}
we get
\begin{equation}\label{lmu2}
\begin{array}{ll}
{\big |}L_{\mu'(s)}G_{\mu'(s)}(\sigma ,x;s,y){\big |}\leq \left(C^0  + \frac{c}{2\mu'(s)} \right) \frac{\sqrt{2}}{\sqrt{8\pi \rho\mu'(s)(\sigma-s)}^n}\exp\left(-\frac{|x-y|^2}{8\rho\mu'(s)(\sigma-s)} \right).
\end{array}
\end{equation}
We are rough with the second term in the bracket of the right side but since we are concerned with the local analysis at finite time at this moment this is o.k.. Furthermore, the solution to a possible problem at infinite time may not be solvable by a better estimate. Using a Taylor expansion of $\mu(\sigma)$ we get a slight improvement of a constant factor
\begin{equation}\label{betterconst}
\left(C^1  + \frac{c}{2\mu'(s)^{\frac{2}{3}}} \right) \mbox{ instead of }\left(C^0 + \frac{c}{2(\mu'(s))} \right) 
\end{equation} 
with another constant $C^1$, but this is still singular at infinity. So if we have to deal with infinite time in transformed coordinates (which are always finite times in original coordinates), then we should be careful. Do we have to? Well, not really.
Consider the equation for $u^{\rho,l,m}_i,~1\leq i\leq n$ above, which we know at a given time step $l\geq 1$ since we have constructed the local solution $v^{\rho,l}_i,¸1\leq i\leq n$. As we know this function at time step $l$ we have the representation

\begin{equation}\label{ulmrep}
 \begin{array}{ll}
 u^{\rho,l,m}_i(\sigma,x)-\int_{{\mathbb R}^n}u^{\rho,l,m}_i(m-1,y)p^{l,m}(\sigma,x;m-1,y)dy=\\
 \\
-\rho\int_{m-1}^{\sigma}\int_{{\mathbb R}^n}\mu^{\tau,2}(s)\sum_{j=1}^n u^{\rho,l,m}_j\frac{\partial u^{\rho,l,m}_i}{\partial x_j}(s,y)p^{l,m}(\sigma,x;s,y)dyds\\
\\
-\int_{m-1}^{\sigma}\int_{{\mathbb R}^n}\mu(s) u^{\rho,l,m}_i(s,y)p^{l,m}(\sigma,x;s,y)dyds\\
\\ 
+\rho \int_{m-1}^{\sigma}\int_{{\mathbb R}^n} \mu^{\tau,2}(s)\sum_{j,p=1}^n\int_{{\mathbb R}^n}\left( \frac{\partial}{\partial x_i}K_n(z-y)\right)\times\\
\\
\times \sum_{j,p=1}^n\left( \frac{\partial u^{\rho,l,m}_p}{\partial x_j}\frac{\partial u^{\rho,l,m}_j}{\partial x_p}\right)(s,y)p^{l,m}(\sigma,x;s,z)dy dz ds.
\end{array}
\end{equation} 
All the terms on the right side of (\ref{ulmrep}) have a factor $\mu(s)$ (which occurs also in $\mu^{\tau,2}$), which cancels with the estimate in (\ref{lmu2}) and a fortiori with the improved constant in (\ref{betterconst}). Hence there is indeed no problem as $m\uparrow \infty$ as we get uniform upper bounds for all substep numbers $m\geq 1$ via the representation in (\ref{ulmrep}). The second term on the left side of (\ref{ulmrep}) has to be compared with the initial data. At first glance it looks as this difference is H\"{o}lder in time for some H\"{o}lder constant between zero and one, but we shall observe that the behavior is indeed better (although even a worse H\"{o}lder growth would suffice for our purposes, as it suffices to observe the growth at discrete times).  

As we study the subscheme having constructed the local solution of the incompressible Navier Stokes equation $v^{\rho,l}_i,"1\leq i\leq n$ at time step $l\geq 1$ we use the relation 
 \begin{equation}\label{uvrel2}
 \frac{1}{1+\tau(\sigma)}v^{\rho,l}_i(\tau(\sigma) ,x)=u^{\rho,l,m}(\sigma,x)
 \end{equation}
 for all $\sigma \in [m-1,m]$ and where $\tau(\sigma)$ is the inverse of $\sigma(\tau)$ in (\ref{timedil}). As we have already remarked, the reason to consider the functional series $u^{\rho,l,m}_i,~1\leq i\leq n$ at all is to observe that the growth of $v^{\rho,l}_i,~1\leq i\leq n,~l\geq 1$ can be controlled appropriately at each time step such that we get a linear global upper bound for $v^{\rho}_i,~1\leq i\leq n$ in the end. 
For the localized scheme on we shall see that we even get an uniform upper bound. We also shall consider variations of (\ref{uvrel2}) and prove global linear upper bounds and global upper bounds with growth of order $\mu >0$ directly.
 Note that from (\ref{uvrel2}) we have for all $l\geq 1$, all $m\geq 1$, and all $x\in {\mathbb R}^n$
\begin{equation}\label{uvrelloc1}
\frac{1}{1+\tau(m-1)}v^{\rho,l}_i(\tau(m-1) ,x)=u^{\rho,l,m}(m-1,x)=u^{\rho,l,m-1}(m-1,x)
\end{equation}
for the initial data of the time-local subscheme problem for $u^{\rho,l,m}_i,~1\leq i\leq n$ and 
\begin{equation}\label{uvrelloc1}
\frac{1}{1+\tau(\sigma)}v^{\rho,l}_i(\tau(\sigma) ,x)=u^{\rho,l,m}(\sigma,x)\mbox{ for } \sigma\in [m-1,m],
\end{equation} 
where we know that $\tau(\sigma)\in [0,1]$ for all $\sigma \in [m-1,m]$ and all $m\geq 1$. Now assume that for $l\geq 1$ we have  
\begin{equation}
v^{\rho,l-1}_i(l-1,.)\in C^q\cap H^q,~\mbox{where } {\big |}v^{\rho,l-1}_i(l-1,.){\big |}_{C^q\cap H^q}\leq C
\end{equation}
for some $C>0$. 
From $\sup_{\tau\in [l-1,l]}{\big |}D^{\alpha}_x\delta v^{r,l,0}_i(\tau,.){\big |}_{C^q\cap H^q}\leq \frac{1}{4}$ for all $0\leq |\alpha|\leq q$ and from the local contraction result theorem \ref{loccontr} we know that for small $\rho>0$ (time step size is discussed at the end of this paper)
\begin{equation}
u^{\rho,l,m}_i(\sigma,.)\in C^q\cap H^q, \mbox{ where }{\big |}u^{\rho,l,m}_i(\sigma,.){\big |}_{C^q\cap H^q}\leq C+\frac{1}{2}
\end{equation}
for all $\sigma \in [m-1,m]$ uniformly for all $m\geq 1$. We conclude that $u^{\rho,l,m}_i,~1\leq i\leq n$ solves (\ref{Navleray3}) with the same argument which we used in order to obtain local regular existence for $v^{\rho,l}_i,~1\leq i\leq n$ from the local contraction result.
Now we are able to get
\begin{itemize}
 \item[i)] 
 a proof of preservation of a global upper bound for
  $$\sup_{m\geq 1}{\big |}u^{\rho,l,m}_i(m,.){\big |}_{C^{q}\cap H^{q}}$$ using the subscheme $u^{\rho,l,m,p}_i$ directly; since we use a global time delay transformation (a factor of the form $\lambda (1+\tau)$), we need to prove that the step size $\rho$ is essentially independent of the time horizon $T>0$ of the problem. We are explicit about this in this update.
 \item[ii)]  a proof of preservation of a global upper bound of the subscheme which leads to a global upper bound for $$\sup_{m\geq 1}{\big |}u^{\rho,l,m}_i(m,.){\big |}_{C^{q}\cap H^{q}}$$  using the local adjoint of a fundamental solution. The adjoint is useful in order to generalise the result to equations with variable viscosity.
 \item[iii)] a constructive scheme with a localized version of the comparison function $u^{*,\rho,t_0}_i,~1\leq i\leq n$, which satisfies for all $1\leq i\leq n$
\begin{equation}
\lambda(1+(\tau-t_0))u^{*,\rho,t_0}_i=v^{\rho,t_0}_i, 
\end{equation} 
and which leads to a uniform upper bound
\begin{equation}
\sup_{t<\infty}{\big |}v_i(t,.){\big |}_{H^m\cap C^m}\leq C
\end{equation}
for a constant $C>0$ which is independent of time. Moreover, this constructive version of a global existence scheme leads to the result stated in Corollary \ref{maincor}.
\end{itemize}
We note that items i) and ii) above lead to growth estimates of $v^{\rho,l}_i,~1\leq i\leq n$ of course.
We discussed local adjoints and their application elsewhere, and consider only the possibilities i) and iii) in the following. Only a short remark is made concerning item ii). Furthermore, note that all these strategies have to be supplemented by different possible (non-constructive and constructive) strategies which are applied in order choose the time step size an which will be considered below. 

Ad i) We set $\lambda =1$ at this item. We may use classical representations of local solutions for comparison functions $u^{\rho,l,m}_i$ with transformed fundamental solutions, where we can use the estimates for the fundamental solution above since the term in these estimates with $\frac{1}{\mu(s)}$ (or $\frac{1}{\mu(s)^{\frac{2}{3}}}$) is cancelled by the coefficient $\mu(s)$ from the source terms - note that $\mu(s)$ occurs in $\mu^{\tau,1}$ and $\mu^{\tau,2}$ and is the relevant factor here) . Classical representations of the solutions $u^{\rho,l,m}_i$ using  first order spatial derivatives of the fundamental solutions for estimates of derivatives up to order $q$ lead to uniform preservation of upper bounds independent of the substep number $m$.
More precisely, for each step numbers $m$ and $\beta+1_j=\alpha$ along with $1\leq |\alpha|\leq q$ we consider the representation
\begin{equation}\label{ulmrep*}
 \begin{array}{ll}
 D^{\alpha}_xu^{\rho,l,m}_i(\sigma,x)-\int_{{\mathbb R}^n}D^{\alpha}_xu^{\rho,l,m}_i(m-1,y)p^{l,m}(\sigma,x;m-1,y)dy=\\
 \\
-\rho\int_{m-1}^{\sigma}\int_{{\mathbb R}^n}\sum_{j=1}^n D^{\beta}_x\left( u^{\rho,l,m}_j\frac{\partial u^{\rho,l,m}_i}{\partial x_j}\right) (s,y)\mu^{\tau,2}(s)p^{l,m}_{,j}(\sigma,x;s,y)dyds\\
\\
-\int_{m-1}^{\sigma}\int_{{\mathbb R}^n} D^{\alpha}_xu^{\rho,l,m}_i(s,y)\mu(s)p^{l,m}(\sigma,x;s,y)dyds\\
\\ 
+\rho \int_{m-1}^{\sigma}\int_{{\mathbb R}^n} \sum_{j,p=1}^n\int_{{\mathbb R}^n}\left( \frac{\partial}{\partial x_i}K_n(z-y)\right)\times\\
\\D^{\beta}_x\left( \frac{\partial u^{\rho,l,m}_p}{\partial x_j}\frac{\partial u^{\rho,l,m}_j}{\partial x_p}\right)(s,y)\mu^{\tau,2}(s)p^{l,m}_{,j}(\sigma,x;s,z)dy dz ds
\end{array}
\end{equation} 
on $[m-1,m]\times {\mathbb R}^n$, and where we put the coefficients of the source terms $\mu^{\tau,k}(s)$ and $\mu$ next to $p^{l,m}_{,i}$. For $\alpha=0$ we can use a similar representation with first order derivatives of the fundamental solution, i.e., with $p^{l,m}_{,i}$ by using partial integration of the nonlinear Burgers term and spatial local integration of the Laplacian kernel for the Leray projection term. Note that for all $0\leq |\alpha|\leq q$ for the potential damping term in (\ref{dampu}) we may use
the fundamental solution $p^{l,m}$, where no first order derivative of the fundamental solution is used.
We mention that the transformation with the global factor $(1+\tau)$ in equation (\ref{uvtrans}) has the effect that the factors $\mu^{\tau,2}$ and $\mu^{\tau,1}$ become large in the coefficients, but for a fixed time horizon $T>0$ of the original Navier Stokes equation problem for the velocity components $v_i$ and small time step size $\rho$ the damping term of the comparison function turns out to be dominant. We shall also show below that the growth increments of the nonlinear terms have regular upper bounds which are small relative to the potential damping term due to the spatial derivative of the Gaussian in the local solution representations used. This also leads to the conclusion that there is no essential dependence of the time step size $\rho >0$ on the time horizon $T>0$.  Note that such a global transformation as it is used in this item i) is not appropriate in order to obtain numerical schemes especially for longer time. However, we can still obtain a global scheme. Here, the time variable occurring in $\mu^{\tau,k}$ for $0\leq k\leq 2$ is always local of the form $\tau_l=(\tau-l-1)$ where $\tau\in [l-1,l]$ such that the size of the factors $\mu^{\tau,k}$ does not depend on the time step number $l$ essentially. Especially, we have regular coefficients and we have uniform upper bounds for the coefficients $\mu^{\tau,k}$, which is essential.

Note that in the damping term 
\begin{equation}\label{dampu}
-\int_{m-1}^{\sigma}\int_{{\mathbb R}^n} D^{\alpha}_xu^{\rho,l,m}_i(s,y)\mu(s)p^{l,m}(\sigma,x;s,y)dyds
\end{equation}
obtained form the time dilatation transformation we apply the full derivative $D^{\alpha}_x$ to $u^{\rho,l,m}_i$ which we can do because we know that $u^{\rho,l,m}_i\in C^q\cap H^q$. This has the advantage that we convolute the source term $-u^{\rho,l,m}_i\mu(s)$ with the positive fundamental solution $p^{l,m}$ itself. Hence for small time step size $\rho>0$ this term with $\sigma=m$ becomes close to $D^{\alpha}_xu^{\rho,l,m}_i(m-1,y)\mu(m-1)$ in the sense that for any $\epsilon >0$ there is a size $\rho >0$ such that (as $m-(m-1)=1$) we have 
\begin{equation}\label{diff}
\begin{array}{ll}
{\big |}-\int_{m-1}^{m}\int_{{\mathbb R}^n} D^{\alpha}_xu^{\rho,l,m}_i(s,y)\mu(s)p^{l,m}(m,x;s,y)dyds\\
\\
+D^{\alpha}_xu^{\rho,l,m}_i(m-1,y)\mu(m-1){\big |}\leq {\epsilon}.
\end{array}
\end{equation}
Now consider the value $D^{\alpha}_iu^{\rho,l,m}(m-1,x)$ for some $1\leq i\leq n$ and some multiindex $\alpha$ with $0\leq |\alpha|\leq q$. If this value becomes close to $C$, let's say greater or equal to $C-1$ for some $C\gg 1$, then the damping term is close to to $C-1-\epsilon$ for small $\rho >0$ or arbitrarily close to $C-1$ as $\rho>0$ becomes small). On the other hand all other terms on the right side of (\ref{ulmrep*}) have the factor $\rho>0$ while the damping term has not. Perhaps we need to add a remark here concerning the left side of (\ref{ulmrep*}) which may be written as
\begin{equation}
\begin{array}{ll}
D^{\alpha}_xu^{\rho,l,m}_i(\sigma,x)-D^{\alpha}_xu^{\rho,l,m}_i(m-1,x)\\
\\
+D^{\alpha}_xu^{\rho,l,m}_i(m-1,x)-\int_{{\mathbb R}^n}D^{\alpha}_xu^{\rho,l,m}_i(m-1,y)p^{l,m}(\sigma,x;m-1,y)dy.
\end{array}
\end{equation}
There may be some concern that the growth term
\begin{equation}\label{leftgrowth}
D^{\alpha}_xu^{\rho,l,m}_i(m-1,x)-\int_{{\mathbb R}^n}D^{\alpha}_xu^{\rho,l,m}_i(m-1,y)p^{l,m}(\sigma,x;m-1,y)dy
\end{equation}
can dominate the damping term for small time.
\begin{rem}
Note that the second term in (\ref{ulmrep*}) can be put with a positive sign on the right sign such that the value $D^{\alpha}_xu^{\rho,l,m}_i(\sigma,x)$ is represented by the sum of 'its diffusion'
\begin{equation}\label{viscdamp}
\int_{{\mathbb R}^n}D^{\alpha}_xu^{\rho,l,m}_i(m-1,y)p^{l,m}(\sigma,x;m-1,y)dy
\end{equation}
plus some nonlinear terms and a potential damping term with negative sign. Actually, the  term in (\ref{viscdamp}) is a viscosity damping term. First not that the term in (\ref{viscdamp}) is a convolution with respect to the spatial variables. Using Fourier tansforms ${\cal F}$ with respect to the spatial variables convolutions transform to multiplications an the term in (\ref{viscdamp}) becomes
\begin{equation}\label{fourier}
{\cal F}\left( D^{\alpha}_xu^{\rho,t_0}_i\right)(\sigma,\xi)={\cal F}\left( D^{\alpha}_xu^{\rho,t_0}_i\right)(0,\xi)\exp\left(-\rho\nu \mu'|\xi|^2(\sigma-t_0)\right).
\end{equation}
We shall use this viscosity damping below for the construction of upper bounds which are independent of the time horizon in (iii).
\end{rem}
 
However, this does not really matter for our purposes here in item i) as we need a preservation of the norm only for discrete times. For this purpose we have to observe that the density in (\ref{leftgrowth}) integrates to $1$ such that the size of (\ref{leftgrowth}) becomes small as $\rho$ becomes small. And a closer analysis shows that the damping term even dominates for small time. 
Let is point out this a little. 
The right term in (\ref{leftgrowth}) may be splitted into two summands, and both summands estimated with standard upper bounds of the local and global Gaussian respectively. Given $x\in {\mathbb R}^n$ for a ball $B^n_{\epsilon}(x)$ of radius $\epsilon$ around $x$ we have
\begin{equation}
\begin{array}{ll}
\int_{{\mathbb R}^n}D^{\alpha}_xu^{\rho,l,m}_i(m-1,y)p^{l,m}(\sigma,x;m-1,y)dy\\
\\
=\int_{B^n_{\epsilon}(x)}D^{\alpha}_xu^{\rho,l,m}_i(m-1,y)p^{l,m}(\sigma,x;m-1,y)dy\\
\\
+\int_{{\mathbb R}^n\setminus B^n_{\epsilon}(x)}D^{\alpha}_xu^{\rho,l,m}_i(m-1,y)p^{l,m}(\sigma,x;m-1,y)dy\\
\\
\leq \int_{0<r\leq \epsilon}
{\big |}D^{\alpha}_xu^{\rho,l,m}_i(m-1,y(r)){\big |}\frac{\tilde{C}}{(\sigma-m-1)^{\delta}r^{n-2\delta}}r^{n-1}dr\\
\\
+\int_{{\mathbb R}^n\setminus B^n_{\epsilon}(x)}{\big |}D^{\alpha}_xu^{\rho,l,m}_i(m-1,y){\big |}\frac{\tilde{C}}{\sqrt{\rho(\sigma-(m-1))}^n}\exp\left(-\lambda_0 \frac{(x-y)^2}{\rho (\sigma-(m-1))}\right) dy, 
\end{array}
\end{equation}
for some generic constant $\tilde{C}>0$ and some $\lambda_0 >0$ which may be chosen independently of time step numbers and $\rho$, and where we introduced polar coordinates with radial component $r>0$. The convolutions may be estimated using a Young inequality (concerning upper bounds with respect to $H^m$), where the local upper bound of the Gaussian may be estimated with respect to a $L^1$-norm. This leads to
\begin{equation}
{\Big |}\frac{\tilde{C}}{(\sigma-m-1)^{\delta}r^{n-2\delta}}r^{n-1}{\Big |}_{L^1(r\leq {\epsilon})}\leq \frac{\tilde{C}\epsilon^{2\delta}}{(\sigma-m-1)^{\delta}}.
\end{equation}
As the damping term has
\begin{equation}
-D^{\alpha}_xu^{\rho,l,m}_i(m-1,y)\mu(m-1)(\sigma-(m-1))
\end{equation}
for small $\rho>0$ and $\sigma\in [m-1,m]$ we would like to choose 
\begin{equation}
\frac{\epsilon^{2\delta}}{(\sigma-m-1)^{\delta}}\in o(\sigma-(m-1))
\end{equation}
as $\sigma$ goes to $m-1$ while the complementary integral 
\begin{equation}\label{complement}
\int_{{\mathbb R}^n\setminus B^n_{\epsilon}(x)}{\big |}D^{\alpha}_xu^{\rho,l,m}_i(m-1,y){\big |}\frac{\tilde{C}}{\sqrt{\rho(\sigma-(m-1))}^n}\exp\left(-\lambda_0 \frac{(x-y)^2}{\rho (\sigma-(m-1))}\right) dy
\end{equation}
goes to zero for this choice of $\epsilon$. As a continuous square integrable function is bounded and in order to obtain an upper bound we may replace ${\big |}D^{\alpha}_xu^{\rho,l,m}_i(m-1,y){\big |}$ 
in the latter integral by a constant  get for the  choice
\begin{equation}
\epsilon =(\sigma-(m-1))^{\frac{1-\delta}{2\delta}}~\mbox{for}~\delta\in (0,0.5)
\end{equation}
for a constant $C^*$ and another constant $C_0$ we get the upper bound
\begin{equation}\label{complement2}
\begin{array}{ll}
\int_{r\geq \epsilon}\frac{C^*\tilde{C}}{\sqrt{\rho(\sigma-(m-1))}^n}\exp\left(-\lambda_0 \frac{r^2}{\rho (\sigma-(m-1))}\right)r^{n-1} dr\\
\\
\leq \int_{r\geq \epsilon}\frac{C_0C^*\tilde{C}}{\sqrt{\rho(\sigma-(m-1))}^n}\exp\left(-\frac{\lambda_0}{2} \frac{r^2}{\rho (\sigma-(m-1))}\right) dr,
\end{array}
\end{equation}
where the upper bound of the relevant factor of integrand at $r=\epsilon=\frac{1-\delta}{\delta}$ for the right side of (\ref{complement2}) we get the upper bound
\begin{equation}
\frac{2\rho}{\lambda_0}\frac{C_0C^*\tilde{C}}{\sqrt{\rho(\sigma-(m-1))}^{n+\frac{1}{2}{\Big (}\frac{\delta}{1-\delta}-1 {\Big )}}}\exp\left(-\frac{\lambda_0}{2} \frac{(\sigma-(m-1))^{\frac{\delta}{1-\delta}-1}}{\rho}\right)
\end{equation}
which goes to zero for $\delta \in (0,0.5)$ as $\sigma\downarrow (m-1)$.

It follows that we have a preservation of the upper bound $C$, i.e., for $m\geq 1$ and a $\rho$ (but dependent on $m$ for the transformation in (\ref{uvtrans}) with a global factor $(1+\tau)$) we have  
\begin{equation}\label{subschemem}
\sup_{z\in {\mathbb R}^n}{\big |}D^{\alpha}_xu^{\rho,l,m-1}_i(m-1,z){\Big |}\leq C \Rightarrow \sup_{z\in {\mathbb R}^n}{\big |}D^{\alpha}_xu^{\rho,l,m}_i(m,z){\Big |}\leq C
\end{equation}
first for some finite substep numbers $m\geq 1$. For the localized relation in (\ref{uvl0}) we have (\ref{subschemem}) for all substep numbers at once. For certain localized versions of the form (\ref{uvl0}) we can prove stronger upper bounds as we shall observe below. If we use a global factor $(1+\tau)$ as in (\ref{uvtrans}), then we need have to take care of the coefficients $\mu^{\tau,2}$. Note that $\mu$ has a factor $\frac{1}{1+\tau}$ such that the coefficient factor $\mu^{\tau,1}$ of the Laplacian is bounded and $\mu^{\tau,2}$ is linear in $\tau$. This is not really a problem. We can go time steps with respect to $\tau$ which are smaller than one such that the subscheme needs to be considered only for a finite time horizon. For example we could proceed with time steps of size $\left[l-1,l-\frac{1}{2}\right]$ with respect to $\tau$ which leads to time step size $\frac{1}{\sqrt{3}}$ with respect to $\sigma$ (remarkable that the price for a strong damping term is so small in the context of numerical intentions). If we consider the relation in (\ref{uvtrans}) with the global factor on an infinite interval then a first idea is to offset linear growth of the coefficient factors $\mu^{\tau,2}$ by using a subscheme with variable time step size $\rho_m\sim \frac{1}{m}$ (which is still global). However this seems to be not sufficient as the the damping term decreases with the factor $\frac{1}{1+\tau}$ for (\ref{uvtrans}). However the preservation claim in (\ref{subschemem}) is still valid as may be argued by contradiction again. We loose constructiveness, but there is no harm on the analytical level (if we trust classical mathematics). So any approach is possible from the point of view of classical mathematics. However, in order to satisfy the 'constructivist' we consider a constructive version of the proof below.

For the localized transformation in (\ref{uvl0}) an analogous argument leads to the preservation of an upper bound with respect to a $C^q\cap H^q$-norm uniformly in time, where $\rho$ is independent of the time step number $m\geq 1$ of the subscheme. For the transformation in (\ref{uvtrans}) we need to care for the situation where $\tau$ becomes large such that the factors $\mu^{\tau,k}:~k=0,1,2$ become large. However, this is also not a real problem. For the global factor transformation (\ref{uvtrans}) we can set up a scheme $u^{\rho_m,l,m-1}_i,~1\leq i\leq n$ with substep sizes $\rho_m$ which vary with $m$. We may have $\sum_{m\geq 1}\rho_m=c<\infty$ for \ref{subschemem}, but then we may argue by contradiction that we can continue. In a set-theoretical-oriented style you would set up a transfinite scheme. This is nothing extraordinary, and less 'metaphysically oriented' classical mathematicians can always succeed here using a contradiction argument. Let us add some more remarks.
For $p^{l,m}$ itself determined by (\ref{plm}) in the Levy expansion form above we observe that we get a Gaussian upper bound for $p^{l,m}$ with a dispersion constant which is smaller then the original dispersion constant. Furthermore, for the first order derivatives $p^{l,m}_{,j}$ we get upper bounds of the form 
\begin{equation}
\begin{array}{ll}
{\Big |}\frac{(x-y)_j}{4\rho\mu'(s)(\sigma-s)} \frac{1}{\sqrt{4\pi \rho\mu'(s)(\sigma-s)}^n}\exp\left(-\frac{|x-y|^2}{8\rho\mu'(s)(\sigma-s)} \right){\Big |}\\
\\
\leq {\big |}\frac{c}{(x-y)_j}{\big |} \frac{1}{\sqrt{4\pi \rho\mu'(s)(\sigma-s)}^n}\exp\left(-\frac{|x-y|^2}{16\rho\mu'(s)(\sigma-s)} \right),
\end{array}
\end{equation}
where $\mu'=\mu^{\tau,1}$ or a similar function if we use a transformation as in (\ref{alt}), and where
\begin{equation}
c:=\sup_{y\in {\mathbb R}^n}z\exp(-z^2).
\end{equation}
The local singularity for the Gaussian gets an additional factor $\frac{1}{(x_j-y_j)}$, but this leads to local upper bounds of the form
\begin{equation}
\frac{C}{(\sigma-m-1)|x-y|^{n+1-2\alpha}}
\end{equation}
for some $\alpha\in (0.5.1)$ which is clearly integrable for the considered dimensions $n\geq 3$ and we can then estimate as in the local contraction result above. Outside a ball we have exponential damping as $\rho>0$ becomes small.
All the terms with spatial derivatives in the equations for $u^{\rho,l,m}_i,~1\leq i\leq n$ have a factor $\rho$ which makes them small compared to the damping term introduced via the time dilatation transformation and an additional (linear) factor. As this source term is the only term of the representation in (\ref{ulmrep*}) which has no small factor $\rho >0$ and as all the other terms on the right side have this factor. 

Next we prove that the step size $\rho >0$ can be chosen essentially independently of the time horizon $T>0$ even in this case of a global time factor in the time delay transformation which we consider in this item i). 
First note that in (\ref{ulmrep*}) the nonlinear terms are of the form 
\begin{equation}\label{ulmrep*222}
 \begin{array}{ll}
-\rho\int_{m-1}^{\sigma}\int_{{\mathbb R}^n}f_i(s,y)p^{l,m}_{,j}(\sigma,x;s,y)dyds
\end{array}
\end{equation} 
on $[m-1,m]\times {\mathbb R}^n$, where
\begin{equation}
\begin{array}{ll}
f_i\in M_f:={\Big \{} \sum_{j=1}^n D^{\beta}_x\left( u^{\rho,l,m}_j\frac{\partial u^{\rho,l,m}_i}{\partial x_j}\right) (.,.)\mu^{\tau,2}(.),\\
\\
\sum_{j,p=1}^n\int_{{\mathbb R}^n}\left( \frac{\partial}{\partial x_i}K_n(.-y)\right)D^{\beta}_x\left( \frac{\partial u^{\rho,l,m}_p}{\partial x_j}\frac{\partial u^{\rho,l,m}_j}{\partial x_p}\right)(.,y)\mu^{\tau,2}(.)dy{\Big \}}.
\end{array}
\end{equation}
Note that we have
\begin{equation}
\sup_{\sigma\in [m-1,m],x\in {\mathbb R}^n, f_i\in M_f}{\big |}f_i(\sigma,x){\big |}\leq C_f
\end{equation}
for some finite constant $C_f>0$.
Rough pointwise upper bounds are of the form
\begin{equation}\label{pointupper}
\begin{array}{ll}
\rho\int_{m-1}^{\sigma}\int_{{\mathbb R}^n}\frac{\tilde{C}_0}{\sqrt{\rho(\sigma-(m-1))}^n}\frac{{\big |}f_i(s,y)(x_j-y_j){\big |}}{\rho (\sigma-(m-1))}\exp\left(-\lambda_0 \frac{(x-y)^2}{\rho (\sigma-(m-1))}\right)dyds
\end{array}
\end{equation}
for some constant $\tilde{C}_0>0$ which depends only on the dimension and the ellipticity constant $\nu$ (or, respectively, on  $\lambda_0$). We can use such rough upper bounds for the integral up to time $\sigma >m-1$ over the complement a ball $B_{\rho^r}(x)$ of radius $\rho^r$ around $x$ in ${\mathbb R}^n$ for $r\in \left(\frac{1}{3},\frac{1}{2} \right)$. We have
\begin{equation}\label{pointupper1}
\begin{array}{ll}
\rho\int_{m-1}^{\sigma}\int_{{\mathbb R}^n\setminus B_{\rho^r}(x)}\frac{\tilde{C}_0}{\sqrt{\rho(\sigma-(m-1))}^n}\frac{{\big |}f_i(s,y)(x_j-y_j){\big |}}{\rho (\sigma-(m-1))}\exp\left(-\lambda_0 \frac{(x-y)^2}{\rho (\sigma-(m-1))}\right)dyds\\
\\
\leq \rho^{1-r}\int_{m-1}^{\sigma}\frac{C_1}{\sqrt{\rho(\sigma-(m-1))}^n}C_f\exp\left(-0.5 \lambda_0 \frac{1}{ \rho^{1-2r}(\sigma-(m-1))}\right)dy\\
\\
\leq \rho^{1-r}C_2C_f (\sigma-(m-1))
\end{array}
\end{equation}
for some constants $C_1,C_2>0$ which depend only on dimension $n$ and the ellipticity constant $\nu >0$ (resp. $\lambda_0$). This estimate can be strengthened: we have
\begin{equation}\label{pointupper1*}
\begin{array}{ll}
\rho\int_{m-1}^{\sigma}\int_{{\mathbb R}^n\setminus B_{\rho^r}(x)}\frac{\tilde{C}_0}{\sqrt{\rho(\sigma-(m-1))}^n}\frac{{\big |}f_i(s,y)(x_j-y_j){\big |}}{\rho (\sigma-(m-1))}\exp\left(-\lambda_0 \frac{(x-y)^2}{\rho (\sigma-(m-1))}\right)dyds\\
\\
\leq \rho^{1-r}\int_{m-1}^{\sigma}\frac{C_1(\rho^{1-2r}(\sigma-(m-1)))^k(\sigma-(m-1))^{\frac{n}{2(1-2r)}-\frac{n}{2}}}{\sqrt{\rho^{1-2r}(\sigma-(m-1))}^{\frac{n}{1-2r}+2k}}C_f\exp\left(-0.5 \lambda_0 \frac{1}{ \rho^{1-2r}(\sigma-(m-1))}\right)dy\\
\\
\leq (\rho^{1-2r})^{2k+1}\tilde{C}_2C_f (\sigma-(m-1))^{k+1},
\end{array}
\end{equation}
where the new constant $\tilde{C}_2$ still depends only on the dimension $n$.

For the complementary integral on the ball $B_{\rho^r}(x)$ we use a finer estimate. We have
\begin{equation}\label{ulmrep*333}
 \begin{array}{ll}
{\Big |}\rho\int_{m-1}^{\sigma}\int_{B_{\rho^r}(x)}f_i(s,y)p^{l,m}_{,j}(\sigma,x;s,y)dyds{\Big |}\\
\\
={\Big |}\rho\int_{m-1}^{\sigma}\int_{B_{\rho^r}(x)}f_i(s,y)\frac{(x_j-y_j)}{\sigma -s}p^{l,m,*}(\sigma,x;s,y)dyds{\Big |},
\end{array}
\end{equation} 
where
\begin{equation}
p^{l,m,*}(\sigma,x;s,y)=\frac{1}{4\rho\nu\mu'(s)}\left( G_{\mu'(s)}(\sigma,x;s,y)+\int_{s}^{\sigma}G_{\mu'(s)}(\sigma,x;\zeta,z)\phi(\zeta ,z;s,y)dzd\zeta\right) 
\end{equation}
Here we take the derivative with respect to the spatial variable $xj$ of
\begin{equation}
G_{\mu'(s)}(\sigma,x;s,y)=\frac{1}{\sqrt{4\pi \rho\nu\mu'(s)(\sigma-s)}^n}\exp\left(-\frac{|x-y|^2}{4\rho\nu\mu'(s)(\sigma-s)} \right) 
\end{equation}
as the leading term of all orders of the classical Levy expansion 
\begin{equation}\label{fund}
p^{l,m}(\sigma,x;s,y)=G_{\mu'(s)}(\sigma,x;s,y)+\int_{s}^{\sigma}G_{\mu'(s)}(\sigma,x;\zeta,z)\phi(\zeta ,z;s,y)dzd\zeta,
\end{equation}
where the higher order term
\begin{equation}\label{phi}
\phi(\sigma ,x;s,y)=\sum_{p\geq 1}\rho^p(L_{\mu'}G_{\mu'(s)})_p(\sigma ,x;s,y),
\end{equation}
have higher order powers of $\rho$ (cf. the text above for the recursive definition of $(L_{\mu'}G_{\mu'(s)})_p$ which is also standard). From that definition it is clear that we have a antisymmetry in (\ref{ulmrep*333}), i.e.,
\begin{equation}\label{ulmrep*333}
 \begin{array}{ll}
{\Big |}\rho\int_{m-1}^{\sigma}\int_{B_{\rho^r}(x)}f_i(s,y)\frac{(x_j-y_j)}{\sigma -s}p^{l,m,*}(\sigma,x;s,y)dyds{\Big |}=\\
\\
{\Big |}\rho\int_{m-1}^{\sigma}\int_{B^{x_j\geq y_j}_{\rho^r}(x)}{\big (}f_i(s,y)-f_i(s,y^{-j}){\big )}\frac{(x_j-y_j)}{\sigma -s}p^{l,m,*}(\sigma,x;s,y)dyds{\Big |},
\end{array}
\end{equation} 
where in the last term we integrate over a half-sphere 
\begin{equation}
B^{x_j\geq y_j}_{\rho^r}(x)=\left\lbrace y\in B_{\rho^r}(x)|x_j\geq y_j\right\rbrace 
\end{equation}
and 
$y^{-j}=\left(y^{-j}_1,\cdots ,y^{-j}_n\right)$ is the vector with $y^{-j}=(-1)^{\delta_{ij}}$ if $y=(y_1,\cdots ,y_n)$ and $\delta_{ij}$ is the Kronecker $\delta$. 
For each $1\leq j\leq n$ we use continuity of the 'data' $f_i$. Especially for each $\epsilon  >0$ there exists a $\rho >0$ such that for ${|}x-y{|}\leq \rho^r$ we have 
\begin{equation}
{\big |}f_i(s,x_1,\cdots ,x_j,\cdots ,x_n)-f_i(s,x_1,\cdots ,-x_j,\cdots ,x_n){\big |}\leq \epsilon.
\end{equation}
Furthermore, for the first order spatial derivatives of the fundamental solution $p^{l,m}$ we get the standard estimate
\begin{equation}
\begin{array}{ll}
{\Big |}\rho\int_{m-1}^{\sigma}\int_{B_{\rho^r}(x)}\frac{(x_j-y_j)}{\sigma -s}p^{l,m,*}(\sigma,x;s,y)dyds{\Big |}\leq \rho^{1+\delta_0} C_4(\sigma -(m-1)),
\end{array}
\end{equation}
where $C_4$ is a constant which depends only on the ellipticity constant $\lambda_0$ and the dimension and $\delta_0$ is a small positive constant (dimension-free). Here we use 
\begin{equation}
{\Big |}G_{\mu'(s),j}(\sigma,x;s,y){\Big |}=\frac{\frac{|x_j-y_j|}{4\rho\nu\mu'(s)(\sigma-s)}}{\sqrt{4\pi \rho\nu\mu'(s)(\sigma-s)}^n}{\Bigg |}\exp\left(-\frac{|x-y|^2}{4\rho\nu\mu'(s)(\sigma-s)} \right){\Bigg |} 
\end{equation}
for the leading term of the classical Levy expansion of $p^{l,m}_{,j}$, where the higher order terms have similar estimates multiplied by powers of $\rho$, and the factor $\frac{1}{1-\rho}$ can be absorbed by $C_4$ for small $\rho$.
Hence,
\begin{equation}\label{epsilonest}
\begin{array}{ll} 
{\Big |}\rho\int_{m-1}^{\sigma}\int_{B^{x_j\geq y_j}_{\rho^r}(x)}{\big (}f_i(s,y)-f_i(s,y^{-j}){\big )}\frac{(x_j-y_j)}{\sigma -s}p^{l,m,*}(\sigma,x;s,y)dyds{\Big |}\\
\\
\leq \rho^{1+\delta0} C_4(\sigma -(m-1))\epsilon .
\end{array}
\end{equation} 
For arbitrary small $\epsilon >0$ we find a small $\rho >0$ such that the latter relation holds, and for this $\rho$ we still have (\ref{pointupper1*})! Note that for the global time transformation of this item i) we have potential damping term of order $\frac{1}{T}$ with respect to the time horizon. Now the estimate in (\ref{epsilonest}) and in (\ref{pointupper1*}) for $k>0$ ensure that we can choose
\begin{equation}
\rho \gtrsim\frac{1}{T^s} \mbox{ for some }s\in \left(r,1\right) 
\end{equation}
and this ensures that the global regular existence result for $v^{\rho}_i,~1\leq i\leq n$ via a scheme $u^{\rho,l,m}_i,~1\leq i\leq n,~l,m\geq 1$ leads to a global regular existence result for the original velocity component functions $v_i,~1\leq i\leq n$.

Now, via the time dilatation argument we can argue that the functions $v^{\rho,l}_i(l,.)$ and their multivariate spatial derivatives up to order $q$ inherit a linear upper bound such that we get a global linear upper bound altogether. The argument is as follows. The relation in (\ref{ulmrep*}) suggests that we may consider different norms or magnitude for preservation arguments. This is true indeed. For example we could assume that for time step number $l\geq 1$ and some constant $C>0$ we have established a global upper bound
\begin{equation}
\max_{1\leq i\leq n,~0\leq |\alpha|\leq p}\sup_{y\in {\mathbb R}^n}{\big |}D^{\alpha}_x v^{\rho,l-1}_i(l-1,y){\big |}\leq C(1+(l-1)),
\end{equation}
and then use the relation (\ref{ulmrep*}) in order to prove preservation in with respect to this norm. What is needed is just a local contraction result. We have considered this only for $H^q\cap C^q$ but others are certainly possible. As we are flexible with respect to the norm let us not loose this. Let us keep the norm unspecified and just write $|.|$. As we consider the derivatives explicitly the reader may have in mind
\begin{equation}
|.|=|.|_{L^2}~~~~\mbox{or}~~~~|.|=\sup_{y\in {\mathbb R}^n}|.|.
\end{equation}
Let us consider 
\begin{equation}
\max_{1\leq i\leq n,~0\leq |\alpha|\leq p}{\big |}D^{\alpha}_x v^{\rho,l-1}_i(l-1,.){\big |}\leq C(1+(l-1)).
\end{equation}
Recall that for $\tau \in [l-1,l]$ we have
\begin{equation}\label{vu}
v^{\rho,l}_i(\tau,x)=(1+\tau)u^{\rho,l}_i(\sigma,x),~\mbox{ where }(\sigma,x)=\left(\frac{\tau}{\sqrt{1-\tau^2}},x\right). 
\end{equation}
This means that for $l\geq 1$ we have
\begin{equation}
\begin{array}{ll}
\frac{\max_{1\leq i\leq n,~0\leq |\alpha|\leq p}{\big |}D^{\alpha}_x v^{\rho,l-1}_i(l-1,.){\big |}}{ (1+(l-1))}\\
\\
=\max_{1\leq i\leq n,~0\leq |\alpha|\leq p}{\big |}D^{\alpha}_x u^{\rho,l-1}_i(l-1,.){\big |}\leq C.
\end{array}
\end{equation}
As a local classical existence result holds for $v^{\rho,l}_i,~1\leq i\leq n$ on $[l-1,l]\times {\mathbb R}^n$ we have classical existence for the function $u^{\rho,l}_i,~1\leq i\leq n$ as well, where $u^{\rho,l}_i(\sigma,.)\in H^q\cap C^q$ for all $\sigma \geq 0$, and local $t$-time coordinates transfer to global $\sigma$-time coordinates. However in order to estimate the growth at each time step we do a local analysis of a transformed global function $u^{\rho,l}_i,~1\leq i\leq n$ known to exist (as we know $v^{\rho,l}_i,~1\leq i\leq n$ and due to (\ref{vu})). 
Hence, the local analysis of a globally transformed local function $u^{\rho,l}_i,~1\leq i\leq n$ is done in order to estimate the growth of $u^{\rho,l}_i$ for all $1\leq i\leq n$, on the whole domain $[0,\infty)\times {\mathbb R}^n$, and therefore the growth of $v^{\rho,l}_i$ from time $l-1$ to time $l$ for all $1\leq i \leq n$ via the relation (\ref{vu}). This is done for each $l\geq 1$ by recursive local analysis of the functional series $u^{\rho,l,m}_i,~m\geq 0$ for all $1\leq i\leq n$, where it is shown that for all time step $l$ upper bounds are inherited for each substep $m$, i.e., that     
\begin{equation}\label{upperboundinherit}
{\big |}D^{\alpha}_x u^{\rho,l-1,m-1}_i(l-1,y){\big |}\leq C\Rightarrow {\big |}D^{\alpha}_x u^{\rho,l-1,m}_i(l,y){\big |}\leq C
\end{equation}
for all $y\in {\mathbb R}^n$ and $m\geq 1$ at a time step $l\geq 1$,
as observed above.

At each time step $l$ we have a step size $\rho >0$ such that the contraction result holds on the time interval $[l-1,l-1+\rho]$, and the subscheme then guarantees that we have linear growth on this time interval of the original value function $v^{\rho,l}_i$. This type of argument leads to global existence but has a non-constructive aspect. However, note that from the perspective of classical mathematics even this weak form of knowledge about contraction leads to global existence by a contradiction argument. We start with this weak form of argument (as in a certain respect it is the simplest form of argument), and then we shall provide stronger constructive arguments below. Concerning the weak classical argument we first observe that from (\ref{upperboundinherit}) we have for all $1\leq i\leq n$ and $0\leq |\alpha|\leq q$ that
\begin{equation}
{\big |}D^{\alpha}_x u^{\rho,l-1}_i(l-1,y){\big |}\leq C\Rightarrow {\big |}D^{\alpha}_x u^{\rho,l}_i(l,y){\big |}\leq C
\end{equation}
for all time steps and all $y\in {\mathbb R}^n$. We conclude that
\begin{equation}
{\big |}D^{\alpha}_x v^{\rho,l}_i(l,y){\big |}=(1+l){\big |}D^{\alpha}_x u^{\rho,l}_i(l,y){\big |}\leq C(1+l)
\end{equation}
for all $y\in {\mathbb R}^n$. It is clear that the latter conclusion leads to
\begin{equation}
{\big |}D^{\alpha}_x v^{\rho,l}_i(l,y){\big |}\leq C'(1+\tau)
\end{equation}
Hence, from the local contraction results and the inheritance of a global upper bound by he subscheme we have (with a generic adaption of $C$, i.e. with $C$ multiplied by a factor which counts the number of terms in the norm) we have for example
\begin{equation}
{\big |}v^{\rho,l}_i(l,.){\big |}_{C^q\cap H^q}=(1+l){\big |}u^{\rho,l}_i(l,y){\big |}_{C^q\cap H^q}\leq C(1+l)
\end{equation}
for some $q\geq 2$ (generic $C>0$). It is clear that this implies
\begin{equation}
{\big |}v^{\rho,l}_i(\tau,.){\big |}_{C^q\cap H^q}=(1+l){\big |}u^{\rho,l}_i(l,y){\big |}_{C^q\cap H^q}\leq C(1+\tau)
\end{equation}
for all $l\geq 1$ and $\tau\in [l-1,l]$ (we suppress a local index of $\tau$ all the time for convenience). This is a description for a global linear upper bound for the original velocity component functions $v_i$ in $C^0\left([0,T],H^q\cap C^q \right)$. There is no maximal finite $T>0$ by simple contradiction arguments. 
It remains the matter of time step size. At each time step $l\geq 1$ we have to choose the time step size $\rho>0$ such that the contraction result holds. If we consider the contraction result for the function $v^{\rho,l}_i,~1\leq i\leq n$ or for a related function $v^l_i,~1\leq i\leq n$ in original time coordinates, then the argument becomes nonconstructive. Note that in the contraction results for functions $v^{\rho,l}_i,~1\leq i\leq n$ the time step size for contraction depends on the size of the data, i.e., the final data of the previous time step. Our rough estimate leads to a time step size which is reciprocal proportional to the square of the upper bound of this data size. Hence if we have a linear bound for the global velocity function $v^{\rho,l}_i$ or $v^{l}_i$, then the sum of infinitely many time step sizes may be finite. For this reason this simple argument has a non-constructive aspect. However it is clear that for the function $u^{\rho,l}_i$ the initial data are always constant $C>0$ (independent of the time step number $l\geq 1$), such that the $\rho >0$ dependent on $C>0$ can be chosen uniformly, and the scheme becomes constructive for any finite time horizon $T>0$, because a choice of the time step size $\rho \gtrsim \frac{1}{T^s}$ for $s\in \left(\frac{2}{3},1\right)$ is sufficient as we have shown via (\ref{epsilonest}) and in (\ref{pointupper1*}). Note that even nonconstructive schemes are allowed in classical mathematics, i.e., in mathematics with axiomatic systems like Zermelo-Fr\"{a}nkel or Bernays-G\"{o}del, but we have to show that it leads to a global result. It is for this reason that we consider alternative versions in this paper, which are more constructive in order to convince the constructivist that our argument holds in a more restrictive sense.  
As we have mentioned, the weak classical form of argument can be closed by the remark that the assumption that the scheme leads to a bounded solution for finite time only leads to a contradiction: assume that for all time steps $l\geq 1$ a time step size $\rho_l>0$ has been chosen such that a contraction result for $v^{\rho,l}_i,~1\leq i\leq n$ is valid on the domain $[\sum_{m=1}^{l-1}\rho_m,\sum_{m=1}^{l}\rho_m]\times {\mathbb R}^n$. The contraction result above shows that for all integer time step numbers $l\geq 1$ such a time step size $\rho_l$ can be chosen such that a) a local contraction result holds on this domain, and b) if $v^{\rho,l-1}_i,~1\leq i\leq n$ is the solution of the Navier Stokes equation at time $l-1$ (determined by the initial data $h_i,~1\leq i\leq n$ via $l-1$ time steps of the main recursive scheme for $v^{\rho,l}_i,~1\leq i\leq n$ on the original time level), then $(t,x)\rightarrow v^{\rho,l}_i(t,x),~ 1\leq i\leq n$ is a solution on the time interval $[\sum_{m=1}^{l-1}\rho_m,\sum_{m=1}^{l}\rho_m]\times {\mathbb R}^n$. Furthermore, the subscheme time-dilatation argument shows that the the original value function $v^{\rho,l}_i(\tau,.),~ 1\leq i\leq n$ has a linear upper bound $C+C\tau$ for all time step numbers $l\geq 1$ and for a constant $C>0$ which is independent of time. Now if the limit $\sum_{l= 1}^{\infty}\rho_l=T<\infty$, then we have
\begin{equation}
(1+T)D^{\alpha}_xu^{\rho,l}_i(T,.)=D^{\alpha}_xv^{\rho,l}_i(T,.)\leq C(1+T)
\end{equation}
for $0\leq |\alpha|\leq q$ and where $C>0$ is a preserved upper bound of all subschemes $u^{\rho,l,m}_i,~1\leq i\leq n$ which is independent of time. As we have the upper bound $(1+T)C$ of the value function and its derivatives for for $0\leq |\alpha|\leq q$ we have a time step size $\rho>0$ such that the local contraction result holds on some time interval $[T,T+\rho]$, a contradiction to the assumption that the $T>0$ can be a maximal time where the velocity solution is bounded. This leads to global regular existence with regularity of order $q\geq 2$ for any given $q\geq 2$. We note that even  this weak and not completely constructive form of the scheme still leads to global existence in the framework of classical mathematics by transfinite induction with respect to time, i.e., by transfinite ordinal number induction. Or- if you have no intimacy with set theory, you may call it an argument by contradiction. As we said the scheme becomes constructive for any finite time horizon if we consider the contraction for the comparison function $u^{\rho,l}_i,~1\leq i\leq n$.

Concerning item ii) we only remark that both estimates in item i) and item iii) can also be obtained by using the information of $v^{\rho,l}_i$ in order to get a representation for $u^{\rho,l,m}_i,~1\leq i\leq m$ where the convection term is treated as a first order coefficient term (known by the knowledge of $v^{\rho,l}_i,~1\leq i\leq n$)  and where the Leray projection term is the only source term in a representation with a corresponding fundamental solution which includes the knowledge of the first order convection terms. We may  then use the adjoint of the fundamental solution in order to estimate higher order spatial derivatives.

Ad iii) the considerations so far lead to us to the conclusion that the possibility of a constructive auto-control scheme  is related to the inheritance of a uniform upper bound for the original velocity components $v_i,~1\leq i\leq n$ from time $t_0$ to time $t_0+\rho$ which is related in turn to the inheritance of a uniform upper bound of a comparison function $u^{*,\rho,t_0}_i, 1\leq i\leq n$ for a related to time interval, where we add a $*$ upperscript to denote the name of the latter comparison function as it is a localized variation of the comparison functions $u^{\rho,t_0}_i, 1\leq i\leq n$ considered earlier. Here we argued that in item i) above that even for the global time transformation used there the time step size $\rho$ may be chosen only weakly dependent of the the time horizon, i.e., in the form $\rho \gtrsim \frac{1}{T^s}$ for some appropriate positive $s<1$. Now we sharpen the results. On a subinterval of the (original time $t$) time interval $[t_0,t_0+\rho)$ we consider the comparison function $u^{\rho,t_0}_i,~1\leq i\leq n$, i.e., on the subinterval of the unit time interval $\tau\in [t_0,t_0+1)$ we consider the comparison function $u^{*,\rho,t_0}_i,~1\leq i\leq n$ defined by
\begin{equation}\label{*eqloc}
\lambda(1+(\tau-t_0))u^{*,\rho,t_0}_i(\sigma,.)=v^{\rho}_i(\tau,.)=v_i(t,.),~\sigma=\frac{\tau-t_0}{\sqrt{1-(\tau-t_0)^2}}
\end{equation}
on an appropriate time interval $\tau\in \left[t_0,t_0+\delta\right]$ corresponding to the original time interval $t\in \left[t_0,t_0+\delta\rho\right]$ such that
\begin{equation}\label{sigmatime}
\sigma \in \left[0,\frac{\delta}{\sqrt{1-\delta^2}}\right]. 
\end{equation}
The parameters $\lambda >0$ and $\delta >0$ will be chosen to be small positive constants. 
Assume that we have computed $v^{\rho,l-1}_i(l-1,.),~ 1\leq i\leq n$ for $l\geq 1$ where at each time step $l\geq 1$ we have $t_0=l-1$ in the transformation above such that
\begin{equation}\label{timelocalscheme}
v^{\rho,l}_i(\tau,x)=\lambda (1+(\tau-(l-1))u^{*,\rho,l}_i(\sigma,x)
\end{equation}
where $\sigma\in \left[0,\frac{\delta}{\sqrt{1-\delta^2}}\right]$ is defined as in (\ref{*eqloc}) and corresponds to a local time variable $\tau\in \left[l-1,l-1+\delta\right]$. The following scheme may be iterated for substeps of length $\frac{\delta}{\sqrt{1-\delta^2}}$ with respect to the $\sigma$-variable correspondind to time intervals with respect to time variable $\tau$ of length $\delta$, i.e. at substep $m$ of time step $l$ to the time interval  
\begin{equation}
\left[l-1+(m-1)\delta,l-1+m\delta\right]
\end{equation}
for finitely many substeps $m\geq 1$ until time $\tau=l$ is achieved, such that the next main time step $l$ can be initialized. For each substep number $m\geq $ we consider the function 
$u^{*,\rho,t_0}_i(\sigma,.),~ 1\leq i\leq n$ with $t_0=l-1+(m-1)\frac{\delta}{\sqrt{1-\delta^2}}$ on the time interval $\left[0,\frac{\delta}{\sqrt{1-\delta^2}} \right]$.  
This means that at each time step number $l\geq 1$ we go by time step size $\delta\rho$ in original time coordinates $t$ and by time step size $\delta$ in time coordinates $\tau=\rho t$ which in turn corresponds to a time interval for $\sigma$ of the form give in (\ref{sigmatime}).
Assume inductively that at the beginning of time step $l$ we have
\begin{equation}
{\big |}v^{\rho,l-1}_i(l-1,.){\big |}_{H^q\cap C^q}\leq C
\end{equation}
for $q\geq 2$, we want to show that this upper bound is inherited, i.e., that we have 
\begin{equation}
{\big |}v^{\rho,l-1}_i(l-1+(m-1)\delta,.){\big |}_{H^q\cap C^q}\leq C,
\end{equation}
where we use the comparison function $u^{*,\rho,t_0}_i(\sigma,.),~ 1\leq i\leq n$ with the appropriate time upper script $t_0$. Since the following estimates do not depend essentially on $t_0$, we do not specify $t_0$ in the following in order to keep notation simple. At the beginning of each time substep we have inductively
\begin{equation}
{\big |}v^{\rho,l-1}_i(t_0,.){\big |}_{H^q\cap C^q}\leq C
\end{equation}
and, hence,
\begin{equation}
{\big |}u^{*,\rho,t_0}_i(0,.){\big |}_{H^q\cap C^q}\leq \frac{C}{\lambda}.
\end{equation}
Note that
\begin{equation}\label{*eqloc2}
\begin{array}{ll}
\lambda(1+\delta){\big |}u^{*,\rho,t_0}_i\left(\frac{\delta}{\sqrt{1-\delta^2}},. \right) ,.){\big |}_{H^q\cap C^q}={\big |}v^{\rho}_i\left( t_0+\delta,.\right) {\big |}_{H^q\cap C^q}\\
\\
={\big |}v_i\left( t_0+\rho\delta,.\right) {\big |}_{H^q\cap C^q}.
\end{array}
\end{equation}
Hence, it suffices to show that 
\begin{equation}\label{*eqloc3}
\begin{array}{ll}
{\big |}u^{*,\rho,t_0}_i\left(\frac{\delta}{\sqrt{1-\delta^2}},. \right) ,.){\big |}_{H^q\cap C^q}\leq \frac{C}{\lambda(1+\delta)},
\end{array}
\end{equation}
because this leads to
\begin{equation}\label{*eqloc3}
\begin{array}{ll}
{\big |}v^{\rho}_i\left( t_0+\delta,.\right) {\big |}_{H^q\cap C^q}={\big |}v_i\left( t_0+\rho\delta,.\right) {\big |}_{H^q\cap C^q}\\
\\
=(1+\lambda \delta){\big |}u^{*,\rho,t_0}_i\left(\frac{\delta}{\sqrt{1-\delta^2}},. \right) ,.){\big |}_{H^q\cap C^q}\leq C.
\end{array}
\end{equation}
Applied to the scheme, at each time step number we get inductively with respect to the time substep number $m$
\begin{equation}\label{*eqloc3}
\begin{array}{ll}
{\big |}v^{\rho}_i\left( l-1+m\delta,.\right) {\big |}_{H^q\cap C^q}={\big |}v_i\left( l-1+m\rho\delta,.\right) {\big |}_{H^q\cap C^q}\leq C,
\end{array}
\end{equation}
for all substeps $m\geq 1$, i.e., an uniform global upper bound at discrete times (and especially at time $l$ after finitely many time steps), and since this is inductively true for all time step numbers $l\geq 1$, this leads to an global uniform upper bound independent of time.
In order to estimate the growth of this comparison function we consider the related equation. We have
\begin{equation}
\begin{array}{ll}
\frac{\partial v^{\rho,l}_i}{\partial \tau}=\lambda u^{*,\rho,l}_i(\sigma,x)+\lambda (1+(\tau-(l-1))\frac{\partial}{\partial \sigma}u^{*,\rho,l}_i(\sigma,x)\frac{d\sigma}{d\tau}\\
\\
=\lambda u^{*,\rho,l}_i(\sigma,x)+\lambda (1+(\tau-(l-1))\frac{\partial}{\partial \sigma}u^{*,\rho,l}_i(\sigma,x)\frac{1}{\sqrt{1-\tau^2}^3},
\end{array}
\end{equation}
and the equation (with $\tau_l=\tau-(l-1)$)
\begin{equation}\label{Navleray2aa}
\left\lbrace \begin{array}{ll}
\frac{\partial u^{*,\rho,l}_i}{\partial \sigma}=\rho\nu\sqrt{1-\tau^2_l}^3\sum_{j=1}^n \frac{\partial^2 u^{*,\rho,l}_i}{\partial x_j^2}-\frac{\sqrt{1-\tau^2_l}^3}{1+(\tau-(l-1))}u^{*,\rho,l}_i\\
\\ 
-\rho\sqrt{1-\tau^2_l}^3\lambda (1+(\tau-(l-1))\sum_{j=1}^n u^{\rho,l}_j\frac{\partial u^{*,\rho,l}_i}{\partial x_j}\\
\\ +\rho\sqrt{1-\tau^2_l}^3\lambda (1+(\tau-(l-1))\sum_{j,m=1}^n\int_{{\mathbb R}^n}\left( \frac{\partial}{\partial x_i}K_n(x-y)\right)\times\\
\\
\times \sum_{j,m=1}^n\left( \frac{\partial u^{*,\rho,l}_m}{\partial x_j}\frac{\partial u^{*,\rho,l}_j}{\partial x_m}\right) (\sigma,y)dy,\\
\\
\mathbf{u}^{*,\rho,l}(0,.)=\frac{1}{\lambda}\mathbf{v}^{\rho,l-1}(l-1,.),
\end{array}\right.
\end{equation}
which is considered on the time interval in (\ref{sigmatime}).

The comparison function $u^{*,\rho,t_0}_i,~1\leq i\leq n$ (with $t_0=l-1$ at some time step $l\geq 1$) has local representations in terms of convolutions, where the nonlinear terms in the representation are convolutions with the first spatial derivative of the Gaussian type fundamental solution $G$ of a heat equation $G_{,\tau}-\rho \nu \mu^{\tau,1} \Delta G=0$ (cf. below). More precisely, for the multivariate spatial derivatives of order $|\alpha|\geq 1$ we get for the increment
\begin{equation}
D^{\alpha}_x\delta u^{*,\rho,t_0}(\sigma,x)=D^{\alpha}_x u^{*,\rho,t_0}(\sigma,x)-D^{\alpha}_xu^{*,\rho,t_0}(0,x)
\end{equation}
the representation
\begin{equation}\label{Navlerayusubschemepre}
\begin{array}{ll}
D^{\alpha}_x\delta u^{*,\rho,t_0}_i(\sigma,x)=\int_{{\mathbb R}^n}D^{\alpha}_xu^{*,\rho,t_0}_i(0,y)G(\sigma,x;0,y)dy-D^{\alpha}_xu^{*,\rho,t_0}(0,x)\\
\\
-\int_{0}^{\sigma}\int_{{\mathbb R}^n}\mu^*(s)  u^{*,\rho,t_0,p}_{i,\alpha}(s,y)G(\sigma,x;s,y)dyds\\
\\
-\rho\int_{0}^{\sigma}\int_{{\mathbb R}^n}\mu^{*,\tau,2}(s)\sum_{j=1}^n \left( u^{*,\rho,t_0}_j\frac{\partial u^{\rho,t_0}_i}{\partial x_j}\right)_{,\beta} (s,y)G_{,j}(\sigma,x;s,y)dyds\\
\\
+ \rho\int_{0}^{\sigma}\int_{{\mathbb R}^n} \mu^{*,\tau,2}(s)\sum_{j,r=1}^n\int_{{\mathbb R}^n}\left( \frac{\partial}{\partial x_i}K_n(z-y)\right)\times\\
\\
\times \sum_{j,r=1}^n\left( \frac{\partial u^{*,\rho,t_0}_r}{\partial x_j}\frac{\partial u^{*,\rho,t_0}_j}{\partial x_r}\right)_{,\beta} (s,y)G_{,j}(\sigma,x;s,z)dydzds,
\end{array}
\end{equation}
where the coefficients $\mu^{*}$ and $\mu^{*,\tau,2}$ can be read off by comparison with (\ref{Navleray2aa}) - they are just localized versions of the coefficients defined earlier.
For the increment of the value function itself we can still get for $u^{\rho,t_0}_i\in H^m\cap C^m,~m\geq 2$ for all $1\leq i\leq n$ for a ball $B^n_R(x)$ of arbitrary large radius $R>0$ in ${\mathbb R}^n$ around $x$ the representation
\begin{equation}\label{Navlerayusubschemepre222}
\begin{array}{ll}
\delta u^{*,\rho,t_0}_i(\sigma,x)=\int_{{\mathbb R}^n}u^{*,\rho,t_0}_i(0,y)G(\sigma,x;0,y)dy
-u^{*,\rho,t_0}(0,x)\\
\\
-\int_{0}^{\sigma}\int_{{\mathbb R}^n}\mu(s)  u^{\rho,t_0}_{i}(s,y)G(\sigma,x;s,y)dyds\\
\\
-\sum_j\rho\int_{0}^{\sigma}\int_{{\mathbb R}^n}\mu^{\tau,2}(s)\sum_{j=1}^n F_{ij}(\mathbf{u}^{\rho,t_0})(s,y)G_{,j}(\sigma,x;s,y)dyds\\
\\
+ \rho\int_{0}^{\sigma}\int_{B^n_R(x)} \mu^{\tau,2}(s)\sum_{j,r=1}^n\int_{{\mathbb R}^n}\left( K_n(z-y)\right)\times\\
\\
\times \sum_{j,r=1}^n\left( \frac{\partial u^{\rho,t_0}_r}{\partial x_j}\frac{\partial u^{\rho,t_0}_j}{\partial x_r}\right) (s,y)G_{,i}(\sigma,x;s,z)dydzds\\
\\
+ \rho\int_{0}^{\sigma}\int_{{\mathbb R}^n\setminus B^n_R(x)} \mu^{\tau,2}(s)\sum_{j,r=1}^n\int_{{\mathbb R}^n}\left( K_{n,i}(z-y)\right)\times\\
\\
\times \sum_{j,r=1}^n\left( \frac{\partial u^{\rho,t_0}_r}{\partial x_j}\frac{\partial u^{\rho,t_0}_j}{\partial x_r}\right) (s,y)G(\sigma,x;s,z)dydzds\\
\\
+\mbox{boundary terms},
\end{array}
\end{equation}

Next we analyze the viscosity damping. As we remarked, starting from the representations in (\ref{Navlerayusubschemepre222}) and in (\ref{Navlerayusubschemepre}) we can estimate viscosity damping based on the Plancherel identity and the fact that the summand
\begin{equation}
\int_{{\mathbb R}^n}D^{\alpha}_xu^{\rho,l,m}_i(m-1,y)p^{l,m}(\sigma,x;m-1,y)dy
\end{equation}
is a convolution with respect to the spatial variables. Indeed a Fourier tansform operation ${\cal F}$ with respect to the spatial variables leads for $0\leq |\alpha|\leq q$ to
\begin{equation}\label{fourier}
{\cal F}\left( D^{\alpha}_xu^{\rho,t_0}_i\right)(\sigma,\xi)={\cal F}\left( D^{\alpha}_xu^{\rho,t_0}_i\right)(0,\xi)\exp\left(-\rho\nu \mu'|\xi|^2(\sigma-t_0)\right).
\end{equation}
For $\sigma-t_0=\Delta'$ the integer $H^q$ norms can be computed by the $L^2$ norms of the expressions in (\ref{fourier}). As the function 
${\cal F}\left(D^{\alpha}_xu^{\rho,t_0}_i(0,.)\right) $ is continuous ,  for each $\epsilon >0$ we find small $a,\Delta '>0$ and a ball $B^n_a$ of radius $a$ around $0$ in ${\mathbb R}^n$ such that
\begin{equation}\label{fourier}
\begin{array}{ll}
{\big |}{\cal F}\left( D^{\alpha}_xu^{\rho,t_0}_i\right)(\Delta',.){\big |}^2_{L^2}\\
\\
\leq \int_{{\mathbb R}^n}{\cal F}\left( D^{\alpha}_xu^{\rho,t_0}_i\right)(0,\xi)^2\left( 1-2\rho\nu \mu'|a|^2\Delta'\right)d\xi\\
\\
+\int_{B^n_a}{\cal F}\left( D^{\alpha}_xu^{\rho,t_0}_i\right)(0,\xi)^22\rho\nu \mu'|\xi|^2\Delta'd\xi+\epsilon\\
\\
\leq {\big |}{\cal F}\left( D^{\alpha}_xu^{\rho,t_0}_i\right)(\Delta',.){\big |}^2_{L^2}\left(1-2\rho\nu \mu'|a|^2\Delta'+
C^a_{max}2\rho\nu \mu'\Delta'\frac{\pi^n}{\Gamma(n/2+1)} a^n\right) +\epsilon
\end{array}
\end{equation}
where 
\begin{equation}
C^{a}_{max}:=\frac{\max_{\xi\in B^n_a}{\cal F}\left( D^{\alpha}_xu^{\rho,t_0}_i\right)(0,\xi)^2}{{\big |}{\cal F}\left( D^{\alpha}_xu^{\rho,t_0}_i\right)(\Delta',.){\big |}^2_{L^2}}
\end{equation}
and where $\Gamma$ is the well-known $\Gamma$-function. Hence for$n\geq 3$ and $a=\Delta^s$ for $s\in \left(\frac{1}{3},\frac{1}{2}\right)$ we can offset growth of order $(\Delta')^2$. However, the growth increment of the nonlinear terms has an upper bound  $C\rho\Delta'$ for some $C>0$, and we get a global scheme if we choose a time step size $\rho=\Delta'$. Hence this growth can be offset by viscosity damping. We summarize
\begin{lem}
For $\Delta'>0$ small enough and the choice $\rho=\Delta'$ and $a=(\Delta')^{s}$ with $s\in \left(  \frac{1}{3},\frac{1}{2}\right) $ viscosity damping offsets the growth of the nonlinear terms in the representations (\ref{Navlerayusubschemepre222}) and in (\ref{Navlerayusubschemepre}). 
\end{lem}

Having analyzed potential damping and viscosity damping it is not difficult to prove
\begin{lem} Given $t_0\geq 0$ 
there exist $\delta >0$ and $\rho >0$ and $C>0$ (all independent of $t_0$) such that (\ref{*eqloc3}) holds. All the quantifiers have a constructive interpretation. For the constructive nes of the $\rho$ quantifier we refer to the arguments which lead to (\ref{pointupper1*}) and the related estimate of the complementary intergal over the ball above in item i).
\end{lem}

\begin{proof}
In the following we use quantifiers in the classical sense, but remark that some additional computations lead to explicit dependences for all the constants used and such that they do no depend on time and only on the parameter $\nu$ and the dimension of the problem. As the explicit computation would make the argument less transparent we shall do these explicit computations elsewhere.  Inductively, at time $t_0\geq 0$ we have
\begin{equation}
u^{\rho,t_0}_i(t_0,.)\in H^2\cap C^2.
\end{equation}
Hence, there exist constants $0<C^{\alpha}_{t_0}<\infty$ such that for $0\leq |\alpha|\leq q$ we have
\begin{equation}
{\big |}D^{\alpha}_xu^{\rho,t_0}_i(t_0,.){\big |}\leq C^{\alpha}
\end{equation}
and such that for any $\epsilon >0$ there exists a $c>0$ such that for the domain
\begin{equation}
D_c:=\left\lbrace y\in {\mathbb R}^n|{\big |}\forall~0\leq |\alpha|\leq q:~{\big |}
D^{\alpha}_xu^{\rho,t_0}_i(t_0,y){\big |}\leq c\right\rbrace 
\end{equation}
and such that
\begin{equation}
\forall~0\leq \alpha|\leq q~\mbox{ we have}~\int_{{\mathbb R}^n\setminus D_c}{\big |}D^{\alpha}_xu^{\rho,t_0}_i(t_0,y){\big |}^2dy \leq \epsilon
\end{equation}
In order to preserve the upper bound it suffices to show that for small $\epsilon >0$ there exist a $\rho,\lambda >0$ such that t
\begin{equation}
\int_{{\mathbb R}^n\setminus D_c}{\big |}D^{\alpha}_xu^{\rho,t_0}_i(t_0+\delta_0,y){\big |}^2dy \leq C-\epsilon,
\end{equation}
where
\begin{equation}
\delta_0:=\frac{\delta}{\sqrt{1-\delta^2}}.
\end{equation}
In this sense and for an appropriate domain $D_c$ we consider $x\in D_c$ and the values $\delta u^{\rho,t_0}_i(\delta_0,x)$ and $\delta D^{\alpha}_xu^{\rho,t_0}_i(\delta_0,x)$ for $0\leq |\alpha|\leq q\geq 2$, where we assume that 
\begin{equation}
{\big |}u^{*,\rho,t_0}_i(0,.){\big |}_{H^q\cap C^q}\leq \frac{C}{\lambda}.
\end{equation}
First for $1\leq |\alpha|\leq q$ consider the representation in (\ref{Navlerayusubschemepre}) for $\sigma =\delta_0$. For the damping term we observe that there is a small $\epsilon >0$ and a small $\rho >0$ and a $\delta \in \left(0,\frac{\epsilon}{1-\epsilon} \right) $ (remember $\delta_0=\frac{\delta}{\sqrt{1-\delta^2}}$) such that
\begin{equation}
\begin{array}{ll}
{\big |}\int_{0}^{\delta_0}\int_{{\mathbb R}^n}\frac{\sqrt{1-\tau^2_l(s)}^3}{1+(\tau(s)-(l-1))} u^{*,\rho,t_0,p}_{i,\alpha}(s,y)G(\sigma,x;s,y)dyds{\big |}\\
\\
\geq  \frac{\sqrt{1-\left((1-\epsilon) \delta\right) ^2}^3}{1+(1-\epsilon)\delta}{\big |}\delta_0\frac{C}{\lambda}-\epsilon  {\big |}=
\frac{1-((1-\epsilon)\delta)^2}{1+(1-\epsilon)\delta}{\big |}\delta\frac{C}{\lambda}-\epsilon  {\big |}>\frac{C}{\lambda}\left(\frac{\delta}{(1+\delta)}\right)=\frac{C}{\lambda}\left(1- \frac{1}{(1+\delta)}\right),
\end{array}
\end{equation}
which is sufficient since we have to show that for $1\leq |\alpha|\leq q$ there exist $\delta _0>0$, $\rho>0$ and $C>0$ such that
\begin{equation}\label{*eqloc33}
\begin{array}{ll}
{\big |} D^{\alpha}_xu^{*,\rho,t_0}_i\left(\delta_0,. \right) ,.){\big |}\leq \frac{C}{\lambda(1+\delta)},
\end{array}
\end{equation}
where for $\alpha=0$ a similar argument holds.
Here, the positive difference 
\begin{equation}
 \frac{1-((1-\epsilon)\delta)^2}{1+(1-\epsilon)\delta}{\big |}\delta\frac{C}{\lambda}-\epsilon  {\big |}-\frac{C}{\lambda}\left(\frac{\delta}{(1+\delta)}\right):=\Delta_0>0
\end{equation}
is large enough such that there exists a $\rho>0$ (possibly smaller then the $\rho$ above)
such that the growth of the term
\begin{equation}\label{Navlerayusubschemepre2**}
\begin{array}{ll}
\delta D^{\alpha}_xu^{*,\rho,t_0}_i(\sigma,x)=\int_{{\mathbb R}^n}u^{*,\rho,t_0}_i(0,y)G(\sigma,x;0,y)dy
-D^{\alpha}_xu^{*,\rho,t_0}(0,x)< \frac{\Delta_0}{2},
\end{array}
\end{equation}
and where for the same $\rho>0$ we can choose a $\lambda >0$ (possible smaller then the $\lambda >0$ chosen before) such that the nonlinear growthterms satisfy
\begin{equation}
\begin{array}{ll}
{\Big |}-\rho\int_{0}^{\sigma}\int_{{\mathbb R}^n}\mu^{*,\tau,2}(s)\sum_{j=1}^n \left( u^{*,\rho,t_0}_j\frac{\partial u^{\rho,t_0}_i}{\partial x_j}\right)_{,\beta} (s,y)G_{,j}(\sigma,x;s,y)dyds\\
\\
+ \rho\int_{0}^{\sigma}\int_{{\mathbb R}^n} \mu^{*,\tau,2}(s)\sum_{j,r=1}^n\int_{{\mathbb R}^n}\left( \frac{\partial}{\partial x_i}K_n(z-y)\right)\times\\
\\
\times \sum_{j,r=1}^n\left( \frac{\partial u^{*,\rho,t_0}_r}{\partial x_j}\frac{\partial u^{*,\rho,t_0}_j}{\partial x_r}\right)_{,\beta} (s,y)G_{,j}(\sigma,x;s,z)dydzds{\Big |}< \frac{\Delta_0}{2}.
\end{array}
\end{equation}
The argument for value function itself is similar, where inside a large radius $R$ around $x$ we have similar representations, and use the strong damping via the Gaussian outside that radius as $\rho >0$ becomes small.
\end{proof}

Finally we mention that the considerations above lead to decreasing maxmum of the value function and to decay to zero at infinite time.

\section{Appendix}

In this appendix we add a short rather self-contained text which proves the existence of global regular upper bounds (in this first section of the appendix we use some notation which can be found in the introduction of the main text).\footnote{We thank Prof. Koch, Bonn, for some questions and critical remarks in an attempt to interpreter an earlier draft of this appendix (informal discussion).  Since the sharper arguments above show that there are upper bounds which are independent of the time horizon, some remarks in the appendix about time scaling may seem to be a bit 'special', but they may be considered to be rather independent considerations. The upshot of the discussion so far concerning this appendix seems to be that the matter of time step size needs to be clarified. For the global time transformation this is indeed an issue, where it is sufficient to show that the time step size $\rho >0$ has no or only a weak dependence on the time horizon $T>0$, for example of order $\rho\sim \frac{1}{T^s}$ for some $s\in (0,1)$. We express this in the update adding the qualification 'not essentially dependent on the time horizon' to the quantifier in Lemma \ref{lem2} and supplement the argument in this respect. Accordingly, in (\ref{rhocontr})) we remark that no transition of time scaling is needed for the argument. The other remarks below and the preliminaries above may also be helpful in order to make the interplay between damping and spatial effects of the operators more transparent. }  

Global existence proofs in this paper are based on the preservation of upper bounds with respect to $H^q\cap C^q$-norms for $q\geq 2$ of the comparison functions $u^{*,\rho,t_0}_i$ at each time step $t_0=l-1$ for $l\geq 1$. In order to prove the existence of global regular upper bounds (which may depend linearly on time) it is sufficient to prove the preservation of upper bounds of a comparison function, i.e., it is sufficient to prove that there exists a positive constant $C_{\lambda}>0$ such that for $t_0\geq 0$ there exists $\delta_0>0$ independent of $t_0$ such that for all $0\leq |\alpha|\leq q\geq 2$ we have
\begin{equation}\label{preservationapp}
\sqrt{\int_{{\mathbb R}^n}{\big |}D^{\alpha}_xu^{*,\rho,t_0}_i(0,y){\big |}^2dy}\leq C_{\lambda}\rightarrow \sqrt{\int_{{\mathbb R}^n}{\big |}D^{\alpha}_xu^{*,\rho,t_0}_i(\delta_0,y){\big |}^2dy}\leq C_{\lambda}.
\end{equation}
In order to prove such a preservation, spatial features of the operators can be used, and, especially the appearance of first order derivatives in the nonlinear Burgers and Leray projection term.
For $1\leq |\alpha|\leq q\geq 2$ and with $\alpha=\beta+1_j$ we can use the representation 
\begin{equation}\label{Navlerayusubschemepreapp}
\begin{array}{ll}
D^{\alpha}_xu^{*,\rho,t_0}_i(\delta_0,x)=\int_{{\mathbb R}^n}D^{\alpha}_xu^{*,\rho,t_0}_i(0,y)G(\delta_0,x;0,y)dy\\
\\
-\int_{0}^{\delta_0}\int_{{\mathbb R}^n}\mu^{*,t_0}(s)  u^{*,\rho,t_0}_{i,\alpha}(s,y)G(\delta_0,x;s,y)dyds\\
\\
-\rho\int_{0}^{\delta_0}\int_{{\mathbb R}^n}\mu^{*,t_0,\tau,2}(s)\sum_{j=1}^n \left( u^{*,\rho,t_0}_j\frac{\partial u^{\rho,t_0}_i}{\partial x_j}\right)_{,\beta} (s,y)G_{,j}(\delta_0,x;s,y)dyds\\
\\
+ \rho\int_{0}^{\delta_0}\int_{{\mathbb R}^n} \mu^{*,t_0,\tau,2}(s)\sum_{j,r=1}^n\int_{{\mathbb R}^n}\left( \frac{\partial}{\partial x_i}K_n(z-y)\right)\times\\
\\
\times \sum_{j,r=1}^n\left( \frac{\partial u^{*,\rho,t_0}_r}{\partial x_j}\frac{\partial u^{*,\rho,t_0}_j}{\partial x_r}\right)_{,\beta} (s,y)G_{,j}(\delta_0,x;s,z)dydzds.
\end{array}
\end{equation}
A similar representation can be used in the case $|\alpha|=0$ (cf. below and the main text). 
Here, the time dependent coefficient 
\begin{equation}\label{nonlinearmu}
\mu^{*,t_0,\tau,2}(s)=\sqrt{1-(\tau (s)-t_0)^2}^3\lambda (1+(\tau(s)-t_0),
\end{equation}
of the nonlinear terms have an additional small parameter $\lambda >0$ (next to a small time step size $\rho >0$), while the coefficient
\begin{equation}
\mu^{*,t_0}(s)=\frac{\sqrt{1-(\tau (s)-t_0)^2}^3}{(1+(\tau(s)-t_0)}.
\end{equation}
of the damping term 
\begin{equation}\label{dampingappend}
-\int_{0}^{\delta_0}\int_{{\mathbb R}^n}\mu^{*,t_0}(s)  u^{*,\rho,t_0}_{i,\alpha}(s,y)G(\delta_0,x;s,y)dyds
\end{equation}
has neither a small step size coefficient $\rho$ nor a small parameter $\lambda >0$. At time $\sigma=0$ in transformed coordinates corresponding to $\tau=t_0$ $\tau$-coordinates we have initial data
\begin{equation}
u^{*,\rho,t_0}_i(0,.)=\frac{v^{\rho}_i(t_0,.)}{\lambda}.
\end{equation}
Observe that even in the case of a local comparison function $u^{*,\rho,t_0}_i,~1\leq i\leq n$ the proof of a preservation as in (\ref{preservationapp}) for small $\lambda >0$ does not lead to a strong contraction property of the Navier Stokes equation but to the preservation of an initial data norm (for the original velocity components) multiplied with the factor $\frac{1}{\lambda}$.
Prima facie small positive $\lambda$ seem to be disadvantageous, because we get the factor $\frac{1}{\lambda^2}$ as a factor of the nonlinear term in the local scheme, and this is compensated only by one factor $\lambda$ of the scaled equation. However, we have another coefficient factor $\rho>0$ of the nonlinear terms which can be chosen to be small (e.g. $\rho\sim \lambda^2$), and in the representation in (\ref{Navlerayusubschemepreapp}) the nonlinear Burgers and Leray projection term have a spatial convolution with the first order spatial derivative of the Gaussian.
Moreover, and in contrast the damping term in (\ref{dampingappend}) is 'convoluted' with the fundamental solution $G$ itself, i,e the fundamental solution of the equation 
of the equation
\begin{equation}
\frac{\partial G}{\partial \tau}-\rho\nu \sqrt{1-(\tau-t_0)^2}^3\Delta G=0.
\end{equation}
For small time $G$ itself becomes close to a $\delta$ distribution, such that the damping term in (\ref{dampingappend}) becomes close to $\mu^{*,t_0}(s)u^{*,\rho,t_0}_{i,\alpha}(0,x)\delta_0$. Note that the damping term does not have a small $\lambda$ or a small $\rho$ factor such that it can compensate the growth of the nonlinear terms as the latter growth terms are spatial convolutions with not with respect to the Gaussian, but with respect to the first order derivative of the Gaussian $G_{,i}$ and have a coefficient $\rho\lambda$ which can be chosen to be small.  
(The 'convolutions' involved are technical convolutions only with respect to the spatial variables -for this reason we put marks around this term). Let is consider the growth of the nonlinear terms more closely. For $\delta_0$ and any $\delta_1\in (0,\delta_0)$ we consider the nonlinear terms in 
(\ref{Navlerayusubschemepreapp}). We write these nonlinear terms in the form
\begin{equation}\label{Navlerayusubschemepreapp2}
\begin{array}{ll}
-\rho\int_{0}^{\delta_1}\int_{{\mathbb R}^n}\mu^{*,t_0,\tau,2}(s)\sum_{j=1}^n \left( u^{*,\rho,t_0}_j\frac{\partial u^{\rho,t_0}_i}{\partial x_j}\right)_{,\beta} (s,y)G_{,j}(\delta_0,x;s,y)dyds\\
\\
+ \rho\int_{0}^{\delta_1}\int_{{\mathbb R}^n} \mu^{*,t_0,\tau,2}(s)\sum_{j,r=1}^n\int_{{\mathbb R}^n}\left( \frac{\partial}{\partial x_i}K_n(z-y)\right)\times\\
\\
\times \sum_{j,r=1}^n\left( \frac{\partial u^{*,\rho,t_0}_r}{\partial x_j}\frac{\partial u^{*,\rho,t_0}_j}{\partial x_r}\right)_{,\beta} (s,y)G_{,j}(\delta_0,x;s,z)dydzds\\
\\
-\rho\int_{\delta_1}^{\delta_0}\int_{{\mathbb R}^n}\mu^{*,t_0,\tau,2}(s)\sum_{j=1}^n \left( u^{*,\rho,t_0}_j\frac{\partial u^{\rho,t_0}_i}{\partial x_j}\right)_{,\beta} (s,y)G_{,j}(\delta_0,x;s,y)dyds\\
\\
+ \rho\int_{\delta_1}^{\delta_0}\int_{{\mathbb R}^n} \mu^{*,t_0,\tau,2}(s)\sum_{j,r=1}^n\int_{{\mathbb R}^n}\left( \frac{\partial}{\partial x_i}K_n(z-y)\right)\times\\
\\
\times \sum_{j,r=1}^n\left( \frac{\partial u^{*,\rho,t_0}_r}{\partial x_j}\frac{\partial u^{*,\rho,t_0}_j}{\partial x_r}\right)_{,\beta} (s,y)G_{,j}(\delta_0,x;s,z)dydzds
\end{array}
\end{equation}
For the latter two summands in (\ref{Navlerayusubschemepreapp2}), i.e., the integrals from $\delta_1$ to $\delta_0$, and  in a ball of radius $R>0$ around $x$ and for $\mu\in \left(\frac{1}{2},1\right)$ we can use the estimate
\begin{equation}\label{gcommaia}
{\big |}G_{,i}(\sigma,x;s,y){\big |}\leq \frac{C}{|\sigma-s|^{\mu}|x-y|^{n+1-2\mu}}
\end{equation}
(for some constant $C>0$ which depends only on the dimension $n$) which becomes small for a small radius $R>0$. Outside a ball of radius $R$ around $x$ we have for some $c>0$
\begin{equation}
{\big |}G_{,i}(\sigma,x;s,y){\big |}\leq \frac{c}{|\sigma-s|^{\frac{n+1}{2}}}\exp\left(-c\frac{R^2}{\rho(\sigma-s)}\right) \downarrow 0
\end{equation}
as $\rho \downarrow 0$. Hence there exists a $\rho >0$ such that the modulus of the two latter summands in (\ref{Navlerayusubschemepreapp2}) have the upper bound $0.5\delta_1$. Hence the modulus of the nonlinear terms in (\ref{Navlerayusubschemepreapp2}) has the upper bound
\begin{equation}\label{Navlerayusubschemepreapp3}
\begin{array}{ll}
0.5\delta_1+{\Big |}\rho\int_{0}^{\delta_1}\int_{{\mathbb R}^n}\mu^{*,t_0,\tau,2}(s)\sum_{j=1}^n \left( u^{*,\rho,t_0}_j\frac{\partial u^{\rho,t_0}_i}{\partial x_j}\right)_{,\beta} (s,y)G_{,j}(\delta_0,x;s,y)dyds{\Big |}\\
\\
+ {\Big |}\rho\int_{0}^{\delta_1}\int_{{\mathbb R}^n} \mu^{*,t_0,\tau,2}(s)\sum_{j,r=1}^n\int_{{\mathbb R}^n}\left( \frac{\partial}{\partial x_i}K_n(z-y)\right)\times\\
\\
\times \sum_{j,r=1}^n\left( \frac{\partial u^{*,\rho,t_0}_r}{\partial x_j}\frac{\partial u^{*,\rho,t_0}_j}{\partial x_r}\right)_{,\beta} (s,y)G_{,j}(\delta_0,x;s,z)dydzds{\Big |}.
\end{array}
\end{equation} 
In order to estimate the latter two terms in (\ref{Navlerayusubschemepreapp3}) and in a ball of a small radius $R$ around $x$ we may use the representation 
\begin{equation}\label{fundreptimea}
\begin{array}{ll}
G_{,i}(\delta_0,x;s,y)=\frac{(x_i-y_i)}{\nu\mu^{\tau,1}(s)(\delta_0-s)}G(\delta_0,x;s,y)\\
\\
+\mbox{lower order singularity terms with respect to time},
\end{array}
\end{equation}
where $\delta-s\geq \delta_0-\delta_1>0$ the factor $(x_i-y_i)$ (resp. the factor $\frac{x_i-y_i}{\delta_0-s}$) changes sign at $x$. As the 'convoluted' functions
\begin{equation}
(s,y)\rightarrow \mu^{*,t_0,\tau,2}(s)\sum_{j=1}^n \left( u^{*,\rho,t_0}_j\frac{\partial u^{\rho,t_0}_i}{\partial x_j}\right)_{,\beta} (s,y)
\end{equation}
and
\begin{equation}
(s,y)\rightarrow \mu^{*,t_0,\tau,2}(s)\sum_{j,r=1}^n\int_{{\mathbb R}^n}\left( \frac{\partial}{\partial x_i}K_n(z-y)\right)
\times \sum_{j,r=1}^n\left( \frac{\partial u^{*,\rho,t_0}_r}{\partial x_j}\frac{\partial u^{*,\rho,t_0}_j}{\partial x_r}\right)_{,\beta} (s,y)
\end{equation}
are bounded continuous in a ball $B_{r_{\delta_1}}(x)$ around $x$ of small radius $r_{\delta_1}$ we can use the representation above to obtain an upper bound $\frac{1}{4}\delta_1$ of 
\begin{equation}\label{Navlerayusubschemepreapp5}
\begin{array}{ll}
{\Big |}\rho\int_{0}^{\delta_1}\int_{B_{r_{\delta_1}}(x)}\mu^{*,t_0,\tau,2}(s)\sum_{j=1}^n \left( u^{*,\rho,t_0}_j\frac{\partial u^{\rho,t_0}_i}{\partial x_j}\right)_{,\beta} (s,y)G_{,j}(\delta_0,x;s,y)dyds{\Big |}\\
\\
+ {\Big |}\rho\int_{0}^{\delta_1}\int_{B_{r_{\delta_1}}(x)} \mu^{*,t_0,\tau,2}(s)\sum_{j,r=1}^n\int_{B_{r_{\delta_1}}(x)}\left( \frac{\partial}{\partial x_i}K_n(z-y)\right)\times\\
\\
\times \sum_{j,r=1}^n\left( \frac{\partial u^{*,\rho,t_0}_r}{\partial x_j}\frac{\partial u^{*,\rho,t_0}_j}{\partial x_r}\right)_{,\beta} (s,y)G_{,j}(\delta_0,x;s,z)dydzds{\Big |}.
\end{array}
\end{equation}  
Outside the ball $B_{r_{\delta_1}}(x)$ we can find $\rho >0$ small such that via the usual Gaussian upper bound of $G$ we get in total the upper bound $0.5\delta_1$ for the last two terms in (\ref{Navlerayusubschemepreapp3}), and in total the upper bound $\delta_1$ for the whole expression in (\ref{Navlerayusubschemepreapp3}), which is an upper bound for the nonlinear terms (of the representation of $D^{\alpha}_xu^{*,\rho,t_0}_i(\sigma,x)$ on the interval $[0,\delta_0]$ in (\ref{Navlerayusubschemepreapp})). We have analyzed the viscosity damping above. It is interesting (especially with regard to the study of inviscid limits) that we can rely even on weaker arguments. Indeed it is sufficient observe that for $\delta_1\in (0,\delta_0)$ we find $\rho>0$ small such that the modulus of
\begin{equation}
\begin{array}{ll}
D^{\alpha}_xu^{*,\rho,t_0}_i(\delta_0,x)-D^{\alpha}_xu^{*,\rho,t_0}_i(0,x)\\
\\
=\int_{{\mathbb R}^n}D^{\alpha}_xu^{*,\rho,t_0}_i(0,y)G(\delta_0,x;0,y)dy-D^{\alpha}_xu^{*,\rho,t_0}_i(0,x)
\end{array}
\end{equation}
has the upper bound $\delta_1$. Now, concerning the damping term we have for small $\delta_0$ and small time step size $\rho>0$ we have
\begin{equation}
-\sqrt{\int_{{\mathbb R}^n}\left[ \int_{0}^{\delta_0}\int_{{\mathbb R}^n}\mu^{*,t_0}(s)  u^{*,\rho,t_0}_{i,\alpha}(s,y)G(\delta_0,x;s,y)dyds\right]^2dx}\leq -(C_{\lambda-\delta_1)} 
\end{equation}
such that this $C_{\lambda}$ is preserved for the choice of small $\delta_1$ with $3\delta_1\leq \delta_0$ which is clearly possible.
For $|\alpha|=0$ we can use the representation  
\begin{equation}\label{Navlerayusubschemepre2app}
\begin{array}{ll}
u^{\rho,t_0}_i(\sigma,x)=\int_{{\mathbb R}^n}u^{\rho,t_0}_i(0,y)G(\sigma,x;0,y)dy\\
\\
-\int_{0}^{\sigma}\int_{{\mathbb R}^n}\mu(s)  u^{\rho,t_0}_{i}(s,y)G(\sigma,x;s,y)dyds\\
\\
-\sum_j\rho\int_{0}^{\sigma}\int_{{\mathbb R}^n}\mu^{\tau,2}(s)\sum_{j=1}^n F_{ij}(\mathbf{u}^{\rho,t_0})(s,y)G_{,j}(\sigma,x;s,y)dyds\\
\\
+ \rho\int_{0}^{\sigma}\int_{B^n_R(x)} \mu^{\tau,2}(s)\sum_{j,r=1}^n\int_{{\mathbb R}^n}\left( K_n(z-y)\right)\times\\
\\
\times \sum_{j,r=1}^n\left( \frac{\partial u^{\rho,t_0}_r}{\partial x_j}\frac{\partial u^{\rho,t_0}_j}{\partial x_r}\right) (s,y)G_{,i}(\sigma,x;s,z)dydzds\\
\\
+ \rho\int_{0}^{\sigma}\int_{{\mathbb R}^n\setminus B^n_R(x)} \mu^{\tau,2}(s)\sum_{j,r=1}^n\int_{{\mathbb R}^n}\left( K_{n,i}(z-y)\right)\times\\
\\
\times \sum_{j,r=1}^n\left( \frac{\partial u^{\rho,t_0}_r}{\partial x_j}\frac{\partial u^{\rho,t_0}_j}{\partial x_r}\right) (s,y)G(\sigma,x;s,z)dydzds\\
\\
+\mbox{boundary terms},
\end{array}
\end{equation}
and as the radius $R>0$ can be chosen to be large the argument for the preservation of an upper bound for the comparison function is analogous. We close this introduction of the appendix with the remark that a change of the time variable $t=\rho \tau$ does not cause any problem with the step size. The following remark is trivial but there has been some confusion about this point. If we consider the Navier Stokes equation in transformed time $\tau =\rho t$ for small $\rho >0$, then the transformed velocity component functions $v^{\rho}_i,\ 1\leq i\leq n$ with $v_i(t,.)=v^{\rho}_i(\tau,.),~1\leq i\leq n$ satisfy the Navier Stokes equation 
\begin{equation}\label{Navleray0}
\left\lbrace \begin{array}{ll}
\frac{\partial v^{\rho}_i}{\partial \tau}-\rho\nu\sum_{j=1}^n \frac{\partial^2 v^{\rho}_i}{\partial x_j^2} 
+\rho\sum_{j=1}^n v^{\rho}_j\frac{\partial v^{\rho}_i}{\partial x_j}=\\
\\ \hspace{1cm}\rho\sum_{j,m=1}^n\int_{{\mathbb R}^n}\left( \frac{\partial}{\partial x_i}K_n(x-y)\right) \sum_{j,m=1}^n\left( \frac{\partial v^{\rho}_m}{\partial x_j}\frac{\partial v^{\rho}_j}{\partial x_m}\right) (t,y)dy,\\
\\
\mathbf{v}^{\rho}(0,.)=\mathbf{h}(.).
\end{array}\right.
\end{equation}
If the parameter $\rho >0$ does not depend essentially on the time horizon $T>0$ where a global regular upper bound for the function $v^{\rho}_i,~1\leq i\leq n$ can be proved, then the obtained upper bound can be transferred to the velocity components $v_i,~1\leq i\leq n$. Indeed the same upper bound holds for the norm
\begin{equation}
\sup_{\tau\in [0,T]}{\big |}v^{\rho}_i(\tau,.){\big |}_{H^q\cap C^q}\leq C\Rightarrow \sup_{t\in [0,\rho T]}{\big |}v_i(\tau,.){\big |}_{H^q\cap C^q}\leq C,
\end{equation}
where the expression '... not depend essentially ...' above may be interpreted as $\rho \gtrsim\frac{1}{T^s}$ for some $s\in (0,1)$. For the local transformation the dependence we have a $\rho$ which is indpendent of the time horizon anyway - so in this sense the question is a bit outdated form the point of view of the stronger arguments above. In the case of a global time transformation (as considered mainly in this appendix). Note that a certain spatial feature of the operator is reflected in the representation of the local solution representations. Especially, for $G_{,i}=\frac{|x-y|}{(\sigma -s)}G^*$, we use the fact that the Gaussian $G$ is fundamental solution of a heat equation with purely time dependent coefficients such we have spatial antisymmetry of first order spatial  derivatives of the Gaussian $G_{,i}$ with respect to the spatial variable $x_i-y_i$, i.e., we may use
\begin{equation}\label{ulmrep*333app}
 \begin{array}{ll}
{\Big |}\rho\int_{m-1}^{\sigma}\int_{B_{\rho^r}(x)}f_i(s,y)\frac{(x_j-y_j)}{\sigma -s}G^*(\sigma,x;s,y)dyds{\Big |}=\\
\\
{\Big |}\rho\int_{m-1}^{\sigma}\int_{B^{x_j\geq y_j}_{\rho^r}(x)}{\big (}f_i(s,y)-f_i(s,y^{-j}){\big )}\frac{(x_j-y_j)}{\sigma -s}G^*\sigma,x;s,y)dyds{\Big |},
\end{array}
\end{equation} 
where 
\begin{equation}
B^{x_j\geq y_j}_{\rho^r}(x)=\left\lbrace y\in B_{\rho^r}(x)|x_j\geq y_j\right\rbrace 
\end{equation}
and 
$y^{-j}=\left(y^{-j}_1,\cdots ,y^{-j}_n\right)$ is the vector with $y^{-j}=(-1)^{\delta_{ij}}$ if $y=(y_1,\cdots ,y_n)$ and $\delta_{ij}$ is the Kronecker $\delta$. 
For each $1\leq j\leq n$ we may use continuity of the 'data' $f_i$, and get
estimate in (\ref{epsilonest}) and in (\ref{pointupper1*}) for $k>0$ ensure that the nonlinear growth terms have upper bounds of order $\sim \rho^{1+\delta_0}$ for some $\delta_0>0$ at least such that we may get global schemes via
\begin{equation}
\rho \gtrsim\frac{1}{T^s} \mbox{ for some }s\in \left(r,1\right) 
\end{equation}
and this ensures that the global regular existence result for $v^{\rho}_i,~1\leq i\leq n$ via a scheme $u^{\rho,l,m}_i,~1\leq i\leq n,~l,m\geq 1$ leads to a global regular existence result for the original velocity component functions $v_i,~1\leq i\leq n$. Having said this we can leave the appendix almost unchanged with the only modification that the quantifier for $\rho$ can be interpreted in a constructive way relying on the arguments which lead to (\ref{epsilonest}) and in (\ref{pointupper1*}).

\subsection{Statement of the linear upper bound theorem}

We are concerned with the incompressible Navier Stokes equation Cauchy problem in Leray projection form
\begin{equation}\label{Navleray}
\left\lbrace \begin{array}{ll}
\frac{\partial v_i}{\partial t}-\nu\sum_{j=1}^n \frac{\partial^2 v_i}{\partial x_j^2} 
+\sum_{j=1}^n v_j\frac{\partial v_i}{\partial x_j}=\\
\\ \hspace{1cm}\sum_{j,m=1}^n\int_{{\mathbb R}^n}\left( \frac{\partial}{\partial x_i}K_n(x-y)\right) \sum_{j,m=1}^n\left( \frac{\partial v_m}{\partial x_j}\frac{\partial v_j}{\partial x_m}\right) (t,y)dy,\\
\\
\mathbf{v}(0,.)=\mathbf{h},
\end{array}\right.
\end{equation}
to be solved for $\mathbf{v}=\left(v_1,\cdots ,v_n \right)^T$ on the domain $\left[0,\infty \right)\times {\mathbb R}^n$. The constant $\nu >0$ is the viscosity and $K_n$ is the Laplacian kernel of dimension $n$, and $\mathbf{h}=\left(h_1,\cdots ,h_n \right)^T$ are the initial data.  
We have
\begin{thm}\label{mthm}
We are interested in $n\geq 3$. We assume $\nu >0$ and that for all $1\leq i\leq n$ we have $h_i\in H^m\cap C^m$ for $m\geq 2$, where $H^m$ denotes the Sobolev space of order $m$ and $C^m$ denotes the space of classical differentiable functions with continuous derivatives up to order $m$ as usual. Then we claim that there is a global regular solution $v_i,~1\leq i\leq n$ of (\ref{Navleray}) with
\begin{equation}\label{ex}
v_i\in C^1\left(\left(0,\infty \right), H^m\cap C^m\right)\cap C^0\left(\left[0,\infty \right), H^m\cap C^m\right),
\end{equation}
and such that there exists a constant $C>0$ with
\begin{equation}\label{eq3}
{\big |}v_i(t,.){\big |}_{H^m}\leq C+Ct
\end{equation}
for all $t\geq 0$.
\end{thm}

\begin{rem}\label{thmrem}
For the global existence of a regular solution $v_i,~1\leq i\leq n$ with (\ref{ex}) it suffices to prove that
\begin{equation}\label{eq3rem1}
\forall T>0~\exists C>0~\forall 0\leq t\leq T:~{\big |}v_i(t,.){\big |}_{H^m}\leq C+Ct
\end{equation}
This is a weak interpretation of the statement in (\ref{eq3}). A stronger interpretation 
is to claim that 
there exists a constant $C>0$ such that for all $T>0$ we have
\begin{equation}\label{eq3rem2}
\exists C>0~\forall T>0~\forall 0\leq t\leq T:~{\big |}v_i(t,.){\big |}_{H^m}\leq C+Ct
\end{equation}
Both claims can be proved, but the first claim is certainly easier to prove and sufficient for global existence.
\end{rem}

\section{Proof of the theorem}

We mention that the method is quite general and uniqueness is not claimed by this method (therefore the formulation 'a' solution in the theorem above).  

\begin{lem}\label{lem1}(existence of local solutions)
For data $h_i\in H^m\cap C^m$ for all $1\leq i\leq n$, and with $m\geq 2$ there exists a small time horizon $\rho>0$ such that there is a local solution $v_i:[0,\rho]\times {\mathbb R}^n\rightarrow {\mathbb R},~1\leq i\leq n$ of the initial value problem (\ref{Navleray}), and such that
\begin{equation}
v_i\in C^1\left(\left(0,\rho\right],H^m\cap C^m\right)\cap C^0\left(\left[0,\rho \right], H^m\cap C^m\right)~\mbox{ for }1\leq i\leq n.
\end{equation}
\end{lem}
The lemma can be proved by a local contraction argument. It is clear that the lemma \ref{lem1} may be applied to the equation in (\ref{Navleray}) with data $v_i(t_0,.)\in H^m\cap C^m$ if we have such data at time $t_0\geq 0$. For convenience of a later transformation let a time transformed function $v^{\rho,t_0}_i,~1\leq i\leq n$ have the time dependence $\tau$ with
\begin{equation}\label{eq5}
t-t_0=\rho(\tau-t_0),~\mbox{ where }v^{\rho,t_0}_i(\tau,.)=v^{t_0}_i(t,.)
\end{equation}
for $t\in \left[t_0,t_0+\rho\right]$ and $\tau\in \left[t_0,t_0+1\right]$. Here $v^{t_0}_i$ is just the restriction of $v_i$ to the interval $[t_0,t_0+\rho]$. Especially we have $v^{t_0}_i(t_0,.)=v_i(t_0,.)$. This function $v^{\rho,t_0}_i,~1\leq i\leq n$ is a (hypothetical) solution of 
\begin{equation}\label{Navleray2a}
\left\lbrace \begin{array}{ll}
\frac{\partial v^{\rho,t_0}_i}{\partial \tau}-\rho\nu\sum_{j=1}^n \frac{\partial^2 v^{\rho,t_0}_i}{\partial x_j^2} 
+\rho\sum_{j=1}^n v^{\rho,t_0}_j\frac{\partial v^{\rho,t_0}_i}{\partial x_j}=\\
\\ \hspace{1cm}\rho\sum_{j,m=1}^n\int_{{\mathbb R}^n}\left( \frac{\partial}{\partial x_i}K_n(x-y)\right) \sum_{j,m=1}^n\left( \frac{\partial v^{\rho,t_0}_m}{\partial x_j}\frac{\partial v^{\rho,t_0}_j}{\partial x_m}\right) (\tau,y)dy,\\
\\
\mathbf{v}^{\rho,t_0}(t_0,.)=\mathbf{v}(t_0,.),
\end{array}\right.
\end{equation}
on the domain $[t_0,t_0+1]\times {\mathbb R}^n$, and where time $\tau\in [t_0,t_0+1]$. Now by lemma \ref{lem1} for some $t_0>0$ we have $v^{\rho,t_0}_i(t_0,.)\in H^m\cap C^m,~1\leq i\leq n$ for $m\geq 2$ if $h_i\in H^m\cap C^m$. In order to observe local growth we compare for given $t_0\geq 0$ and for a $1\leq i\leq n$ the function $v^{\rho,t_0}_i$ at times $\tau\in [t_0,t_0+1)$ with a function $u^{\rho,t_0}_i$ which are related by
\begin{equation}\label{uv0}
(1+\tau )u^{\rho,t_0}_i(\sigma,x)=v^{\rho,t_0}_i(\tau ,x),
\end{equation}
where 
\begin{equation}\label{sigma}
\sigma=\frac{\tau-t_0}{\sqrt{1-(\tau-t_0)^2}}.
\end{equation}
Note that $\tau\in [t_0,t_0+1)$ corresponds to $\sigma\in [0,\infty)$. If it is convenient we compare both functions also on finite horizons, where for some $c_e\in (0,1)$ we consider $v^{\rho,t_0}_i(\tau,.)$ for $\tau\in [t_0,t_0+c_e]$ and compare with $u^{\rho,t_0}_i(\sigma,.)$ for $\sigma \in [0,c_{\sigma_e}]$, and where $c_e$ and $c_{\sigma_e}$ are related by (\ref{sigma}), i.e.,
\begin{equation}
c_{\sigma_e}=\frac{c_e-t_0}{\sqrt{1-(c_e-t_0)^2}}.
\end{equation}

\begin{rem}
We mention that the global factor in (\ref{uv0}) may be localized in the sense that
\begin{equation}\label{uv*}
(1+(\tau -t_0) )u^{\rho,t_0}_i(\sigma,x)=v^{\rho,t_0}_i(\tau ,x)
\end{equation}
is a variation of a scheme. This may simplify some steps as the coefficients $\mu^{\tau,k};~k=1,2$ are then bounded for all $\tau$, but we shall see that (\ref{uv0}) is also sufficient in order to prove global regular existence.
\end{rem}
The following argument can be performed in various forms- the argument by the function $u^{\rho,t_0}_i,~1\leq i\leq n$ defined in (\ref{uv0}) is just one possibility. Alternatively, in order avoid trivial confusions concerning the step sizes $\rho >0$, we consider a similar argument for functions without upperscript $\rho$ defining for $\lambda >0$ the functions $u^{t_0}_i,~1\leq i\leq n$ by
\begin{equation}\label{lambdau}
\lambda(1+t)u^{t_0}_i(s,.)=v^{t_0}_i,~s=\frac{t-t_0}{\sqrt{1-(t-t_0)^2}},
\end{equation}
where $\lambda>0$ is small. In the latter case the problem for $v^{t_0}_i,~1\leq i\leq n$, i.e., a problem of form (\ref{Navleray}) with data $v^{t_0}_i(t_0,.)=v_i(t_0,.)$ transforms then to
\begin{equation}\label{Navlerayuut0}
\left\lbrace \begin{array}{ll}
\frac{\partial u^{t_0}_i}{\partial \sigma}-\mu^{t,1}\nu\sum_{j=1}^n \frac{\partial^2 u^{t_0}_i}{\partial x_j^2} 
+\mu^{t,2}\sum_{j=1}^n u^{t_0}_j\frac{\partial u^{t_0}_i}{\partial x_j}+\mu u^{t_0}_i=\\
\\  \mu^{t,2}\sum_{j,r=1}^n\int_{{\mathbb R}^n}\left( \frac{\partial}{\partial x_i}K_n(x-y)\right) \sum_{j,r=1}^n\left( \frac{\partial u^{t_0}_r}{\partial x_j}\frac{\partial u^{t_0}_j}{\partial x_r}\right) (\sigma,y)dy,\\
\\
\mathbf{u}^{
t_0}(0,.)=\frac{1}{\lambda(1+t_{0})}\mathbf{v}^{t_0}(t_0,.),
\end{array}\right.
\end{equation}
where $\mu^{t,1}=\sqrt{1-(t-t_0)^2}^3$, $\mu=\frac{\sqrt{1-(t-t_0)^2}^3}{1+t}$, and  $\mu^{t,2}=\lambda(1+t)\sqrt{1-(t-t_0)^2}^3$. You observe that for small $\lambda$ the damping term is dominant for a small time horizon. If we want to have the existence of the constant $C>0$ independent of the time horizon $T>0$ we can replace $\lambda(1+t)$ in (\ref{lambdau}) above by $\lambda(1+(t-t_0))$. This also leads to the opportunity to define a constructive version of following the argument which we consider in remark \ref{2833} below. This leads to several independent proofs of the claim, where we observe also some interconnections. In this respect note that we have for some $c^{\lambda}>0$ and all $0\leq s\leq c^{\lambda}$ and corresponding $\sigma\in [0,c^{\sigma}_e]$ that
$$\lambda(1+t)u^{t_0}_i(s,.)=v^{t_0}_i(t,.)=v^{\rho,t_0}_i(\tau,.)=(1+\tau)u^{\rho,t_0}_i(\sigma,.).$$
In the following we may use the abbreviation $u^{\lambda,t_0}_i,~1\leq i\leq n$, where for all $s\in [0,c^{\lambda}]$ and corresponding $\sigma\in [0,c^{\sigma}_e]$ $$u^{\lambda,t_0}_i(s,.):=\lambda\frac{1+t}{1+\tau}u^{t_0}_i(s,.)=u^{\rho,t_0}_i(\sigma,.).$$ 
Here, recall from above that $\sigma\in [0,c^{\sigma}_e]$ corresponds to $\tau\in [t_0,t_0+c_e]$ and to $t\in [t_0,t_0+\rho c_e]$ by the relation $t-t_0=\rho(\tau-t_0)$.
It is clear that for $\rho >0$ small enough we have $\rho c^{\sigma}_e\leq c^{\lambda}$ such that in the following we may assume that
for all $s\in [0,\rho c^{\sigma}_e]$ and corresponding $\sigma\in [0,c^{\sigma}_e]$ $$u^{\lambda,t_0}_i(s,.)=u^{\rho,t_0}_i(\sigma,.).$$ 
We denote the inverse of the time dilatation by $\tau\equiv \tau(\sigma)$ (easily computed explicitly). Note that
\begin{equation}
\frac{\partial}{\partial \tau}v^{\rho,l}_i(\tau ,x)=u^{\rho,t_0}_i(\sigma,x)+(1+\tau)\frac{\partial}{\partial \sigma}u^{\rho,t_0}_i(\sigma,x)\frac{d \sigma}{d \tau},
\end{equation}
where for $\tau\in \left[ t_{0},t_0+1\right) $ we have
\begin{equation}
\begin{array}{ll}
\frac{d\sigma}{d \tau}=\frac{1}{\sqrt{1-\left( \tau-t_{0}\right) ^2}^3}.
\end{array}
\end{equation}
The equation for $u^{\rho,t_0}_i,~1\leq i\leq n$ is of the form
\begin{equation}\label{Navlerayu2}
\left\lbrace \begin{array}{ll}
\frac{\partial u^{\rho,t_0}_i}{\partial \sigma}-\rho\mu^{\tau,1}\nu\sum_{j=1}^n \frac{\partial^2 u^{\rho,t_0}_i}{\partial x_j^2} 
+\rho\mu^{\tau,2}\sum_{j=1}^n u^{\rho,t_0}_j\frac{\partial u^{\rho,t_0}_i}{\partial x_j}+\mu u^{\rho,t_0}_i=\\
\\ \rho \mu^{\tau,2}\sum_{j,r=1}^n\int_{{\mathbb R}^n}\left( \frac{\partial}{\partial x_i}K_n(x-y)\right) \sum_{j,r=1}^n\left( \frac{\partial u^{\rho,t_0}_r}{\partial x_j}\frac{\partial u^{\rho,t_0}_j}{\partial x_r}\right) (\sigma,y)dy,\\
\\
\mathbf{u}^{\rho,
t_0}(0,.)=\frac{1}{1+t_{0}}\mathbf{v}^{\rho,t_0}(t_0,.),
\end{array}\right.
\end{equation}
where for $k\in \left\lbrace 0,1,2 \right\rbrace $
\begin{equation}\label{mu}
\mu =\mu(\sigma)=\frac{\sqrt{1-\tau^2_{t_0}(\sigma)}^3}{1+\tau(\sigma)},~\mu^{\tau, k}:=(1+\tau(\sigma))^k\mu.
\end{equation}
Here we indicate that under the squareroot we have the shifted $\tau_{t_0}=\tau-t_0$. Note that in a localized version we have a shifted $\tau$ everywhere, but this is not needed for an analytical proof.
Note that $\sigma\in [0,\infty)$ as $\tau\in \left[t_{0},t_0+1\right)$. Next we assume that for some $m\geq 2$ and some $t_0\geq 0$ we have
\begin{equation}\label{vbound}
{\big |}v^{\rho,t_0}_i(t_0,.){\big |}_{H^m\cap C^m}\leq \left(1+t_0\right) C
\end{equation}
for some constant $C>0$. This means that for some constant $C>0$ we have 
\begin{equation}\label{ubound}
{\big |}u^{\rho,t_0}_i(0,.){\big |}_{H^m\cap C^m}\leq  C.
\end{equation}
Now the equation in (\ref{Navlerayu2}) has a strong damping term $\mu u^{\rho,t_0}_i$ (without a possibly small $\rho$), while all spatial derivative terms in (\ref{Navlerayu2}) have this factor $\rho$.    This leads to a preservation of an upper bound by $u^{\rho,t_0}$ if $\rho>0$ is small enough. 
\begin{rem}
Note that it may look as if we consider only a time effect but we use some spatial structure of the operator here. However, representations as in the proof of Lemma \ref{lem2} show the spatial effects. Higher order derivatives in the convolution representations may always involve a first order spatial derivative of the Gaussian. This is obvious for representations of the derivatives of the value functions. However, as all the nonlinear terms involve at least one first order derivative the value function itself may be represented in terms of convolutions involving Gaussian first order derivatives.
\end{rem} 
The following Lemma \ref{lem2} is formulated in a rather strong fashion (stronger than needed in order to prove the existence of regular linear upper bounds).
\begin{lem}\label{lem2}
Assume that $t_0\geq 0$ and $\tilde{\rho}>0$ are such that Lemma \ref{lem1} holds for $v^{\tilde{\rho},t_0}_i$ with $1\leq i\leq n$, and such that (\ref{vbound}) holds with $\tilde{\rho}$. Then there is some $0<\rho\leq \tilde{\rho}$ (which does not depend essentially on the time horizon $T>0$) such that
\begin{equation}\label{claimlemma2}
{\big |}u^{\rho,t_0}_i(0,.){\big |}_{H^m\cap C^m}\leq  C\Rightarrow {\big |}u^{\rho,t_0}_i(\sigma,.){\big |}_{H^m\cap C^m}\leq  C \mbox{ for all $\sigma \geq 0$.}
\end{equation}
\end{lem}
\begin{rem}\label{loc}
Clearly if (\ref{vbound}) holds for $\tilde{\rho}$ then it holds for $\rho\leq \tilde{\rho}$ and has the consequence (\ref{ubound}) for this $\rho>0$. Furthermore it is sufficient to prove Lemma \ref{lem2} for a finite interval of times $\sigma\in [0,T]$ for some $T\geq 1$ for example. It is also sufficient and easier to prove Lemma \ref{lem2} for discrete times $\sigma$  but the  sharper result can be obtained also.
\end{rem}
\begin{rem}
The phrase '$\rho$ does not depend essentially on the time horizon $T$' means $\rho \gtrsim \frac{1}{T^s}$ foe some $s\in (0,1)$.
\end{rem}

Now for contradiction assume that there is a maximal $t_{max}$ such that a solution $v_i$ with $v_i\in C^1\left( (0,t_{max}]\times H^m\cap C^m\right) \cap C^0\left( [0,t_{max}]\times H^m\cap C^m\right)$ for $1\leq i\leq n$ with a linear upper bound as in (\ref{eq3}) exists, i.e., assume that $t_{max}>0$ is maximal such that
\begin{equation}\label{vboundorgm}
{\big |}v_i
(t,.){\big |}_{H^m\cap C^m}\leq \left(1+t\right) C~\mbox{ for all }t\leq t_{max}.
\end{equation}
for some constant $C>0$ (which depends only on $\nu,n,h_i,~1\leq i\leq n$ and an arbitrary horizon $T>t_{max}$).  
According to our definitions above we have 
\begin{equation}
v_i(t_{max},.)=v^{t_{max}}_i(t_{\max},.)=v^{\rho,t_{max}}_i(t_{max},.)
\end{equation}
 such that for
(\ref{vboundorgm}) we have
\begin{equation}\label{vboundm}
{\big |}v^{\rho,t_{max}}_i
(t_{max},.){\big |}_{H^m\cap C^m}\leq \left(1+t_{max}\right) C
\end{equation}
for all $1\leq i\leq n$ for the same constant $C>0$. According to (\ref{uv0}) and (\ref{sigma}) with $t_0=t_{max}=\tau$ this implies that
\begin{equation}\label{uvtmax}
 {\big |}u^{\rho,t_{\max}}_i(0,.){\big |}_{H^m\cap C^m}
=\frac{{\big |}v^{\rho,t_{max}}_i(t_{max} ,.){\big |}_{H^m\cap C^m}}{1+t_{max}}\leq C.
 \end{equation}
Then according to Lemma \ref{lem1} (applied to data $v^{t_{max}}_i(t_{max},.)=v^{\rho,t_{max}}_i(t_{max},.)\in H^m\cap C^m$ for all $1\leq i\leq n$) and 
Lemma \ref{lem2} (resp. Lemma \ref{lem2*}) there exists a $\rho>0$ and a constant $0<c_e< 1$ such that 
$v^{\rho,t_{max}}_i(\tau,.)\in H^m\cap C^m$ exists for $\tau\in [t_{max},t_{max}+c_e]$ for some constant $c_e>0$ and such that
\begin{equation}\label{uboundm}
{\big |}u^{\rho,t_{max}}_i(\sigma,.){\big |}_{H^m\cap C^m}\leq  C~(\mbox{resp.}~{\big |}u^{t_{max}}_i(s,.){\big |}_{H^m\cap C^m}\leq  C)
\end{equation}
for all $\sigma\in \left[0,c_{\sigma_e} \right] $ (resp. $s\in \left[0,\rho c_{\sigma_e} \right] $ ), and where $c_{\sigma_e}$ is the constant in $\sigma$-coordinates corresponding to the constant $c_e$ in $\tau$-coordinates. 
Now given $t_0\geq 0$ consider the function $u^{t_0}_i,~1\leq i\leq n$, where for all $\sigma \in [0,c_{\sigma_e}]$ and corresponding $s\in [0,\rho c_{\sigma_e}]$ and all $x\in {\mathbb R}^n$ we have (according to the definitions of $u^{t_0}_i,u^{\lambda,t_0}_i$ above)
\begin{equation}\label{relussigma}
 u^{\lambda, t_0}_i(s,x)=u^{\rho,t_0}_i(\sigma,x).
\end{equation}
Now (\ref{uboundm}) and (\ref{relussigma}) for $t_0=t_{max}$ implies that for all $s\in [0,\rho c_{\sigma_e}]$ and corresponding $\sigma \in [0,c^{\sigma}_e]$ we have
\begin{equation}\label{uboundmt}
{\big |}u^{\lambda,t_{max}}_i(s,.){\big |}_{H^m\cap C^m}={\big |}u^{\rho,t_{max}}_i(\sigma,.){\big |}_{H^m\cap C^m}\leq  C.
\end{equation}
There are several variations of arguments which are independent or partial independent. Note that the first equation in (\ref{relussigma}) is only needed if we use only lemma \ref{lem2} and intend to transfer results for functions in $\sigma$- and $\tau$-coordinates (i.e. results functions with upperscript $\rho$) to functions in $s$ and $t$-coordinates. The simplest variation uses just Lemma \ref{lem2*} and then the right inequality in (\ref{uboundm}). Let us consider this variation of argument first.
Then from the right inequality of (\ref{uboundm}) we have for $s\in \left[0,c^e_{\sigma} \right]$ and corresponding $t\in \left[t_{max},t_{max}+\rho c_e \right]$ that
\begin{equation}
{\big |}v^{\lambda,t_{max}}_i(t,.){\big |}_{H^m\cap C^m}=(1+t) {\big |}u^{\lambda,t_{max}}_i(s,.){\big |}_{H^m\cap C^m}\leq (1+t)C,
\end{equation}
such that the statement in (\ref{eq3}) of the main theorem is proved. 
Alternatively, if we want to use transfer results for functions with upperscript $\rho$ to functions with $s$ and $t$ coordinates, then we start indeed with (\ref{uboundmt}).
First, note that for $s=0$ we have $t=\tau=t_{max}$ and
\begin{equation}
u^{\lambda,t_{max}}_i(0,.):=\lambda\frac{1+t_{max}+(t-t_{max})|_{t=t_{max}}}{1+t_{max}+(\tau-t_{max})|_{t=t_{max}}}u^{t_{max}}_i(0,.)=u^{t_{max}}_i(0,.),
\end{equation}
which implies that for $C'=C/\lambda$ we have 
\begin{equation}
{\big |}u^{t_{max}}_i(0,.){\big |}_{H^m\cap C^m}\leq C'.
\end{equation}
In this case we may apply the relation (\ref{claimlemma2}) in Lemma 2.8. with $C'=C/\lambda$ and have certainly for $s\in [0,c^{\sigma}_e]$ that
\begin{equation}
{\big |}u^{t_{max}}_i(s,.){\big |}_{H^m\cap C^m}\leq C'.
\end{equation}
Hence for $s\in [0,\rho c^{\sigma}_e]$ and corresponding $t\in \left[ t_{max},t_{max}+\rho c_e\right] $ we have
\begin{equation}
C(1+t)=\lambda (1+t)C'\geq \lambda (1+t){\big |}u^{t_{max}}_i(s,.){\big |}_{H^m\cap C^m}={\big |}v^{t_{max}}_i(t,.){\big |}_{H^m\cap C^m}.
\end{equation}
This variation of argument uses the relation (\ref{relussigma}) only for $s=0=\sigma$, where further information is available in remark \ref{2425}.  A third variation is to use the other inequality of  (\ref{relussigma}), i.e. the upper bound for $u^{\rho,t_{max}}_i(s,.)$ with respect to the $H^m\cap C^m$-norm. This latter variation   is considered in remark \ref{uvtt} below.
All these variations imply that there exists some $C>0$ such that
\begin{equation}\label{vtmax}
{\big |}v^{t_{max}}_i(t,.){\big |}_{H^m\cap C^m}\leq \left(1+t\right) C
\end{equation}
for all $t_{max}\leq t\leq t_{max}+\rho c_e$. Then (\ref{vboundorgm}) together with (\ref{vtmax}) lead to a contradiction to the maximality of $t_{max}$ in (\ref{vboundorgm}).

Further remarks concerning the transition from (\ref{uboundm}) or (\ref{uboundmt}) to (\ref{vtmax}) may be helpful: we explain a) in remark \ref{uvtt} the transition from (\ref{uboundmt}) to (\ref{vtmax}) in the weak interpretation of (\ref{eq3rem1}) (which is sufficient for global regular existence), and b) in remark \ref{2425} we add further remarks for the scheme for $u^{t_0}_i$. Finally c) we consider in remark \ref{2833} a proof for a variation of a scheme which leads also to a constructive version of the proof.    
\begin{rem}\label{uvtt}
In order to prove the weaker statement of remark \ref{thmrem} that for a solution $v_i,~1\leq i\leq n$ of the incompressible Navier Stokes equation we have
\begin{equation}\label{eq3rem1*}
\forall T>0~\exists C>0~\forall 0\leq t\leq T:~{\big |}v_i(t,.){\big |}_{H^m}\leq C+Ct
\end{equation}
we consider again the statement (\ref{uv0}) for $t_0=t_{max}$. More precisely we consider the weaker statement that for some $0<\rho <1$ and some $c_e\in (0,1)$ and $\tau\in \left[t_{max},t_{max}+c_e \right]$ and corrresponding $\sigma\in \left[0,c_{\sigma_e}\right]$ we have
\begin{equation}\label{tausigmaobs}
\left(1+\tau \right)u^{\rho,t_{max}}_i(\sigma,.)=v^{\rho,t_{max}}_i(\tau,.).
\end{equation}
Note that for $\tau=t_{max}$ and corresponding $\sigma=0$ we have
\begin{equation}
\left(1+t_{max} \right)u^{\rho,t_{max}}_i(0,.)=v^{\rho,t_{max}}_i(t_{max},.),
\end{equation}
which implies that for some $C>0$ we have
\begin{equation}
{\big |}u^{\rho,t_{\max}}_i(0,.){\big |}_{H^m\cap C^m}=\frac{1}{1+t_{max}}{\big |}v^{t_{\max}}_i(t_{max},.){\big |}_{H^m\cap C^m}\leq C.
\end{equation}

For this $0<\rho <1$ (and also all smaller $\rho >0$) using $\rho (\tau-t_{max})=t-t_{max}$ we get some $c_e\in (0,1)$ and $t\in \left[t_{max},t_{max}+\rho c_e \right]$ and corresponding $\sigma\in \left[0,c_{\sigma_e}\right]$ we have
\begin{equation}\label{locrel}
\left(1+t_{max}+\frac{t-t_{max}}{\rho} \right)u^{\rho,t_{max}}_i(\sigma,.)=v^{\rho,t_{max}}_i\left( t_{max}+\frac{t-t_{max}}{\rho},.\right).
\end{equation}
Hence there is no simple transition to (\ref{vtmax}) with the same $C>0$. From (\ref{locrel}) and the definition of the function $v^{t_{max}}_i$ we then have that for some $0<\rho <1$ that for all $\sigma \in \left[0,c_{\sigma_e}\right]$ and corresponding  $t\in \left[t_{max},t_{max}+\rho c_e \right]$ we have 
\begin{equation}\label{locrel***}
\begin{array}{ll}
\left(1+t \right)\frac{1}{\rho}{\big |}u^{\rho,t_{max}}_i(\sigma,.){\big |}_{H^m\cap C^m}\geq {\Big |}v^{\rho,t_{max}}_i\left( t_{max}+\frac{t-t_{max}}{\rho},.\right){\Big |}_{H^m\cap C^m}\\
\\
={\big |}v^{t_{\max}}_i(t,.){\big |}_{H^m\cap C^m}.
\end{array}
\end{equation}

The statement at equation (\ref{uboundm}) above holds for a $\rho >0$ which depends only on $C,n,\nu$ and $T$, i.e., for given $T>0$ there exists a $\rho>0$ which depends only on $C,n,\nu$ and $T$ such that for all $\sigma\in \left[0,c_{\sigma_e}\right]$ 
\begin{equation}
{\big |}u^{\rho,t_{max}}_i(0,.){\big |}_{H^m\cap C^m}\leq C\Rightarrow {\big |}u^{\rho,t_{max}}_i(\sigma,.){\big |}_{H^m\cap C^m}\leq C.
\end{equation}
Hence, this together with \ref{locrel***} implies that for 
$t\in \left[t_{max},t_{max}+\rho c_e \right]$ and corresponding $\sigma\in[0,c^{\sigma}_e]$ and $C'=C/\rho$ we have 
\begin{equation}\label{locrel**2}
\begin{array}{ll}
C'(1+t)\geq \left(1+t \right)\frac{1}{\rho}{\big |}u^{\rho,t_{max}}_i(\sigma,.){\big |}_{H^m\cap C^m}\geq {\Big |}v^{\rho,t_{max}}_i\left( t_{max}+\frac{t-t_{max}}{\rho},.\right){\Big |}_{H^m\cap C^m}\\
\\
={\big |}v^{t_{\max}}_i(t,.){\big |}_{H^m\cap C^m}.
\end{array}
\end{equation}
Hence the statement in (\ref{eq3}) of the main theorem holds with the constant $C'$. Here we use that $\rho>0$ can be chosen such that it depends on the time horizon $T>0$ but is independent of $t_{max}$. Hence we get (\ref{eq3rem1*}).

\end{rem}

\begin{rem}\label{2425}
We next explain that for the transition from (\ref{uboundmt}) to (\ref{vtmax}) we do not need the introduction of functions with upperscript $\rho>0$ but can introduce a small parameter $\tilde{\rho}$ in the nonlinear transformation alternatively. The parameter $\lambda >0$ is not essential in the following (but may be used for numerical purposes). Let $v^{t_0}_i$ solve
\begin{equation}\label{Navleray2aztmax}
\left\lbrace \begin{array}{ll}
\frac{\partial v^{t_{0}}_i}{\partial t}-\nu\sum_{j=1}^n \frac{\partial^2 v^{t_{0}}_i}{\partial x_j^2} 
+\sum_{j=1}^n v^{t_{0}}_j\frac{\partial v^{t_{0}}_i}{\partial x_j}=\\
\\ \hspace{1cm}\sum_{j,m=1}^n\int_{{\mathbb R}^n}\left( \frac{\partial}{\partial x_i}K_n(x-y)\right) \sum_{j,m=1}^n\left( \frac{\partial v^{t_{0}}_m}{\partial x_j}\frac{\partial v^{t_{0}}_j}{\partial x_m}\right) (t,y)dy,\\
\\
\mathbf{v}^{t_{0}}(t_{0},.)=\mathbf{v}(t_{0},.).
\end{array}\right.
\end{equation}
Define $u^{t_0}_i,~1\leq i\leq n$ for a given constant $\lambda >0$ via
\begin{equation}\label{complambda}
\lambda(1+t)u^{t_0}_i(s,.)=v^{t_0}_i(t,.)~\mbox{with}~s=\tilde{\rho}\frac{t-t_0}{\sqrt{1-(t-t_0)^2}}
\end{equation}
Then we have
\begin{equation}
\frac{\partial v^{t_0}}{\partial t}=\lambda u^{t_0}_i(s,.)+\lambda (1+t) u^{t_0}_{i,s}\frac{\tilde{\rho}}{\sqrt{1-(t-t_0)^2}^3}
\end{equation}
where the last factor is just $\frac{ds}{dt}$. 
The nonlinear terms then get another factor $\lambda$ while the damping term and the Laplacian are independent of $\lambda$, i.e., the function $u^{t_0}_i,~1\leq i\leq n$ satisfies the equation
\begin{equation}\label{Navlerayu22*}
\left\lbrace \begin{array}{ll}
\frac{\partial u^{t_{0}}_i}{\partial s}-\mu^{t,1}\nu\sum_{j=1}^n \frac{\partial^2 u^{t_{0}}_i}{\partial x_j^2} 
+\lambda\mu^{t,2}\sum_{j=1}^n u^{t_{0}}_j\frac{\partial u^{\lambda,t_{0}}_i}{\partial x_j}+\mu u^{t_{0}}_i=\\
\\ \lambda\mu^{t,2}\sum_{j,r=1}^n\int_{{\mathbb R}^n}\left( \frac{\partial}{\partial x_i}K_n(x-y)\right) \sum_{j,r=1}^n\left( \frac{\partial u^{t_0}_r}{\partial x_j}\frac{\partial u^{t_0}_j}{\partial x_r}\right) (s,y)dy,\\
\\
\mathbf{u}^{
t_0}(0,.)=\frac{1}{\lambda(1+t_{0})}\mathbf{v}^{t_{0}}(t_{0},.),
\end{array}\right.
\end{equation}
where
\begin{equation}\label{mu}
\mu=\frac{\sqrt{1-(t-t_0)^2}^3}{\tilde{\rho}(1+t)},~
\mu^{t,1} =\sqrt{1-(t-t_{0})^2}
\end{equation}
and
\begin{equation}\label{mut2}
\mu^{t,2} =\left(1+t\right)  \sqrt{1-(t-t_{0})^2}^3.
\end{equation}
Next assume that $t_{max}>0$ is maximal such that
\begin{equation}
\forall t\leq t_{max}:~{\big |}v^{t_{max}}_i(t,.){\big |}_{H^m\cap C^m}\leq C+Ct.
\end{equation}
Then from local comparison (\ref{complambda}) with $t_0=t_{max}$ we have 
\begin{equation}
\lambda (1+t_{max}){\big |}u^{t_{max}}_i(0,.){\big |}_{H^m\cap C^m}={\big |}v^{t_{max}}_i(t_{max},.){\big |}_{H^m\cap C^m}\leq C+C t_{max},
\end{equation}
where we have some freedom to choose $\lambda >0$. For some $\lambda >0$ (small enough and dependent on $C>0$) we get a Lemma analogous to Lemma \ref{lem2} in the sense that
\begin{equation}
{\big |}u^{t_{max}}_i(0,.){\big |}_{H^m\cap C^m}\leq \frac{C}{\lambda}\Rightarrow {\big |}u^{t_{max}}_i(s,.){\big |}_{H^m\cap C^m}\leq \frac{C}{\lambda}~\mbox{for }s\geq 0.
\end{equation}
Hence for some $c_{s_e}>0$ and $s\in \left[0,c_{s_e} \right]$ and corresponding $c_e>0$ and $t\in \left[t_{max},t_{max}+c_e \right]$ we have
\begin{equation}
{\big |}v^{t_{max}}_i(t,.){\big |}_{H^m\cap C^m}=\lambda (1+t){\big |}u^{t_{max}}_i(t,.){\big |}_{H^m\cap C^m}\leq \lambda (1+t)\frac{C}{\lambda}=C+C t.
\end{equation}
Then (\ref{vboundorgm}) together with (\ref{vtmax}) lead to a contradiction to the maximality of $t_{max}$ in (\ref{vboundorgm}). The analogous lemma could be formulated as follows 
\begin{lem}\label{lem2*}
Assume that $t_0\geq 0$ and $\tilde{\rho}>0$ are such that Lemma \ref{lem1} holds for $v^{t_0}_i$ with $1\leq i\leq n$ on the domain $[t_0,t_0+\tilde{\rho})$, and such that \begin{equation}\label{vbound*}
{\big |}v^{t_0}_i(t_0,.){\big |}_{H^m\cap C^m}\leq \left(1+t_0\right) C
\end{equation} 
holds. Then there is some $0<\rho\leq \tilde{\rho}$ and a $\lambda >0$ depending only on $\nu,n$ and $C>0$ such that for $s\in [0,\rho c^{\sigma}_e]$ (corresponding to $t\in [t_0,t_0+\rho)$) and $C'=C/\lambda$ we have
\begin{equation}\label{claimlemma2}
{\big |}u^{t_0}_i(0,.){\big |}_{H^m\cap C^m}\leq  C'\Rightarrow {\big |}u^{t_0}_i(s,.){\big |}_{H^m\cap C^m}\leq  C'.
\end{equation}
\end{lem}
\end{rem}

\begin{rem}\label{2833}
A localized version of the local comparison function leads to a stronger version of the the main statement of theorem \ref{mthm}, where the quantifier for $C>0$ does not depend on the time horizon $T>0$. Furthermore, this version allows for a constructive interpretation. 
Define $u^{t_0}_i,~1\leq i\leq n$ for a given constant $\lambda >0$ via
\begin{equation}\label{complambda}
\lambda(1+(t-t_0))u^{t_0}_i(s,.)=v^{t_0}_i(t,.)~\mbox{with}~s=\frac{t-t_0}{\sqrt{1-(t-t_0)^2}}
\end{equation}
Then we have
\begin{equation}
\frac{\partial v^{t_0}}{\partial t}=\lambda u^{t_0}_i(s,.)+\lambda (1+(t-t_0)) u^{t_0}_{i,s}\frac{1}{\sqrt{1-(t-t_0)^2}^3},
\end{equation}
and the function $u^{t_0}_i,~1\leq i\leq n$ satisfies the equation
\begin{equation}\label{Navlerayu22*}
\left\lbrace \begin{array}{ll}
\frac{\partial u^{t_{0}}_i}{\partial s}-\mu^{t,1}_{loc}\nu\sum_{j=1}^n \frac{\partial^2 u^{t_{0}}_i}{\partial x_j^2} 
+\lambda\mu^{t,2}_{loc}\sum_{j=1}^n u^{t_{0}}_j\frac{\partial u^{\lambda,t_{0}}_i}{\partial x_j}+\mu_{loc} u^{t_{0}}_i=\\
\\ \lambda\mu^{t,2}_{loc}\sum_{j,r=1}^n\int_{{\mathbb R}^n}\left( \frac{\partial}{\partial x_i}K_n(x-y)\right) \sum_{j,r=1}^n\left( \frac{\partial u^{t_0}_r}{\partial x_j}\frac{\partial u^{t_0}_j}{\partial x_r}\right) (s,y)dy,\\
\\
\mathbf{u}^{
t_0}(0,.)=\frac{1}{\lambda(1+t_{0})}\mathbf{v}^{t_{0}}(t_{0},.),
\end{array}\right.
\end{equation}
where
\begin{equation}\label{mu}
\mu_{loc}=\frac{\sqrt{1-(t-t_0)^2}^3}{1+(t-t_0)},~
\mu^{t,1}_{loc} =\sqrt{1-(t-t_{0})^2}
\end{equation}
and
\begin{equation}\label{mut2}
\mu^{t,2}_{loc} =\left(1+(t-t_0)\right)  \sqrt{1-(t-t_{0})^2}^3.
\end{equation}
\end{rem}

\begin{rem}\label{rhocontr}
The contradiction argument can be applied to $v^{\rho}_i,~1\leq i\leq n$ directly as $v^{\rho}_i(\tau,.),.)=v_i(t,.)$ for all $1\leq i\leq n$ implies that the global regular result can be transferred directly from the function $v^{\rho}_i, 1\leq i\leq n$ to the function the function $v_i, 1\leq i\leq n$. For contradiction assume that there is a maximal $t_{\mbox{max}}$ such that a solution $v_i$ with $v^{\rho}_i\in C^1\left( (0,t_{max}]\times H^m\cap C^m\right) \cap C^0\left( [0,t_{max}]\times H^m\cap C^m\right)$ for $1\leq i\leq n$ with a linear upper bound as in (\ref{eq3}) exists, i.e., assume that $t_{max}>0$ is maximal such that
\begin{equation}\label{vboundorgm*}
{\big |}v^{\rho}_i
(\tau,.){\big |}_{H^m\cap C^m}\leq \left(1+\tau\right) C~\mbox{ for all }\tau\leq t_{max}.
\end{equation}
for some constant $C>0$ (which depends only on $\nu,n,h_i,~1\leq i\leq n$ and an arbitrary horizon $T>t_{max}$,
 such that for
(\ref{vboundorgm}) we have
\begin{equation}\label{vboundm}
{\big |}v^{\rho,t_{max}}_i
(t_{max},.){\big |}_{H^m\cap C^m}\leq \left(1+t_{max}\right) C
\end{equation}
for all $1\leq i\leq n$ for the same constant $C>0$. According to (\ref{uv0}) and (\ref{sigma}) with $t_0=t_{max}=\tau$ this implies that
\begin{equation}\label{uvtmax}
 {\big |}u^{\rho,t_{\max}}_i(0,.){\big |}_{H^m\cap C^m}
=\frac{{\big |}v^{\rho,t_{max}}_i(t_{max} ,.){\big |}_{H^m\cap C^m}}{1+t_{max}}\leq C.
 \end{equation}
Then according to Lemma \ref{lem1} there exists a $\rho>0$ and a constant $0<c_e< 1$ such that 
$v^{\rho,t_{max}}_i(\tau,.)\in H^m\cap C^m$ exists for $\tau\in [t_{max},t_{max}+c_e]$ for some constant $c_e>0$ and such that
\begin{equation}\label{uboundm}
{\big |}u^{\rho,t_{max}}_i(\sigma,.){\big |}_{H^m\cap C^m}\leq  C
\end{equation}
for all $\sigma\in \left[0,c_{\sigma_e} \right] $, and where $c_{\sigma_e}$ is the constant in $\sigma$-coordinates corresponding to the constant $c_e$ in $\tau$-coordinates. 

Hence for $s\in [0,\rho c^{\sigma}_e]$ and corresponding $\tau\in \left[ t_{max},t_{max}+\rho c_e\right] $ we have
\begin{equation}
C(1+\tau)\geq (1+\tau)C\geq {\big |}u^{t_{max}}_i(s,.){\big |}_{H^m\cap C^m}={\big |}v^{\rho,t_{max}}_i(\tau,.){\big |}_{H^m\cap C^m},
\end{equation}
in contradiction to the maximality of $t_{max}$ and $v^{\rho}_i(\tau,.)=v^{\rho,t_{max}}_i$ for $\tau\in \left[t_{max},t_{max}+c_e\right]$
\end{rem}

\begin{rem}
The relation (\ref{vboundm}) implies that from (\ref{uv0}) and (\ref{sigma}) we have for $t_0=t_{max}=\tau$
\begin{equation}\label{uv0*}
{\big |}u^{\rho,t_{max}}_i(0,.){\big |}_{H^q\cap C^q}={\Bigg |}\frac{v^{\rho,t_{max}}_i(t_{max} ,.)}{1+t_{max}}{\Bigg |}_{H^q\cap C^q}\leq C,
\end{equation}
Then from Lemma \ref{lem2} we have that
\begin{equation}\label{claimlem2}
{\big |}u^{\rho,t_{max}}_i(\sigma,.){\big |}_{H^q\cap C^q}\leq C,
\end{equation}
where the Lemma \ref{lem2} claims strongly that this holds for all $\sigma\geq 0$. For the (abstract) construction of a global solution it is sufficient to have (\ref{claimlem2}) for some appropriate discrete sets of $\sigma$s. Indeed some properties of the operator are needed to get the stronger claim.
\end{rem}

\begin{rem}\label{abstr}
In the preceding claim the term 'abstract construction' means that the argument presented here provides no numerical scheme. We interpret the existence quantifier in an abstract sense - similarly as the use of the existence quantifier in a Weierstrass $\epsilon$-$\delta$ definition of continuity. From the point of view of classical mathematics this perfectly legitimate and even popular as such a form of a use of an existence quantifier is to be found usually in 'proofs by contradiction'. Furthermore, similarly as for uniformly continuous functions in the Weierstrass definition for all $\epsilon >0$ there exists a $\delta$ depending on $\epsilon$ but not on the argument of the continuous function, we may use the existence quantor of $\rho>0$ for certain variations of auto-controlled schemes.
\end{rem}

\begin{rem}\label{vvrho}
In the sentence before (\ref{vboundm}) we assume a maximal $t_{\mbox{max}}$ for the original solution $v_i$ with $v_i\in C^1\left( [0,t_{max}]\times H^m\cap C^m\right) $. Note that $v^{\rho,t_{max}}_i(t_{max},.)=v_i(t_{max},.)$ by definition. 
\end{rem}

\subsection{Proof Lemma \ref{lem2} (weaker forms)}

As we mentioned we do not need  lemma \ref{lem2} in its full force in order to prove \ref{eq3}. First it suffices to prove it on a finite time interval for $\sigma\geq 0$ for example we may have $\tau-t_0\leq \frac{1}{2}$ and, correspondingly, $\sigma\in \left[0,\frac{1}{\sqrt{3}} \right]$.

An essential rephrasing of lemma \ref{lem2} is then as follows: given $t_0\geq 0$ and for $\rho>0$ small enough we have
\begin{equation}\label{cl1}
{\big |}u^{\rho,t_0}_i(0,.){\big |}_{H^m\cap C^m}\leq  C\Rightarrow {\big |}u^{\rho,t_0}_i(\sigma,.){\big |}_{H^m\cap C^m}\leq  C \mbox{ for all $\sigma \in \left[ 0,\frac{1}{\sqrt{3}}\right] $.}
\end{equation}

In order to prove (\ref{eq3}) we can even consider a weaker formulation than (\ref{cl1}).
For this purpose, we consider a time discretization $\sigma_p=p\Delta,~1\leq p\leq N$  of the interval $\left[0,\frac{1}{\sqrt{3}} \right]$ with a constant $\Delta >0$ and such that $N\Delta =\frac{1}{\sqrt{3}}$.

We mention that for our purpose it is sufficient (and essential) to prove  
\begin{equation}\label{discrete}
{\big |}u^{\rho,t_0}_i(0,.){\big |}_{H^m\cap C^m}\leq  C\Rightarrow {\big |}u^{\rho,t_0}_i(\sigma_p,.){\big |}_{H^m\cap C^m}\leq  C \mbox{ for all $p \leq N $},
\end{equation}
but the slightly stronger claim (\ref{cl1}) and the stronger statement of lemma \ref{lem2}  can also be proved. We shall prove the claim (\ref{discrete}) first and then show that he claim (\ref{eq3}) follows even from this weak form of Lemma \ref{lem2}. Then we show that the stronger claims (\ref{cl1}) and even th stronger statement in lemma \ref{lem2} can be proved. 

For $1\leq p\leq N-1$ we consider a subscheme of functions $u^{\rho,t_0,p}_i:\left[\sigma_{p-1},\sigma_{p-1}+\Delta \right]\rightarrow {\mathbb R},~1\leq i\leq n$ which solve the Cauchy problems
\begin{equation}\label{Navlerayusubscheme}
\left\lbrace \begin{array}{ll}
\frac{\partial u^{\rho,t_0,p}_i}{\partial \sigma}-\rho\mu^{\tau,1}\nu\sum_{j=1}^n \frac{\partial^2 u^{\rho,t_0,p}_i}{\partial x_j^2} 
+\rho\mu^{\tau,2}\sum_{j=1}^n u^{\rho,t_0,p}_j\frac{\partial u^{\rho,t_0,p}_i}{\partial x_j}+\mu u^{\rho,t_0,p}_i=\\
\\ \rho \mu^{\tau,2}\sum_{j,r=1}^n\int_{{\mathbb R}^n}\left( \frac{\partial}{\partial x_i}K_n(x-y)\right) \sum_{j,r=1}^n\left( \frac{\partial u^{\rho,t_0,p}_r}{\partial x_j}\frac{\partial u^{\rho,t_0,p}_j}{\partial x_r}\right) (\sigma,y)dy,\\
\\
\mathbf{u}^{\rho,
t_0,p}((p-1)\Delta,.)=\mathbf{u}^{\rho,
t_0,p-1}((p-1)\Delta,.),
\end{array}\right.
\end{equation}
where the initial data at each time step $m$ are given recursively from the previous time step and where for $p=1$ we have  $\mathbf{u}^{\rho,t_0,1}(0,.)=\frac{1}{1+t_{0}}\mathbf{v}^{\rho,t_0}(t_0,.)$.
Assuming inductively that for all $1\leq i\leq n$ and some $m\geq 2$ we have data
\begin{equation}
{\big |}u^{\rho,t_0,p-1}_i((p-1)\Delta,.){\big |}_{H^m\cap C^m}\leq C
\end{equation}
by a local contraction argument it follows that for some $\Delta>0$ we have
\begin{itemize}
 \item[i)] the local solution $\mathbf{u}^{\rho,
t_0,p}=\left(u^{\rho,t_0,p}_1,\cdots,u^{\rho,t_0,p}_n\right)$ of (\ref{Navlerayusubscheme}) exists, where $$u^{\rho,t_0,p}_i\in C^0\left([\sigma_{p-1},\sigma_{p-1}+\Delta],H^m\cap C^m\right)\cap C^1\left((\sigma_{p-1},\sigma_{p-1}+\Delta],H^m\cap C^m\right)$$
\item[ii)] for $\Delta >0$ small the increment of the solution $\delta u^{\rho,t_0,p}_i(\sigma,.):=u^{\rho,t_0,p}_i(\sigma,.)-u^{\rho,t_0,p}_i((p-1)\Delta,.)$ becomes small, i.e., for all $1\leq i\leq n$
\begin{equation}\label{increment}
\sup_{\sigma\in [(p-1)\Delta ,p\Delta]}{\big |}\delta u^{\rho,t_0,p}_i(\sigma,.){\big |}_{H^m\cap C^m}\leq c
\end{equation}
where such time step size $\Delta$ and increment upper bound $c>0$ depends only on $n,\rho,C,\nu,K_{n},\mu,\mu^{\tau,k}$ for $0\leq k\leq 2$. We remark that the constant $c$ goes to zero as $\Delta$ goes to zero where this relation is not linear but H\"{o}lder. Moreover, if there is such a pair $\Delta$ and $c$ with (\ref{increment}) for a given $\tilde{\rho}$ the same relation (\ref{increment}) holds for the same pair $\Delta$ and $c$ for any smaller $\rho\leq \tilde{\rho}$. 
\end{itemize}

We note that i) follows from the local existence lemma \ref{lem1} via the relation (\ref{uv0}). We note that that the local contraction result could be proved for $u^{\rho,t,p}_i,~1\leq i\leq n$ satisfying (\ref{Navlerayusubscheme}) for some $\rho>0$ itself, where the existence of the damping term facilitates the proof. However, as it is sufficient to have lemma \ref{lem1} we do not need to prove item i) here. Note that in our specific situation lemma \ref{lem1} implies the existence of $u^{\rho,t,p}_i,~1\leq i\leq n$ for all time steps $p\geq 1$ at once.

Hence we may concentrate on item ii).
As the functions  $$u^{\rho,t_0,p}_i\in C^0\left(\left[(p-1)\Delta,p\Delta \right],H^m\cap C^m  \right) ~1\leq i\leq n$$ exist we may consider classical representations of these functions in terms of the fundamental solution $G_p$ of the equation
\begin{equation}
\frac{\partial G_p}{\partial \sigma}-\rho\mu^{\tau,1}\nu\sum_{j=1}^n \frac{\partial^2 G_p}{\partial x_j^2}=0, 
\end{equation}
where the subscript $p$ indicates that this is with respect to the domain $$\left[ \left(p-1\right)\Delta,p\Delta\right]\times {\mathbb R}^n.$$ In the following we denote by
$p_{m,j}$ the first order spatial derivative of the fundamental solution with respect to the variable $x_j$. Note that a regular  fundamental solution $G_p$ exists by classical methods as we have bounded regular coefficients. For convenience we use the notation $D^{\alpha}_xf$ of $f_{,\alpha}$ and $(fg)_{,\alpha}$ for multivariate spatial derivatives with multiindex $\alpha$ for single functions and some products of functions respectively.
Using convolution rules with respect to the spatial variables for multivariate derivatives $D^{\alpha}_x$ with multiindices $\alpha=(\alpha_1,\cdots,\alpha_n)=\beta+1_j:=(\beta_1,\cdots,\beta_j+1,\cdots,\beta_n)$ for $0\leq |\alpha|\leq m$ we have the functional increment representation
\begin{equation}\label{Navlerayusubschemerep}
\begin{array}{ll}
D^{\alpha}_xu^{\rho,t_0,p}_i(\sigma,x)-\int_{{\mathbb R}^n}D^{\alpha}_xu^{\rho,t_0,p-1}_i((p-1)\Delta,y)G_{p}(\sigma,x;(p-1)\Delta,y)dy\\
\\
=-\int_{(p-1)\Delta}^{\sigma}\int_{{\mathbb R}^n}\mu(s)  u^{\rho,t_0,p}_{i,\alpha}(s,y)G_{p}(\sigma,x;s,y)dyds\\
\\
-\rho\int_{(p-1)\Delta}^{\sigma}\int_{{\mathbb R}^n}\mu^{\tau,2}(s)\sum_{j=1}^n \left( u^{\rho,t_0,p}_j\frac{\partial u^{\rho,t_0,p}_i}{\partial x_j}\right)_{,\beta} (s,y)G_{p,j}(\sigma,x;s,y)dyds\\
\\
+ \rho\int_{(p-1)\Delta}^{\sigma}\int_{{\mathbb R}^n} \mu^{\tau,2}(s)\sum_{j,r=1}^n\int_{{\mathbb R}^n}\left( \frac{\partial}{\partial x_i}K_n(z-y)\right)\times\\
\\
\times \sum_{j,r=1}^n\left( \frac{\partial u^{\rho,t_0,p}_r}{\partial x_j}\frac{\partial u^{\rho,t_0,p}_j}{\partial x_r}\right)_{,\beta} (s,y)G_{p,j}(\sigma,x;s,z)dydzds.
\end{array}
\end{equation}
Note that $G_{p,j}$ denotes the spatial derivative of first order of the fundamental solution $G_p$ with respect to the spatial variable $x_j$.

Next we have
\begin{lem}\label{lem3}
Assume that some $t_0\geq 0$ is given. Assume furthermore that $v^{t_0}_i(t_0,.)=v^{\rho,t_0}_i(t_0,.)\in H^m\cap C^m$ such that Lemma \ref{lem1} holds for $v^{t_0}_i(t_0,.),~1\leq i\leq n$ and a related statement holds for $v^{\rho,t_0}_i(\tau,.),~1\leq i\leq n$ hold for some $\tilde{\rho}>0$. Especially assume that 
\begin{equation}
{\big |}v^{t_0}_i(t_0,.){\big |}_{H^m\cap C^m}={\big |}v^{\tilde{\rho},t_0}_i(t_0,.){\big |}_{H^m\cap C^m}\leq C+Ct_0
\end{equation}
for some $m\geq 2$ for some $C>0$ such that
\begin{equation}
{\big |}u^{\tilde{\rho},t_0}_i(0,.){\big |}_{H^m\cap C^m}\leq C
\end{equation}
for the same $m\geq 2$. Then there is a triple of small positive real numbers $\Delta,c,\rho$ with $0<\rho\leq \tilde{\rho}$ such with respect to the discretization considered above we have that for all $1\leq p\leq N$
\begin{equation}\label{preserv}
{\big |}u^{\rho,t_0}_i(\sigma_{p-1},.){\big |}_{H^m\cap C^m}\leq  C\Rightarrow {\big |}u^{\rho,t_0}_i(\sigma_p,.){\big |}_{H^m\cap C^m}\leq  C,
\end{equation}
holds.
\end{lem} 

Some remarks are in order

\begin{rem}
Note that a) the triple $\Delta,c,\rho$ of small positive numbers can be chosen independently of the time step number $p\leq N$ as $C>0$ is preserved and that b) if  $\Delta,c,\rho$ is such a triple with the property (\ref{preserv}), then any triple $\Delta,c,\rho'$ with $0<\rho'\leq \rho$ is also a triple with the property (\ref{preserv}).
\end{rem}

\begin{rem}
The analysis below shows that the slightly stronger claim (\ref{cl1}) can also be proved by proving it for each substep with a triple $(\Delta,c,\rho)$. Note that this is a claim for finite time $0\leq \sigma\leq \frac{1}{\sqrt{3}}$. As we remarked above, the claim (\ref{cl1}) or the weaker claim (\ref{discrete}) are both indeed sufficient for a proof of existence of a global regular solution for the Navier Stokes equation and Lemma \ref{lem2} could be reformulated using one of them. However, we shall show that even the statement (\ref{claimlemma2}) in lemma \ref{lem2} holds. In the next section we add a remark which is useful in order observe that the stronger claims (\ref{cl1}) and (\ref{claimlemma2}) hold.  The formulation in lemma \ref{lem2} is still a bit stronger, and it is easier to prove it for the claim using the localized transformation in mentioned in remark \ref{loc}. If we use the global factor $(1+\tau)$ we have linear growth of $\mu^{\tau,2}$. This can be offset by using a variable time step size $\rho_p\sim \frac{1}{p}$ for the subscheme $u^{\rho,l,p}_i,~1\leq i\leq n$. However, we do not even need this: as the time $\tau\in [t_0,t_0+\rho)$ is local, the coefficients
$\mu^{\tau,2}$ (and a fortiori the coefficients $\mu^{\tau,k},~k=0,1$) are bounded for all $\sigma \geq 0$. 
 \end{rem}

\begin{proof}
(lemma \ref{lem3}). We assume that $t_0\geq 0$ and $\tilde{\rho}$ are given as in Lemma \ref{lem3} such that
\begin{equation}\label{preserv2}
{\big |}u^{\tilde{\rho},t_0}_i(\sigma_{p-1},.){\big |}_{H^m\cap C^m}\leq  C
\end{equation}
for some $C>0$ and some $p\geq 1$ and all $1\leq i\leq n$. We want to show this preservation of the upper bound $C$ directly proving local properties of the functions $u^{\rho,t_0}_i$. It sufficed to prove
\begin{equation}\label{preserv3}
{\big |}u^{\tilde{\rho},t_0}_i(\sigma_{p},.){\big |}_{H^m\cap C^m}\leq  C,
\end{equation}
i.e., to prove the claim of Lemma \ref{lem3} for one time step, as the claim (\ref{preserv}) follows then immediately by induction. For $0<\rho\leq \tilde{\rho}$ we know that the function $u^{\rho,t_0,p}_i,~1\leq i\leq n$ satisfies (\ref{Navlerayusubscheme}) and may be solved by a local scheme $u^{\rho,t_0,p,q}_i,~1\leq i\leq n,~p\geq 1$ with
\begin{equation}
u^{\rho,t_0,p}_i=\lim_{q\uparrow \infty}u^{\rho,t_0,p,q}_i
\end{equation}
for all $1\leq i\leq n$,
where for $q=0$ we define $u^{\rho,t_0,p,0}_i(t_0,.):=u^{\rho,t_0,p-1}_i(t_0,.)$ for $1\leq i\leq n$ and for $q\geq 1$ the function $u^{\rho,t_0,p,q}_i,~1\leq i\leq n$ is determined as the solution of the Cauchy problem
\begin{equation}\label{Navlerayusubschemesub}
\left\lbrace \begin{array}{ll}
\frac{\partial u^{\rho,t_0,p,q}_i}{\partial \sigma}-\rho\mu^{\tau,1}\nu\sum_{j=1}^n \frac{\partial^2 u^{\rho,t_0,p,q}_i}{\partial x_j^2} 
+\rho\mu^{\tau,2}\sum_{j=1}^n u^{\rho,t_0,p,q-1}_j\frac{\partial u^{\rho,t_0,p,q-1}_i}{\partial x_j}\\
\\
-\mu u^{\rho,t_0,p,q-1}_i+ \rho \mu^{\tau,2}\sum_{j,r=1}^n\int_{{\mathbb R}^n}\left( \frac{\partial}{\partial x_i}K_n(x-y)\right)\times\\
\\
\times \sum_{j,r=1}^n\left( \frac{\partial u^{\rho,t_0,p,q-1}_r}{\partial x_j}\frac{\partial u^{\rho,t_0,p,q-1}_j}{\partial x_r}\right) (\sigma,y)dy,\\
\\
\mathbf{u}^{\rho,
t_0,p,q}((p-1)\Delta,.)=\mathbf{u}^{\rho,
t_0,p-1}((p-1)\Delta,.).
\end{array}\right.
\end{equation}
We now assume inductively that we have the upper bound
\begin{equation}
\max_{1\leq i\leq n}\sup_{s\in [\sigma_{p-1},\sigma_p]}{\Big |}u^{\rho,t_0,p,q-1}_i(s,.){\Big |}_{H^m\cap C^m}\leq C,
\end{equation}
which holds for the time indpendent function $u^{\rho,t_0,p,0}_i,~1\leq i\leq n$ by definition.
Similarly as in (\ref{Navlerayusubschemerep}) we get for each member of $u^{\rho,t_0,p,q}_i,~1\leq i\leq n,~p\geq 1$ for $0\leq |\alpha|\leq m$ and the relation of $\beta+1_j=\alpha$ as considered before (\ref{Navlerayusubschemerep}) above the representation
\begin{equation}\label{Navlerayusubschemesubrep}
\begin{array}{ll}
D^{\alpha}_xu^{\rho,t_0,p,q}_i(\sigma,x)-\int_{{\mathbb R}^n}D^{\alpha}_xu^{\rho,t_0,p-1}_i((p-1)\Delta,y)G_{p}(\sigma,x;(p-1)\Delta,y)dy\\
\\
=-\int_{(p-1)\Delta}^{\sigma}\int_{{\mathbb R}^n}\mu(s)  u^{\rho,t_0,p,q-1}_{i,\alpha}(s,y)G_{p}(\sigma,x;s,y)dyds\\
\\
-\rho\int_{(p-1)\Delta}^{\sigma}\int_{{\mathbb R}^n}\mu^{\tau,2}(s)\sum_{j=1}^n \left( u^{\rho,t_0,p,q-1}_j\frac{\partial u^{\rho,t_0,p,q-1}_i}{\partial x_j}\right)_{,\beta} (s,y)G_{p,j}(\sigma,x;s,y)dyds\\
\\
+ \rho\int_{(p-1)\Delta}^{\sigma}\int_{{\mathbb R}^n} \mu^{\tau,2}(s)\sum_{j,r=1}^n\int_{{\mathbb R}^n}\left( \frac{\partial}{\partial x_i}K_n(z-y)\right)\times\\
\\
\times \sum_{j,r=1}^n\left( \frac{\partial u^{\rho,t_0,p,q-1}_r}{\partial x_j}\frac{\partial u^{\rho,t_0,p,q-1}_j}{\partial x_r}\right)_{,\beta} (s,y)G_{p,j}(\sigma,x;s,z)dydzds.
\end{array}
\end{equation}
At this point we note that $t_0\leq t< t_0+\rho$. 
This corresponds to $t_0\leq \tau <t_0+1$. 
As we have a factor $\sqrt{1-(\tau-t_0)^2}^3$ 
in all coefficients $\mu,\mu^{\tau,2}$ we can extend the following argument to $\tau\in [t_0,t_0+1)$, but -as we have remarked before it is sufficient to consider the horizon $t\in [t_0,t_0+\rho c_e]$ corresponding to $\tau\in [t_0,t_0+c_e]$ for some $c_e\in (0,1)$. For notational convenience we set $c_e=\frac{1}{2}$. This corresponds to a horizon $\sigma \in \left[0,\frac{1}{\sqrt{3}}\right]$ with respect to $\sigma$.      Note that for all $p\geq 1$ we have $\sigma_p=p\Delta\in \left[0,\frac{1}{\sqrt{3}}\right]$). Going back to (\ref{mu}) above on the considered interval we have the lower bound
 \begin{equation}\label{mu*}
 \inf_{\sigma\in \left[0,\frac{1}{\sqrt{3}}\right] }{\big |}\mu(\sigma){\big |}=\inf_{\tau\in \left[t_0,t_0+\frac{1}{2} \right] }{\Bigg |}\frac{\sqrt{1-\tau^2_{t_0}(\sigma)}^3}{1+\tau(\sigma)}{\Bigg |}\geq \frac{3\sqrt{3}}{12+8\tau_0 }=c_{\mu,t_0}
\end{equation}
and the upper bound
\begin{equation}\label{mu**}
 \sup_{\sigma\in \left[0,\frac{1}{\sqrt{3}}\right] }\mu(\sigma)=\sup_{\tau\in[t_0,t_0+0.5]}\frac{\sqrt{1-\tau^2_{t_0}(\sigma)}^3}{1+\tau(\sigma)}(1+\tau(\sigma))^2\leq \left(\frac{3}{2}+t_0 \right)=: c^{\mu}_{t_0},
\end{equation}
where we add the subscript $t_0$ in order to indicate the (linear) dependence of this upper bound on $t_0\geq 0$. Note that this upper bound holds for given finite $t_0\geq 0$ for all $\sigma\in [0,\infty)$. However, at this moment we consider only $\sigma\in [0,\sigma_{c_e}]$ for $c_e=0.5$.  
Now, as $\mu(s)>c_{\mu,t_0}$ for $s\in [(p-1)\Delta,p\Delta)=[\sigma_{p-1},\sigma_p)$ there exists a $0<\rho'\leq \tilde{\rho}$ such that
\begin{equation}\label{potinf}
\begin{array}{ll}
-{\Big |}\int_{(p-1)\Delta}^{p\Delta }\int_{{\mathbb R}^n}\mu(s)  u^{\rho',t_0,p,q-1}_{i,\alpha}(s,y)G_{p}(p\Delta,.;s,y)dyds{\Big |}_{H^m\cap C^m}\\
\\
\leq -c_{\mu,t_0}\left(\sigma_p-\sigma_{p-1}\right)^{\alpha} {\big |}u^{\rho',t_0,p,q-1}_{i,\alpha}(\sigma_{p-1},.){\big |}_{H^m\cap C^m}.
\end{array}
\end{equation}
for $\alpha \in (0,1)$ (you may expect that the latter statement holds up to an $\epsilon>0$ but as $\mu(s)>c_{\mu,t_0}$ it holds indeed as stated).
Note that the relation in (\ref{potinf}) holds for $0<\rho\leq \rho'$ if it holds for $\rho'$ such that we may assume that it holds for generic $\rho>0$ ( consider $\min\left\lbrace \rho,\rho' \right\rbrace$ with the choices so far and call this $\rho>0$ again). Then we have (recall that $\Delta =(\sigma_p-\sigma_{p-1})$)
\begin{equation}\label{aa}
\begin{array}{ll}
{\big |}u^{\rho,t_0,p,q}_i(p\Delta,.){\big |}_{H^m\cap C^m}\\
\\
\leq {\big |}\int_{{\mathbb R}^n}u^{\rho,t_0,p-1}_i((p-1)\Delta,y)G_{p}(\sigma_p,.;(p-1)\Delta,y)dy{\big |}_{H^m\cap C^m}\\
\\
-c_{\mu,t_0} (\sigma_p-\sigma_{p-1})^{\alpha}{\big |} u^{\rho,t_0,p,q-1}_{i}(\sigma_{p-1},.){\big |}_{H^m\cap C^m}\\
\\
+\rho c^{\mu,t_0}{\Big |}\int_{(p-1)\Delta}^{p\Delta}\int_{{\mathbb R}^n}\sum_{j=1}^n \left( u^{\rho,t_0,p,q-1}_j\frac{\partial u^{\rho,t_0,p,q-1}_i}{\partial x_j}\right)(s,y)G_{p,j}(\sigma_p,.;s,y)dyds{\Big |}_{H^{m-1}\cap C^{m-1}}\\
\\
+ \rho c^{\mu,t_0}{\Big |}\int_{(p-1)\Delta}^{p\Delta}\int_{{\mathbb R}^n} \sum_{j,r=1}^n\int_{{\mathbb R}^n}\left( \frac{\partial}{\partial x_i}K_n(z-y)\right)\times\\
\\
\times \sum_{j,r=1}^n\left( \frac{\partial u^{\rho,t_0,p,q-1}_r}{\partial x_j}\frac{\partial u^{\rho,t_0,p,q-1}_j}{\partial x_r}\right) (s,y)G_{p,j}(\sigma_p,.;s,z)dydzds{\Big |}_{H^{m-1}\cap C^{m-1}}.
\end{array}
\end{equation}
For the first term on the right side we have that for each $\epsilon >0$ that there is a generic $\rho>0$ such that
\begin{equation}
{\big |}\int_{{\mathbb R}^n}u^{\rho,t_0,p-1}_i((p-1)\Delta,y)G_{p}(\sigma_p,.;(p-1)\Delta,y)dy{\big |}_{H^m\cap C^m}\leq C+\epsilon
\end{equation}
This means that in order to prove the upper bound $C>0$ for ${\big |}u^{\rho,t_0,p,q}_i(p\Delta,.){\big |}_{H^m\cap C^m}$ it is sufficient to prove prove that for $\rho >0$ small enough the sum of the last two terms on the right side of (\ref{aa}) is strictly smaller than the modulus of the potential term, i.e., smaller than $ c_{\mu,t_0} \Delta^{\alpha}{\big |} u^{\rho,t_0,p,q-1}_{i}(\sigma_{p-1},.){\big |}_{H^m\cap C^m}$, where we recall that $\Delta^{\alpha}=\left(\sigma_p-\sigma_{p-1}\right)^{\alpha}$. In order to analyze these terms we consider the spatial convolutions first. For each fix $s\in [\sigma_{p-1},\sigma_p]$ we may splitting the spatial integrals in a local part and its complement in ${\mathbb R}^n$ and using the fact that the first order spatial derivative $G_{p,j}$ of the fundamental solution locally satisfies
\begin{equation}\label{gpj}
{\Big |}G_{p,j}(\sigma_p,x;s,z){\big |}\leq \frac{c}{|\sigma_p-s|^{\delta}{\big |}x-z{\big |}^{n+1-2\delta}}
\end{equation}
for some $\delta \in (0.5,1)$ and some $c>0$. Hence, $G_{p,j}$ is locally $L^1$ with respect to time by (\ref{gpj}) (where it is sufficient to consider the Gaussian and its derivatives on an open interval with respect to time where it is regular) and certainly globally $L^1$ with respect to space, where we may denote the $l^1$ upper bound with respect to all variables by $C_G$. We have (recall again that $\Delta =(\sigma_p-\sigma_{p-1})$)
\begin{equation}
\begin{array}{ll}
{\Big |}\int_{(p-1)\Delta}^{p\Delta}\int_{{\mathbb R}^n}\sum_{j=1}^n \left( u^{\rho,t_0,p,q-1}_j\frac{\partial u^{\rho,t_0,p,q-1}_i}{\partial x_j}\right)(s,y)G_{p,j}(\sigma_p,.;s,y)dyds{\Big |}_{H^{m-1}\cap C^{m-1}}\\
\\
\leq nC_G\Delta^{1-\delta}\max_{1\leq j\leq n}\sup_{s\in [\sigma_{p-1},\sigma_p]}{\Big |}\left( u^{\rho,t_0,p,q-1}_j u^{\rho,t_0,p,q-1}_i\right)(s,.)ds{\Big |}_{H^{m}\cap C^{m}}\\
\\
\leq nC_G\Delta^{1-\delta}
C_m \left( \max_{1\leq j\leq n}\sup_{s\in [\sigma_{p-1},\sigma_p]}\left( {\Big |} u^{\rho,t_0,p,q-1}_j(s,.){\Big |}_{H^m\cap C^m}\right)\right)  \times\\
\\ 
\times \left( \sup_{s\in [\sigma_{p-1},\sigma_p]}
{\Big |}u^{\rho,t_0,p,q-1}_i(s,.){\Big |}_{H^m\cap C^m}\right) 
\end{array}
\end{equation}
where for $m\geq 2$ the constant $C_m$ is the constant determind in standard estimates for products of functions in strong Sobolev spaces by this constant times the product of the same Sobolev norms of its factors. In order to get a choice for $\rho >0$ we need another estimate for the Leray projection term. This is similar but we have an additional Laplacian kernel. Hence, we first get
\begin{equation}
\begin{array}{ll}
{\Big |}\int_{(p-1)\Delta}^{p\Delta}\int_{{\mathbb R}^n} \sum_{j,r=1}^n\int_{{\mathbb R}^n}\left( \frac{\partial}{\partial x_i}K_n(z-y)\right)\times\\
\\
\times \sum_{j,r=1}^n\left( \frac{\partial u^{\rho,t_0,p,q-1}_r}{\partial x_j}\frac{\partial u^{\rho,t_0,p,q-1}_j}{\partial x_r}\right) (s,y)G_{p,j}(\sigma_p,.;s,z)dydzds{\Big |}_{H^{m-1}\cap C^{m-1}}\\
\\
\leq n^2C_G\Delta^{1-\delta}\max_{1\leq j,r\leq n}\sup_{s\in \left[\sigma_{p-1},\sigma_p \right] }
{\Big |}\int_{{\mathbb R}^n}\left( \frac{\partial}{\partial x_i}K_n(.-y)\right)\times\\
\\
\times \left( u^{\rho,t_0,p,q-1}_ru^{\rho,t_0,p,q-1}_j\right) 
(s,y)dy{\Big |}_{H^{m}\cap C^{m}}\\
\\
\leq n^2C_G\Delta^{1-\delta}C_KC_m\max_{1\leq r\leq n}\sup_{s\in \left[\sigma_{p-1},\sigma_p \right] }
{\Big |} u^{\rho,t_0,p,q-1}_r(s,.){\Big |}_{H^m\cap C^m}\times\\
\\
\times \max_{1\leq j\leq n}\sup_{s\in \left[\sigma_{p-1},\sigma_p \right] }{\Big |}u^{\rho,t_0,p,q-1}_j
(s,.){\Big |}_{H^{m}\cap C^{m}}.
\end{array}
\end{equation}
Here $C_K$ can be chosen to be be the sum of the local $L^1$-norm of first order derivative kernel $K_{n,i}$ (restricted to a bounded ball) and the $L^2$-norm outside this ball of $K_{n,i}\sim\frac{x_i}{|x|^n}$.
Hence, from (\ref{aa}) and the preceding considerations we get for each $\epsilon >0$ a $\rho>0$ such that
\begin{equation}\label{ab}
\begin{array}{ll}
{\big |}u^{\rho,t_0,p,q}_i(p\Delta,.){\big |}_{H^m\cap C^m}\\
\\
\leq C+\epsilon
-c_{\mu,t_0} (\sigma_p-\sigma_{p-1})^{\alpha}{\big |} u^{\rho,t_0,p,q-1}_{i}(\sigma_{p-1},.){\big |}_{H^m\cap C^m}\\
\\
+\rho c^{\mu,t_0}nC_G\Delta^{1-\delta}
C_m \left( \max_{1\leq j\leq n}\sup_{s\in [\sigma_{p-1},\sigma_p]}\left( {\Big |} u^{\rho,t_0,p,q-1}_j(s,.){\Big |}_{H^m\cap C^m}\right)\right)  \times\\
\\ 
\times \left( \sup_{s\in [\sigma_{p-1},\sigma_p]}
{\Big |}u^{\rho,t_0,p,q-1}_i(s,.){\Big |}_{H^m\cap C^m}\right) \\
\\
+ \rho c^{\mu,t_0}n^2C_G\Delta^{1-\delta}C_KC_m\max_{1\leq r\leq n}\sup_{s\in \left[\sigma_{p-1},\sigma_p \right] }
{\Big |} u^{\rho,t_0,p,q-1}_r(s,.){\Big |}_{H^m\cap C^m}\times\\
\\
\times \max_{1\leq j\leq n}\sup_{s\in \left[\sigma_{p-1},\sigma_p \right] }{\Big |}u^{\rho,t_0,p,q-1}_j
(s,.){\Big |}_{H^{m}\cap C^{m}}.
\end{array}
\end{equation}
Using the induction hypothesis (with respect to the iteration index $p$) we have
\begin{equation}\label{ac}
\begin{array}{ll}
{\big |}u^{\rho,t_0,p,q}_i(p\Delta,.){\big |}_{H^m\cap C^m}\\
\\
\leq C+\epsilon
-c_{\mu,t_0} (\sigma_p-\sigma_{p-1})^{\alpha}{\big |} u^{\rho,t_0,p,q-1}_{i}(\sigma_{p-1},.){\big |}_{H^m\cap C^m}\\
\\
+\rho c^{\mu,t_0}nC_G\Delta^{1-\delta}
C_m (1+nC_K)C^2.
\end{array}
\end{equation}
Now we may assume without loss of generality that
\begin{equation}\label{assrhot0qp}
{\big |} u^{\rho,t_0,p,q-1}_{i}(\sigma_{p-1},.){\big |}_{H^m\cap C^m}\geq \frac{C}{2},
\end{equation}
as we could replace the first $C>0$ on the right side of (\ref{ac}) by $\frac{C}{2}$ and come to a similar conclusion. Under the assumption (\ref{assrhot0qp}) we get 
\begin{equation}\label{ac}
\begin{array}{ll}
{\big |}u^{\rho,t_0,p,q}_i(p\Delta,.){\big |}_{H^m\cap C^m}\leq C~
\mbox{if}~\rho \leq \frac{c_{\mu,t_0} (\sigma_p-\sigma_{p-1})^{\alpha}\frac{C}{2}-\epsilon}{c^{\mu,t_0}nC_G\Delta^{1-\delta}C_m (1+nC_K)C^2}.
\end{array}
\end{equation} 
recall again that $\sigma_p-\sigma_{p-1}=\Delta$ such that
\begin{equation}\label{ac}
\begin{array}{ll}
{\big |}u^{\rho,t_0,p,q}_i(p\Delta,.){\big |}_{H^m\cap C^m}\leq C~
\mbox{if}~\rho \leq \frac{c_{\mu,t_0} \Delta^{\alpha+\delta-1}\frac{C}{2}-\epsilon}{c^{\mu,t_0}nC_GC_m (1+nC_K)C^2}.
\end{array}
\end{equation} 
As $\epsilon >0$ is arbitrary this concludes the proof in the discrete case, while for $\alpha\geq \delta >\frac{1}{2}$ the same conclusion can be drawn for all $\sigma\in [0,c_{\sigma_e}]$ (corresponding to $\tau\in \left[t_0,t_0+c_e \right]$), where we are free to choose an appropriate scaling for $\epsilon >0$. As the factor $\sqrt{1-\tau_{t_0}^2}^3$ appears in all relevant parts of the equation the prove can be extended to the whole domain $\sigma\in [0,\infty)$ (corresponding to $\tau\in [t_0,t_0+1)$ and $t\in [t_0,t_0+\rho)$.
As the preservation constant $C>0$ is indpendent of the iteration number $p\geq 1$ we have proved the lemma for all $p\geq 1$ by induction. As the series $u^{\rho,t_0,p,q}_i,~1\leq i\leq n,~p\geq 1$ exists in strong Sobolev-type spaces with a uniform upper bound it is then easy to conclude (either by compactness or by a contraction argument) that
\begin{equation}
u^{\rho,t_0,p}_i:=\lim_{q\uparrow \infty}u^{\rho,t_0,p,q}_i,~1\leq i\leq n
\end{equation}
satisfies the claim of the Lemma and is a local solution of damped equation.

\end{proof}

\end{document}